\documentclass[12pt]{amsart} 
\usepackage[left=3cm,top=2.5cm,bottom=2.5cm,right=3cm]{geometry}
\usepackage{times}
\usepackage{amssymb, amsmath, amsthm}
\usepackage{graphicx,xspace}
\usepackage{epsfig}
\usepackage{enumitem}

\usepackage[usenames,dvipsnames]{xcolor}
\usepackage{tikz}
\usepackage{gensymb}
\usepackage{enumitem}
\usepackage[T1]{fontenc}
\usepackage[utf8]{inputenc}
\usepackage{bm,scalerel}
\usepackage[config, labelfont={normalsize}]{caption,subfig}
\usepackage{mathrsfs}
\definecolor{green}{RGB}{0,127,0}
\definecolor{red}{RGB}{191,0,0}
\usepackage[colorlinks,cite color=red,link color=green,pagebackref=true]{hyperref}
\usepackage{todonotes}
\usepackage{scalerel,stackengine}
\usetikzlibrary{matrix,arrows,calc}
\usetikzlibrary{decorations.markings}
\usetikzlibrary{patterns}
\usetikzlibrary{snakes}
\usetikzlibrary{shapes}

\usepackage[capitalize]{cleveref}

\theoremstyle{plain}

\newtheorem{lemma}{Lemma}[section]
\newtheorem{theorem}[lemma]{Theorem}

\newtheorem{proposition}[lemma]{Proposition}

\newtheorem{question}[lemma]{Question}

\newtheorem{definition}[lemma]{Definition}
\newtheorem{definition-lemma}[lemma]{Definition-Lemma}

\newtheorem{remark}[lemma]{Remark}
\newtheorem{example}{Example}

\DeclareMathOperator{\Part}{\mathcal{P}}

\def\P{\mathcal{P}}

\newcommand{\G}{{G}}

\newcommand{\RR}{\mathbb{R}}
\newcommand{\C}{\mathbb{C}}
\newcommand{\N}{\mathbb{N}}

\newcommand{\F}{\mathcal{F}(\HH)}
\newcommand{\p}{\beta}
\newcommand{\n}{{\gamma}}
\newcommand{\Ci}{\sigma}
\newcommand{\x}{\mathbf{x}_{\overline n} \otimes \mathbf{x}_n}
\newcommand{\y}{\mathbf{y}_{\overline n} \otimes \mathbf{y}_n}
\newcommand{\q}{q_+}
\newcommand{\s}{q_-}
\newcommand{\Sym}[1]{\mathfrak{S}_{#1}}

\newcommand{\RRR}{\mathcal{R}}
\newcommand{\tyldapi}{\tilde{\pi}}

\newcommand{\CCC}{\mathcal{C}}

\newcommand{\Z}{\mathbb{Z}}

\newcommand{\J}{\mathcal{J}}
\def\SLNB{\mathit{c_{-}}}

\def\Pair{\text{\normalfont Pair}}
\def\Cycle{\text{\normalfont Cyc}}
\def\Scycle{\text{\normalfont SemiCyc}}
\def\Lcycle{\mathit{ l_{c}}}
\def\LScycle{{\mathit l_{sc}}}

\def\ends{\mathsf{ l}}
\def\start{\mathsf{ l}}
\def\maks{\mathsf{ r }}

\def\m{\kappa_{\q,\s}}
\def\qMP{{\mu}}

\newcommand{\xx}{\bm{x}}
\newcommand{\yy}{\bm{y}}

\newcommand{\cyc}{{c}}

\def\R{{\mathbb R}}
\def\H{{\bf \mathcal K}_{ n}}
\def\HH{{\bf \mathcal K}}
\def\B{\r_{\q,\s}}
\def\state{\varphi}

\newcommand\reallywidehat[1]{%
\savestack{\tmpbox}{\stretchto{%
  \scaleto{%
    \scalerel*[\widthof{\ensuremath{#1}}]{\kern-.6pt\bigwedge\kern-.6pt}%
    {\rule[-\textheight/2]{1ex}{\textheight}}%WIDTH-LIMITED BIG WEDGE
  }{\textheight}% 
}{0.5ex}}%
\stackon[1pt]{#1}{\tmpbox}%
}

\DeclareMathOperator{\Dyck}{Dyck}

\def\a{\alpha}

\def\r{\mathsf{b}}

\DeclareMathOperator{\Tr}{Tr}

\DeclareMathOperator{\sgn}{sgn}

\DeclareMathOperator{\id}{id}

\DeclareMathOperator{\NC}{\mathcal{NC}}

\DeclareMathOperator{\Span}{Span}

\newcommand\MatchingMeanders[3]{%
  \begin{tikzpicture}[scale=0.7]
        \foreach \x in {1,...,#1}{
       \draw[circle,fill] (\x,0)circle[radius=1mm]node[below]{$ \x$};
    }
    \foreach \x/\y in {#2} {
       \pgfmathsetmacro{\Radius}{\y/2-\x/2}
       \draw(\x,0) arc[radius=\Radius, start angle=180, end angle=0];
;
      \node at (0,2.3) { $ \pi$}; 
        \node at (0,-2) { $ \hat \pi$}; 
    }
    \foreach \x/\y in {#3} {
       \pgfmathsetmacro{\Radius}{\y/2-\x/2}
       \draw(\x,0) arc[radius=\Radius, start angle=-180, end angle=0];
        ;}
         \foreach \x in  {#1,...,1}{
       \draw[circle,fill] (-\x,0)circle[radius=1mm]node[below]{$ \overline{\x}$};
    }
  \end{tikzpicture}
}

\newcommand\MatchingMeandersab[3]{%
  \begin{tikzpicture}[scale=0.7]
        \foreach \x in {1,...,#1}{
       \draw[circle,fill] (\x,0)circle[radius=1mm]node[below]{};
    }
    \foreach \x/\y in {#2} {
       \pgfmathsetmacro{\Radius}{\y/2-\x/2}
       \draw(\x,0) arc[radius=\Radius, start angle=180, end angle=0];
;
    }
    \foreach \x/\y in {#3} {
       \pgfmathsetmacro{\Radius}{\y/2-\x/2}
       \draw(\x,0) arc[radius=\Radius, start angle=-180, end angle=0];
        ;}
         \foreach \x in {-#1,...,-1}{
       \draw[circle,fill] (\x,0)circle[radius=1mm]node[below]{};
    }
    \node at (-4,0.4) { $\overline 4$}; 
     \node at (-3,0.4) { $ \overline 3$}; 
     \node at (-2,0.4) { $\overline  2$};
     \node at (-1,0.4) { $\overline  1$};
       \node at (4,0.4) { $ 4$}; 
     \node at (3,0.4) { $  3$}; 
     \node at (2,0.4) { $  2$};
     \node at (1,0.4) { $  1$};
          \node at (2.5,-0.75) { $\downarrow$};
\node at (2.5,-1.4) { $ C$}; 
   \node at (-2.5,-0.75) { $\downarrow$};
\node at (-2.5,-1.4) { $ \overline C$};
\node at (0,3.2) { $ (2,4)\sim (1,3)$ by $C=(2,3)$};
  \end{tikzpicture}
}

\newcommand\MatchingMeandersabc[3]{%
  \begin{tikzpicture}[scale=0.7]
        \foreach \x in {1,...,#1}{
       \draw[circle,fill] (\x,0)circle[radius=1mm]node[below]{};
    }
    \foreach \x/\y in {#2} {
       \pgfmathsetmacro{\Radius}{\y/2-\x/2}
       \draw(\x,0) arc[radius=\Radius, start angle=180, end angle=0];
;
    }
    \foreach \x/\y in {#3} {
       \pgfmathsetmacro{\Radius}{\y/2-\x/2}
       \draw(\x,0) arc[radius=\Radius, start angle=-180, end angle=0];
        ;}
         \foreach \x in {-#1,...,-1}{
       \draw[circle,fill] (\x,0)circle[radius=1mm]node[below]{};
    }
    \node at (-4,0.4) { $\overline 4$}; 
     \node at (-3,0.4) { $ \overline 3$}; 
     \node at (-2,0.4) { $\overline 2$};
     \node at (-1,0.4) { $  \overline 1$};
       \node at (4,0.4) { $ 4$}; 
     \node at (3,0.4) { $  3$}; 
     \node at (2,0.4) { $ 2$};
     \node at (1,0.4) { $ 1$};
     \node at (2.5,-0.75) { $\downarrow$};
\node at (2.5,-1.4) { $ C$}; 
   \node at (-2.5,-0.75) { $\downarrow$};
\node at (-2.5,-1.4) { $ \overline C$};
\node at (0,3.2) { $ (\overline 3,2)\sim (\overline 2, 3)$ by $C=(2,3)$};
  \end{tikzpicture}
}

\newcommand\MatchingProof[3]{%
  \begin{tikzpicture}[scale=0.8]
        \foreach \x in {1,...,#1}{
       \draw[circle,fill] (\x,0)circle[radius=1mm]node[below]{};
    }
    \foreach \x/\y in {#2} {
       \pgfmathsetmacro{\Radius}{\y/2-\x/2}
       \draw(\x,0) arc[radius=\Radius, start angle=180, end angle=0];
;
    }
    \foreach \x/\y in {#3} {
       \pgfmathsetmacro{\Radius}{\y/2-\x/2}
       \draw(\x,0) arc[radius=\Radius, start angle=-180, end angle=0];
        ;}
         \foreach \x in {-#1,...,-1}{
       \draw[circle,fill] (\x,0)circle[radius=1mm]node[below]{};
    }
     \node at (-2,0.4) { $ \scriptstyle  \start(\Ci^-_p)$};
       \node at (2,0.4) { $ \scriptstyle  \ends(\Ci^+_p)$};
       \node at (-1,0.4) { $ \scriptstyle  \maks(\Ci^-_p)$};
       \node at (1,0.4) { $ \scriptstyle  \maks(\Ci^+_p)$};
       \node at (2,-1.72) { $\downarrow$};
    \node at (2,-2.3) { $  \scriptstyle  \Ci^+_p$};
\node at (-2,-1.72) { $\downarrow$};
\node at (-2,-2.3) { $ \scriptstyle  \Ci^-_p$};  
\node at (0,0) { $\scriptstyle \dots$}; 
\node at (0,1.3) { $\pi$}; 
   \node at (3.8,0) { $\longmapsto$};
  \end{tikzpicture}
}

\newcommand\MatchingProofd[3]{%
  \begin{tikzpicture}[scale=0.8]
        \foreach \x in {1,...,#1}{
       \draw[circle,fill] (\x,0)circle[radius=1mm]node[below]{};
    }
    \foreach \x/\y in {#2} {
       \pgfmathsetmacro{\Radius}{\y/2-\x/2}
       \draw(\x,0) arc[radius=\Radius, start angle=180, end angle=0];
;
    }
    \foreach \x/\y in {#3} {
       \pgfmathsetmacro{\Radius}{\y/2-\x/2}
       \draw(\x,0) arc[radius=\Radius, start angle=-180, end angle=0];
        ;}
         \foreach \x in {-#1,...,-1}{
       \draw[circle,fill] (\x,0)circle[radius=1mm]node[below]{};
    }
     \node at (-2,0.4) { $ \scriptstyle  \start(\Ci^-_p)$};
       \node at (2,0.4) { $ \scriptstyle  \ends(\Ci^+_p)$};
       \node at (-1,0.4) { $ \scriptstyle   \maks(\Ci^+_p)$};
       \node at (1,0.4) { $ \scriptstyle  \maks(\Ci^-_p)$};
       \node at (2,-1.72) { $\downarrow$};
    \node at (2,-2.3) { $ \scriptstyle  \Ci^+_p$};
\node at (-2,-1.72) { $\downarrow$};
\node at (-2,-2.3) { $\scriptstyle \Ci^-_p$};  
\node at (0,0) { $\scriptstyle \dots$}; 
   \node at (3.8,0) { $\longmapsto$};
\node at (0,1.3) { $\pi$}; 
  \end{tikzpicture}
}

  \newcommand\MatchingProoff[3]{%
     \begin{tikzpicture}[scale=0.8]
        \foreach \x in {1,...,#1}{
       \draw[circle,fill] (\x,0)circle[radius=1mm]node[below]{};
    }
    \foreach \x/\y in {#2} {
       \pgfmathsetmacro{\Radius}{\y/2-\x/2}
       \draw(\x,0) arc[radius=\Radius, start angle=180, end angle=0];
;
    }
    \foreach \x/\y in {#3} {
       \pgfmathsetmacro{\Radius}{\y/2-\x/2}
       \draw(\x,0) arc[radius=\Radius, start angle=-180, end angle=0];
        ;}
         \foreach \x in {-#1,...,-1}{
       \draw[circle,fill] (\x,0)circle[radius=1mm]node[below]{};
    }
      \node at (-2,0.4) { $ \scriptstyle  \start(\Ci^-_p)$};
       \node at (2,0.4) { $ \scriptstyle  \ends(\Ci^+_p)$};
           \node at (-1,0.4) { $ \scriptstyle  \maks(\Ci^-_p)$};
       \node at (1,0.4) { $ \scriptstyle  \maks(\Ci^+_p)$};
          \node at (-4.05,0.4) { $ \scriptstyle  \overline k$};
       \node at (4.05,0.4) { $ \scriptstyle  k$};
         \node at (2.5,-1.72) { $\downarrow$};
    \node at (2.5,-2.3) { $\scriptstyle  (\Ci^+_p,k)$};
     \node at (-2.5,-1.72) { $\downarrow$};
\node at (-2.5,-2.3) { $\scriptstyle  (\overline k ,\Ci^-_p)$};  
\node at (0,0) { $\scriptstyle \dots$}; 
\node at (0,1.3) { $\tyldapi$}; 
     \node at (-4.5,1.5) { $(a)$};
  \end{tikzpicture}
}
\newcommand\MatchingProoffd[3]{%
     \begin{tikzpicture}[scale=0.8]
        \foreach \x in {1,...,#1}{
       \draw[circle,fill] (\x,0)circle[radius=1mm]node[below]{};
    }
    \foreach \x/\y in {#2} {
       \pgfmathsetmacro{\Radius}{\y/2-\x/2}
       \draw(\x,0) arc[radius=\Radius, start angle=180, end angle=0];
;
    }
    \foreach \x/\y in {#3} {
       \pgfmathsetmacro{\Radius}{\y/2-\x/2}
       \draw(\x,0) arc[radius=\Radius, start angle=-180, end angle=0];
        ;}
         \foreach \x in {-#1,...,-1}{
       \draw[circle,fill] (\x,0)circle[radius=1mm]node[below]{};
    }
      \node at (-2,0.4) { $ \scriptstyle  \start(\Ci^-_p)$};
       \node at (2,0.4) { $ \scriptstyle  \ends(\Ci^+_p)$};
           \node at (-1,0.4) { $ \scriptstyle  \maks(\Ci^+_p)$};
       \node at (1,0.4) { $ \scriptstyle  \maks(\Ci^-_p)$};
          \node at (-4.05,0.4) { $ \scriptstyle  \overline k$};
       \node at (4.05,0.4) { $ \scriptstyle  k$};
         \node at (2.5,-1.72) { $\downarrow$};
    \node at (2.5,-2.3) { $\scriptstyle (\Ci^-_p,k)$};
     \node at (-2.5,-1.72) { $\downarrow$};
\node at (-2.5,-2.3) { $ \scriptstyle (\overline k,\Ci^+_p)$};  
\node at (0,0) { $\scriptstyle \dots$}; 
   \node at (0,1.3) { $\tyldapi$}; 
      \node at (-4.5,3) { $(b)$}; 
  \end{tikzpicture}
}

\newcommand\MatchingProoffdd[3]{%
     \begin{tikzpicture}[scale=0.8]
        \foreach \x in {1,...,#1}{
       \draw[circle,fill] (\x,0)circle[radius=1mm]node[below]{};
    }
    \foreach \x/\y in {#2} {
       \pgfmathsetmacro{\Radius}{\y/2-\x/2}
       \draw(\x,0) arc[radius=\Radius, start angle=180, end angle=0];
;
    }
    \foreach \x/\y in {#3} {
       \pgfmathsetmacro{\Radius}{\y/2-\x/2}
       \draw(\x,0) arc[radius=\Radius, start angle=-180, end angle=0];
        ;}
         \foreach \x in {-#1,...,-1}{
       \draw[circle,fill] (\x,0)circle[radius=1mm]node[below]{};
    }
      \node at (-2,0.4) { $ \scriptstyle  \start(\Ci^-_p)$};
       \node at (2,0.4) { $ \scriptstyle  \ends(\Ci^+_p)$};
           \node at (-1,0.4) { $ \scriptstyle  \maks(\Ci^+_p)$};
       \node at (1,0.4) { $ \scriptstyle  \maks(\Ci^-_p)$};
          \node at (-4.05,0.4) { $ \scriptstyle  \overline k$};
       \node at (4.05,0.4) { $ \scriptstyle  k$};
     \node at (0,-1.72) { $\downarrow$};
\node at (0,-2.3) { $\scriptstyle (\overline k,\Ci^-_p)^-$};  
\node at (0,0) { $\scriptstyle \dots$}; 
   \node at (0,1.3) { $\tyldapi$}; 
      \node at (-4.5,3) { $(b)$}; 
  \end{tikzpicture}
}

\newcommand\MatchingProofp[3]{%
  \begin{tikzpicture}[scale=0.4]
        \foreach \x in {1,...,#1}{
       \draw[circle,fill] (\x,0)circle[radius=1mm]node[below]{};
    }
    \foreach \x/\y in {#2} {
       \pgfmathsetmacro{\Radius}{\y/2-\x/2}
       \draw(\x,0) arc[radius=\Radius, start angle=180, end angle=0];
;
    }
    \foreach \x/\y in {#3} {
       \pgfmathsetmacro{\Radius}{\y/2-\x/2}
       \draw(\x,0) arc[radius=\Radius, start angle=-180, end angle=0];
        ;}
         \foreach \x in {-#1,...,-1}{
       \draw[circle,fill] (\x,0)circle[radius=1mm]node[below]{};
    }
       \node at (2,-2) { $\downarrow$};
    \node at (2,-3) { $ \scriptstyle  \Ci^+_i$};
    \node at (5,-2) { $\downarrow$};
       \node at (5,-3) { $\scriptstyle   \Ci^+_p$};
\node at (-2,-2) { $\downarrow$};
\node at (-2,-3) { $\scriptstyle \Ci^-_i$};  
\node at (-5,-2) { $\downarrow$};
\node at (-5,-3) { $\scriptstyle \Ci^-_p$};  
\node at (0,0) { $\scriptscriptstyle \dots$}; 
   \node at (7,1) { $\longmapsto$}; 
  \end{tikzpicture}
}

  \newcommand\MatchingProoffpp[3]{%
     \begin{tikzpicture}[scale=0.4]
        \foreach \x in {1,...,#1}{
       \draw[circle,fill] (\x,0)circle[radius=1mm]node[below]{};
    }
    \foreach \x/\y in {#2} {
       \pgfmathsetmacro{\Radius}{\y/2-\x/2}
       \draw(\x,0) arc[radius=\Radius, start angle=180, end angle=0];
;
    }
    \foreach \x/\y in {#3} {
       \pgfmathsetmacro{\Radius}{\y/2-\x/2}
       \draw(\x,0) arc[radius=\Radius, start angle=-180, end angle=0];
        ;}
         \foreach \x in {-#1,...,-1}{
       \draw[circle,fill] (\x,0)circle[radius=1mm]node[below]{};
    }
         \node at (5.5,-2) { $\downarrow$};
    \node at (4,-3) { $\scriptstyle (\Ci^+_i,k,\Ci^+_p)^+$};
     \node at (-5.5,-2) { $\downarrow$};
\node at (-4,-3) { $\scriptstyle (\Ci^-_p,\overline k,\Ci^-_i)^-$};  
\node at (0,0) { $\scriptstyle \dots$}; 
     \node at (-7.5,3) { $(a)$};
  \end{tikzpicture}
}

  \newcommand\MatchingProoffppb[3]{%
     \begin{tikzpicture}[scale=0.4]
        \foreach \x in {1,...,#1}{
       \draw[circle,fill] (\x,0)circle[radius=1mm]node[below]{};
    }
    \foreach \x/\y in {#2} {
       \pgfmathsetmacro{\Radius}{\y/2-\x/2}
       \draw(\x,0) arc[radius=\Radius, start angle=180, end angle=0];
;
    }
    \foreach \x/\y in {#3} {
       \pgfmathsetmacro{\Radius}{\y/2-\x/2}
       \draw(\x,0) arc[radius=\Radius, start angle=-180, end angle=0];
        ;}
         \foreach \x in {-#1,...,-1}{
       \draw[circle,fill] (\x,0)circle[radius=1mm]node[below]{};
    }
         \node at (5.5,-2) { $\downarrow$};
    \node at (4,-3) { $\scriptstyle (\Ci^-_i,k,\Ci^+_p)^+$};
     \node at (-5.5,-2) { $\downarrow$};
\node at (-4,-3) { $\scriptstyle (\Ci^-_p,\overline k,\Ci^+_i)^-$};  
\node at (0,0) { $\scriptstyle \dots$}; 
     \node at (-7.5,3) { $(b)$};
  \end{tikzpicture}
}

  \newcommand\MatchingProoffppc[3]{%
     \begin{tikzpicture}[scale=0.4]
        \foreach \x in {1,...,#1}{
       \draw[circle,fill] (\x,0)circle[radius=1mm]node[below]{};
    }
    \foreach \x/\y in {#2} {
       \pgfmathsetmacro{\Radius}{\y/2-\x/2}
       \draw(\x,0) arc[radius=\Radius, start angle=180, end angle=0];
;
    }
    \foreach \x/\y in {#3} {
       \pgfmathsetmacro{\Radius}{\y/2-\x/2}
       \draw(\x,0) arc[radius=\Radius, start angle=-180, end angle=0];
        ;}
         \foreach \x in {-#1,...,-1}{
       \draw[circle,fill] (\x,0)circle[radius=1mm]node[below]{};
    }
         \node at (5.5,-2) { $\downarrow$};
    \node at (4,-3) { $\scriptstyle (\Ci^-_p,k,\Ci^-_i)^-$};
     \node at (-5.5,-2) { $\downarrow$};
\node at (-4,-3) { $\scriptstyle (\Ci^+_i,\overline k,\Ci^+_p)^+$};  
\node at (0,0) { $\scriptstyle \dots$};  
     \node at (-7.5,3) { $(a)$};
  \end{tikzpicture}
}

  \newcommand\MatchingProoffppd[3]{%
     \begin{tikzpicture}[scale=0.4]
        \foreach \x in {1,...,#1}{
       \draw[circle,fill] (\x,0)circle[radius=1mm]node[below]{};
    }
    \foreach \x/\y in {#2} {
       \pgfmathsetmacro{\Radius}{\y/2-\x/2}
       \draw(\x,0) arc[radius=\Radius, start angle=180, end angle=0];
;
    }
    \foreach \x/\y in {#3} {
       \pgfmathsetmacro{\Radius}{\y/2-\x/2}
       \draw(\x,0) arc[radius=\Radius, start angle=-180, end angle=0];
        ;}
         \foreach \x in {-#1,...,-1}{
       \draw[circle,fill] (\x,0)circle[radius=1mm]node[below]{};
    }
         \node at (5.5,-2) { $\downarrow$};
    \node at (4,-3) { $\scriptstyle (\Ci^-_p,k,\Ci^+_i)^-$};
     \node at (-5.5,-2) { $\downarrow$};
\node at (-4,-3) { $\scriptstyle (\Ci^-_i,\overline k,\Ci^+_p)^+$};  
\node at (0,0) { $\scriptstyle \dots$}; 
     \node at (-7.5,3) { $(b)$};
  \end{tikzpicture}
}

 \newcommand\MatchingProofff[3]{%
     \begin{tikzpicture}[scale=0.8]
        \foreach \x in {1,...,#1}{
       \draw[circle,fill] (\x,0)circle[radius=1mm]node[below]{};
    }
    \foreach \x/\y in {#2} {
       \pgfmathsetmacro{\Radius}{\y/2-\x/2}
       \draw(\x,0) arc[radius=\Radius, start angle=180, end angle=0];
;
    }
    \foreach \x/\y in {#3} {
       \pgfmathsetmacro{\Radius}{\y/2-\x/2}
       \draw(\x,0) arc[radius=\Radius, start angle=-180, end angle=0];
        ;}
         \foreach \x in {-#1,...,-1}{
       \draw[circle,fill] (\x,0)circle[radius=1mm]node[below]{};
    }
      \node at (-2,0.4) { $ \scriptstyle  \start(\Ci^-_p)$};
       \node at (2,0.4) { $ \scriptstyle  \ends(\Ci^+_p)$};
           \node at (-1,0.4) { $ \scriptstyle  \maks(\Ci^-_p)$};
       \node at (1,0.4) { $ \scriptstyle  \maks(\Ci^+_p)$};
          \node at (-4.05,0.4) { $ \scriptstyle  \overline k$};
       \node at (4.05,0.4) { $ \scriptstyle  k$};
     \node at (0,-1.72) { $\downarrow$};
\node at (0,-2.3) { $\scriptstyle (\overline k,\Ci^+_p)^-$};  
\node at (0,0) { $\scriptstyle \dots$}; 
\node at (0,1.3) { $\tyldapi$}; 
       \node at (-4.5,3) { $(a)$}; 
  \end{tikzpicture}
}

\newcommand\MatchingProofequation[3]{%
  \begin{tikzpicture}[scale=0.2]
        \foreach \x in {1,...,#1}{
       \draw[circle,fill] (\x,0)circle[radius=1mm]node[below]{};
    }
    \foreach \x/\y in {#2} {
       \pgfmathsetmacro{\Radius}{\y/2-\x/2}
       \draw(\x,0) arc[radius=\Radius, start angle=180, end angle=0];
;
    }
    \foreach \x/\y in {#3} {
       \pgfmathsetmacro{\Radius}{\y/2-\x/2}
       \draw(\x,0) arc[radius=\Radius, start angle=-180, end angle=0];
        ;}
         \foreach \x in {-#1,...,-1}{
       \draw[circle,fill] (\x,0)circle[radius=1mm]node[below]{};
    }
\node at (0,0) { $\scriptscriptstyle \dots$}; 
  \end{tikzpicture}
}

\title[Positive
definite reflection length in type B/C]{Reflection length with two parameters in the asymptotic representation theory of type B/C and applications}

\author[M.~Bożejko]{Marek Bożejko}
\address{
Instytut Matematyczny,
Uniwersytet Wrocławski,  \mbox{pl.\ Grunwaldzki~2/4,} 50-384
Wrocław, Poland}
\email{bozejko@gmail.com}

\author[M.~Dołęga]{Maciej Dołęga}
\address{
Institute of Mathematics, 
Polish Academy of Sciences, 
ul. Śniadeckich 8, 
00-956 Warszawa, Poland.
}
\email{mdolega@impan.pl}

\author[W.~Ejsmont]{Wiktor Ejsmont}
\address{
Instytut Matematyczny,
Uniwersytet Wrocławski,  \mbox{pl.\ Grunwaldzki~2/4,} 50-384
Wrocław, Poland}
\email{Wiktor.Ejsmont@math.uni.wroc.pl}

\author[Ś. R.~Gal]{Światosław R.~Gal}
\address{
Instytut Matematyczny,
Uniwersytet Wrocławski,  \mbox{pl.\ Grunwaldzki~2/4,} 50-384
Wrocław, Poland}
\email{sgal@math.uni.wroc.pl}
\thanks{
MB is supported from {\it Narodowe Centrum Nauki}, grant NCN  2016/21/B/ST1/00628.\\
MD is supported from {\it Narodowe Centrum Nauki}, grant UMO-2017/26/D/ST1/00186.\\
ŚG is supported from {\it Narodowe Centrum Nauki}, grant NCN 2017/27/B/ST1/01467}

\begin{document}

\begin{abstract}
  We introduce a two-parameter function
  $\phi_{\q,\s}$ on the infinite hyperoctahedral group, which is a
  bivariate refinement of the reflection length. We show that this signed reflection function $\phi_{\q,\s}$ is
  positive definite if and only if it is an
  extreme character of the infinite hyperoctahedral group and we
  classify the corresponding set of parameters $\q,\s$. We construct
  the corresponding representations through a natural action of the
  hyperoctahedral group $B(n)$ on the tensor product of $n$ copies of
  a vector space, which gives a two-parameter analog of the classical
  construction of Schur--Weyl.

  We apply
  our classification to construct a cyclic Fock space of type B
  generalizing the one-parameter construction in type A found
  previously by
  Bożejko and Guta. We also construct a new Gaussian operator
  acting on the cyclic Fock space of type B and we relate its moments with the Askey--Wimp--Kerov
  distribution by using the notion of cycles on pair-partitions, which we
  introduce here. Finally, we explain how to solve the analogous
  problem for the Coxeter groups of type D by using our main result.
\end{abstract}

\maketitle

%\tableofcontents

\section{Introduction}

Positive definite functions on a group $G$ play a prominent role in harmonic analysis, operator theory, free probability and geometric
group theory. When $G$ has a particularly nice structure of
combinatorial/geometric origin one can use it to construct positive
definite functions which leads to many interesting
properties widely applied in these fields. This phenomenon was
first recognized in the pionnering
work of Haagerup \cite{Haagerup1979}, who proved that the function $g
\to q^{\ell_S(g)}$ is positive definite on the free group $F_N$ with
$N$ generators for $-1 \leq q \leq 1$, where $\ell_S$ is the standard word length.
Haagerup's result was applied to show the completely
bounded approximation property (CBAP) of the regular $C^*$-algebra of
the free group $F_N$ \cite{DeCanniereHaagerup1985} and it had an
impact on free probability, non-commutative harmonic analysis and the
operator algebras \cite{HaagerupPisier1993}. Bożejko, Januszkiewicz
and Spatzier, inspired by the work of Haagerup, were studying
arbitrary Coxeter
groups $(G,S)$ and they proved that the \emph{Coxeter function} $g
\to q^{\ell_S(g)}$ is positive definite for
every Coxeter group if and only if $-1 \leq q \leq 1$ \cite{BozejkoJanuszkiewiczSpatzier1988}. Here $\ell_S$
denotes the \emph{Coxeter length}, which is the standard word length
with respect to the set of Coxeter generators:
\begin{equation}
  \label{eq:CoxeterLength}
  \ell_S(g) := \min(k\colon s_1\cdots s_k = g, s_1,\dots,s_k \in S).
  \end{equation}
They additionally showed that this implies that infinite Coxeter groups do
not have Kazhdan's property $(T)$. This result was further generalized
to multi-parameters
\cite{BozejkoSzwarc2003} and also other variants of the Coxeter function (colour-length) were
studied \cite{BozejkoSpeicher1996,BozejkoGalMlotkowski2018}.

All the considered functions share two distinctive features:
\begin{enumerate}[label=(\Roman*), ref=\Roman*]
  \item they are positive definite on the continuous set $-1\leq q \leq
    1$, \label{A}
    \item they are not (generically) invariant by conjugation. \label{B}
    \end{enumerate}
 Functions on a group $G$ which are invariant by conjugation are
 called \emph{central} (also known as \emph{class functions}). In his seminal work Thoma \cite{Thoma1964}
 defined \emph{characters} as positive definite, central functions on
 a group $G$ normalized to take value $1$ at the identity of the
 group. \emph{Extreme characters} are extreme points in the set of all
 characters and they play a prominent role in asymptotic
 representation theory and specifically in the representation theory of
 infinite dimensional groups developed independently by Thoma
 \cite{Thoma1964} in the case of the infinite symmetric group
 $\Sym{\infty}$ and Voiculescu \cite{Voiculescu1974,Voiculescu1976} in the case of infinite dimensional Lie
 groups $U(\infty),SO(\infty),Sp(\infty)$. When $G$ is compact (e.g. finite)
 the extreme characters correspond to normalized traces of the
 irreducible representations, but when $G$ is infinite dimensional,
 the conventional representation theory of irreducible characters does
 not work well and
 the ideas of Thoma and Voiculescu have laid the foundations for the new and
 quickly developing field of asymptotic representation theory.
 This new area of research naturally linked representation theory with
 harmonic analysis, the theory of symmetric functions, probability theory, random matrix theory and mathematical
 physics (see
 \cite{Kerov2003,BorodinOlshanski2017,Meliot2017} for the development
 of the asymptotic representation theory of $\Sym{n}$ and
 $\Sym{\infty}$ with their wide applications).
The most natural way to modify the Coxeter function $g \to q^{\ell_S(g)}$ in order to
obtain its analog which is central on $G$ is to replace the Coxeter
length $\ell_S$
by the \emph{reflection length} $\ell_{\RRR}$, i.e. the minimal number
of reflections that we need to use to decompose $g$ as their product
\[ \ell_{\RRR}(g) = \min(k\colon r_1\cdots r_k = g, r_1,\dots,r_k \in \RRR
  := \{gsg^{-1}: s \in S, g \in G\}).\]
The reflection length and factorizations into reflections in general are ubiquitous in the enumerative problems of finite
Coxeter groups, or more generally (well-generated) complex reflection
groups \cite{Hurwitz1891,Looijenga1974,Bessis2015}. In the case of the infinite symmetric group $\Sym{\infty}$,
which is the inductive limit of the ascending tower of the symmetric groups $\Sym{1} < \Sym{2}
< \cdots$ one can show that the \emph{reflection function} $g \to
q^{\ell_{\RRR}(g)}$ is a character of $\Sym{\infty}$ for
$q^{-1} \in \Z$. This is a straightforward consequence of the
description of the Thoma simplex given by Vershik and Kerov
\cite{VershikKerov1981} -- in this case its restriction to the finite
symmetric group $\Sym{n}$ gives the normalized trace of the natural
action of the symmetric group on $n$ tensor copies af an $N$-dimensional
vector space appearing in the classical construction of Schur--Weyl. Bożejko and Guta
\cite{BozejkoGuta2002} used this positive definite function to construct a white noise functor and also to obtain certain
``exclusion principle'' of the operator counting the number of
one-particle $f$-states. It is interesting to ask whether these are
the only parameters $q$ for which the reflection function is a
character. The answer for this question is affirmative\footnote{this is not discussed in
  the work of Bożejko and Guta \cite{BozejkoGuta2002}, but it can be shown by the same
  methods as we are using in this paper to prove our main theorem -
  see \cref{remark:SymInfty}.}
(after adding the point $q=0$), which shows that the properties
\eqref{A}, \eqref{B} of the Coxeter function seem to be
correlated. In particular positive definiteness of the reflection
function is a different and complicated problem, strongly dependent on the choice of the
underlying group.

In this paper we generalize the notion of the reflection
function in the case of the Coxeter group of type B, that is the hyperoctahedral group $B(n)$, by
introducing its two-parameter version. In contrast to type $A$,
where all the reflections are conjugated, there are two conjugacy
classes of reflections in Coxeter groups of type $B$ and we can refine the reflection function by
introducing the \emph{signed reflection function} which distinguishes \emph{short}
and \emph{long reflections} appearing in the factorization of a given
element (see
\cref{def:SignedReflection} for the definitions):
\[ \sigma_{\q,\s}(g) := \q^{\ell_{\RRR_+(g)}}\s^{\ell_{\RRR_-(g)}}.\]
It is clear that this is a two-parameter refinement of the reflection
function, which can be obtained by the substitution $\q=\s=q$. We prove in
\cref{sec:preliminaries} that the signed reflection function is
central on the hyperoctahedral group for any value of $\q,\s \in \C$
and all orders $n \geq 0$, thus it gives rise to a central function on
the infinite group $B(\infty)$, which is the inductive limit of the
ascending tower of the hyperoctahedral groups $B(1) < B(2) < \cdots$. Our main theorem, proved in \cref{sec:main}, gives
the complete characterization of the set of parameters $\q,\s$ for which the
signed reflection function is positive definite.

\begin{theorem}
  \label{theo:PositiveDefinite'}
 Let $\q,\s \in \C$. The following conditions are equivalent:
  \begin{enumerate}[label=(\roman*), ref=\roman*]
  \item The signed reflection function $\phi_{\q,\s}\colon B(\infty) \to \C$ is
    positive definite on $B(\infty)$;
    \label{th:i}
 \item The signed reflection function $\phi_{\q,\s}\colon B(\infty) \to \C$ is
   a character of $B(\infty)$;
   \label{th:ii}
    \item The signed reflection function $\phi_{\q,\s}\colon B(\infty) \to \C$ is
      an extreme character of $B(\infty)$;
      \label{th:iii}
 \item $\q = \frac{\epsilon}{M+N},
    \s=\frac{M-N}{M+N}$ for $M,N \in \N, M+N\neq 0$, $\epsilon \in
    \{1,-1\}$ or \newline$\q=0, -1 \leq \s \leq 1$.
    \label{th:iv}
  \end{enumerate}
\end{theorem}

Note that the set of parameters for which the signed reflection
function is positive definite is a discrete set (except the degenerate case
$\q=0$), which confirms that the behaviour of the reflection function
is very different to its Coxeter counterpart and that the
properties \eqref{A} and \eqref{B} are correlated. Another difference
is visible in the strong correlation the parameters $\q$ and $\s$, while the parameters in the multivariate versions of the
Coxeter function studied in
\cite{BozejkoSzwarc2003,BozejkoGalMlotkowski2018} can take any value
in the interval $[-1,1]$ independently. The methods used in the
previously mentioned works on the Coxeter functions are not applicable
here and our approach is based on the connection between
the representation theory of the hyperoctahedral group and symmetric
functions. We prove that the Frobenius formula for the
hyperoctahedral group expressed in terms of symmetric functions
provides the full information on the expansion of the signed
reflection function into normalized irreducible characters. Additionally, for the parameters $N,M,\epsilon$ and the
associated parameters $\q,\s$ described in
\cref{theo:PositiveDefinite'} we construct an explicit action of the
hyperoctahedral group $B(n)$ on the tensor product of $n$ copies of
the $N+M$-dimensional space $V$, whose normalized character is given
by the signed reflection function $\phi_{\q,\s}$. This action gives a
representation which is a
two-parameter analog of the construction of Schur-Weyl.

In \cref{sec:Applications} we describe various applications of our
main theorem in theory of operator algebras and probability theory: we start by
constructing a cyclic Fock space of type B and proving in
\cref{commutation} that it admits a new cyclic commutation relation of
type B:
\[ \B(x\otimes y)\B^\ast(\xi\otimes \eta)= \langle  x,\xi
  \rangle\langle  y,\eta \rangle \id +\s \langle  x, \eta
  \rangle\langle  y, \xi \rangle \id  +   \Gamma_{\q}
  (|\xi\rangle\langle x| \otimes |\eta\rangle\langle y| ),\]
where $\B(x\otimes y)$ and $\B^\ast(\xi\otimes \eta)$ are the annihilation and
creation operators deformed by the use of the signed reflection
function $\phi_{\q,\s}$ and $\Gamma_{\q}$ is the differential second
quantisation operator (see \cref{subsec:Fock} for the precise definitions).
  In \cref{subsec:Exclusion} we deduce the analog of the Pauli exclusion
principle on this space. These results extend
the previous work of Bożejko and Guta \cite{BozejkoGuta2002} in Coxeter groups of type
A to type B and refine it by introducing two parameters
$\q,\s$. Note that the different constructions for type B were presented in
\cite{BozejkoEjsmontHasebe2015,BozejkoEjsmontHasebe2017}, where the
role of the signed reflection function was played by the Coxeter
function, which is not a character of $B(\infty)$. In
\cref{subsec:PartitionsB} we introduce the notion of positive and
negative cycles on the set of generalized symmetric pair-partitions, which share
many similarities with the  partitions of type B introduced by
Reiner \cite{Reiner1997}. We find the explicit formula (\cref{thm2}) for the moments of the
generalized cyclic Gaussian operator of type B that we introduce in
\cref{subsec:Gaussian}. This formula is an analog of the Wick-type
formula for Coxeter groups of type B and is naturally expressed in terms of the
combinatorial statistic on the set of
positive and negative cycles, developed in~\cref{subsec:PartitionsB}. We show that the probability
distribution $\mu_{\q,\s}$ of the cyclic Gaussian operator of type B with respect to
the vacuum state is strictly related to the Askey--Wimp--Kerov
distribution (\cref{rozkladgaussa}) and the associated Hermite polynomials:
$$t H_n(t) = H_{n+1}(t) +(n+c){H}_{n-1}(t), \qquad n=0,1,2,\dots,
\qquad {H}_{-1}(t)=0, {H}_0(t)=1.$$
It is known
\cite{BelinschiBozejkoLehnerSpeicher2011} that the Askey--Wimp--Kerov
distribution $\nu_c$ is freely infinitely divisible for $c \in [-1,0]$.
% we expect that 
Our work might shed new light on this
phenomenon. 
We use the connection between the moments of the cyclic Gaussian
operator of type B and the moments of the Askey--Wimp--Kerov
distribution to show that the
bivariate generating function of the symmetric pair-partitions (which
distinguishes negative and positive cycles)
specializes to the moments of the classical probability
measures. In particular, this gives new combinatorial interpretations of the
classical sequences $\frac{2^n}{n+1}\binom{2n}{n}$ and
$\frac{(2n)!}{n!}$ and their deformations with respect to certain
combinatorial statistics.
Finally, we discuss the analogous problem for Coxeter groups of type D and we show that
it can be obtained as a special case of our main result. In
particular, we
classify the positive-definite reflection functions for all the
irreducible infinite Coxeter groups of Weyl type.
%Finally, we show in \cref{subsubsec:BoundedOp} that the space
%of bounded operators on the cyclic Fock space of type B can be
%realized as the von Neumann algebra generated by certain cyclic Gaussian operators.

\section{Coxeter groups of type B/C}
\label{sec:preliminaries}

\subsection{Preliminaries on Coxeter groups of type B/C}
% \subsubsection{Coxeter group of type B/C and rank $n$}

General Coxeter system $(W,S)$ is defined as a group $W$ with chosen
set $S$ of generators such that $W$ can be presented, with help of auxiliary symmetric function $m\colon S\times S\to\mathbb{N}\cup\{\infty\}$
satisfying $m_{st}=1$ if and only if $s=t$, as
\[ W=\langle S| (st)^{m_{st}} \textrm{ for } m_{st}<\infty\rangle. \]
Elements conjugated to the generators are called \emph{reflections}.  Two generators $s$ and $t$ are conjugated
if an only if there exists a sequence of generators $s=s_1,\dots,s_n=t$ such that all $m_{s_is_{i+1}}$ are odd.

The main example for this paper is \emph{the hyperoctahedral group}
$B(n)$, which is the group of symmetries of the
$n$-dimensional hypercube. It is a finite real reflection group of type
B/C and rank $n$ generated by reflections $s_0,\dots,s_{n-1}$ where $s_0$ reflects in the plane $x_1=1$
and, for $i>0$, $s_i$ reflects in the plane $x_i=x_{i+1}$.

In terms of Coxeter groups it is defined with help of the function 
\[ m_{s_is_j}=\begin{cases}
    1&\textrm{if }i=j \in \{0,\dots,n-1\},\\
    2&\textrm{if }|i-j|>1,\\
    4&\textrm{if }\{i,j\} = \{0,1\},\\
    3&\textrm{otherwise.}
  \end{cases}
\]
The Coxeter diagram for $B(n)$  is described in \cref{fig:BN}.
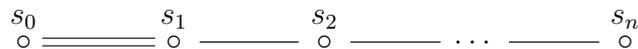
\begin{figure}[h]
\begin{center}
\begin{tikzpicture} 
  [scale=.5,auto=left,every node/.style={circle}]
  \node (n7) at (1,3.6) {$s_{n}$};  
  \node (n6) at (1,3) {$\circ$};
 % \node (n4) at (4,5)  {$\widehat{\pi}_{n-1}$};
  \node (n5) at (-3,3)  {$\dots$};
\node (n3) at (-7,3.6)  {$s_2$};
\node (n1) at (-7,3)  {$\circ$};
\node (n4) at (-11,3.6)  {$s_{1}$};
\node (n2) at (-11,3)  {$\circ$};
\node (n8) at (-15,3.6)  {$s_{0}$};
\node (n0) at (-15,3)  {$\circ$};
  % \node (n9) at (-13,3)  {$\succ$};

  \foreach \from/\to in {n2/n1, n1/n2,n1/n5,n5/n6}
    \draw (\from) -- (\to);
 \draw  (-14.5,3.1) -- (-11.5,3.1);   
  \draw  (-14.5,2.9) -- (-11.5,2.9);      
\end{tikzpicture}
\caption{Coxeter diagram for $B(n).$}
\label{fig:BN}
\end{center}
\end{figure}

One can also define the group $B(n)$ as wreath product $\Z_2\wr
\Sym{n}$ with multiplication given by
\[ (g_1,\dots,g_n;\sigma)\cdot  (g_1',\dots,g_n';\sigma') =
  (g_1g_{\sigma^{-1}(1)}',\dots, g_ng_{\sigma^{-1}(n)}'; \sigma\sigma'),\]
where $g_i,g_i' \in \{1,-1\}, \sigma, \sigma' \in \Sym{n}$. We refer
to this presentation as \emph{the signed model}.

In this model $s_0$ corresponds to $((-1,1,\dots,1),\textrm{id})$ and
$s_i$ corresponds to $((1,\dots,1),(i\ i+1))$.

  Denote by $[\pm n] $ the set of integers $\{\overline{n},\dots,\overline{1},1\dots,n\}$. The
hyperoctahedral group $B(n)$ can be also realised as the group of
permutations $\sigma$ of the set $[\pm n]$ such that
$\sigma(\overline{i}) = \overline{\sigma(i)}$ for \sloppy any $i \in
[\pm n]$ (with the convention that $\overline{\overline{i}} =
i$). These are precisely parmutations of $[\pm n] $ which commute with
the involution $(\overline{n} n)\cdots (\overline{1} 1)$. We will
refer to this realization as \emph{the permutation model}.

\begin{example}
  \label{ex:1}

Let $\sigma = (1 \overline{2} 4)(\overline{1} 2 \overline{4})(3 \overline{5} \overline{3} 5)(6)(\overline{6}) \in B(6)$ be an
element of the hyperoctahedral group $B(6)$ considered as a
permutation. Note that this element is uniquely determined by the
word 
% \[
%  w(\sigma) =\bigl(\begin{smallmatrix}
%     1 & 2 & 3 &  4 & 5 & 6 \\
%     \sigma(1)& \sigma(2)&\sigma(3)&\sigma(4)&\sigma(5)&\sigma(6) 
%   \end{smallmatrix}\bigr)= \bigl(\begin{smallmatrix}
%     1 & 2 & 3 &  4 & 5 & 6 \\
%     -2 &-4&-5 & 1 & 3 &6
%   \end{smallmatrix}\bigr).
% \]
$$w(\sigma) :=
(\sigma(1),\sigma(2),\sigma(3),\sigma(4),\sigma(5),\sigma(6) )=
(\overline{2},\overline{4},\overline{5},1,3,6).$$
This word
is uniquely determined by the permutation
\[
 w_+(\sigma) :=\bigl(\begin{smallmatrix}
    1 & 2 & 3 &  4 & 5 & 6 \\
    |\sigma(1)|& |\sigma(2)|&|\sigma(3)|&|\sigma(4)|&|\sigma(5)|&|\sigma(6)| 
  \end{smallmatrix}\bigr)= \bigl(\begin{smallmatrix}
    1 & 2 & 3 &  4 & 5 & 6 \\
    2 &4&5 & 1 & 3 &6
  \end{smallmatrix}\bigr) = (1 2 4)(3 5)(6)
\]
% $|\sigma(1)|=2,\dots,|\sigma(6)|=6$, 
% a 
% cycle $ (1 2 4)(3 5)(6)$ 
and by the sequence of signs $(\sgn(1)
\text{ in } w(\sigma),\dots, \sgn(6)
\text{ in } w(\sigma)) = (1,-1,1,-1,-1,1)$. Here
\[ \sgn(i) = \begin{cases} 1 &\text{ if } i \in \{1,2\dots,n\},\\-1
    &\text{ if } i \in
    \{\overline{1},\overline{2}\dots,\overline{n}\},\end{cases}\]
and $|\cdot|\colon [\pm n] \to [n]$ is the absolute value with the
  obvious definition.  
Therefore the associated
element in $\Z_2\wr
\Sym{6}$ is given by $(1,-1,1,-1,-1,1;(124)(35)(6))$. This procedure
adapted for arbitrary $n$ gives
an isomorphism between the permutation model and the signed model.
\end{example}

\subsection{Conjugacy classes}
A \emph{partition} $\rho$ \emph{of size $n$}, where $n$ is a non-negative integer (also denoted
$|\rho| = n$ or $\rho \vdash n$) is a non-increasing sequence
$(\rho_1,\dots,\rho_\ell)$ of integers which sum up to $n$:
\[ \sum_i \rho_i = n.\]
The integer $\ell$ is called the \emph{length} of the partition $\rho$
and it is denoted by $\ell(\rho)$. There is a unique partition $\rho =
\emptyset$ of size $0$ by convention.

The conjugacy classes  of $B(n)$
are naturally identified with pairs of partitions $(\rho^+,\rho^-)$ of
total size at most $n$, where the first partition $\rho^+$ has no parts
equal to $1$, i.e.
\begin{equation}
  \label{eq:ConjClassesCond}
  |\rho^+|+|\rho^-| \leq n; \qquad \rho^+_{i} > 1\text{ for } i=1,\dots,\ell(\rho^+).
  \end{equation}

\begin{remark}
  \label{rem:ConjCl}
  Note that these partitions are in a natural bijection with the set
  of pairs of partitions $(\rho^+,\rho^-)$ of
total size equal to $n$. Indeed, if $\tilde{\rho}^+$ is obtained from
$\rho^+$ by removing all its parts equal to $1$, then the pair
$(\tilde{\rho}^+,\rho^-)$ satisfies the conditions \cref{eq:ConjClassesCond}. This is
clearly invertible and for every pair of partitions $(\rho^+,\rho^-)$
which satisfies \cref{eq:ConjClassesCond} there is a unique way of adding an appropriate
number of parts equal to $1$ to $\rho^+$ to obtain a pair of
partitions of total size equal to $n$.
\end{remark}

This
identification is given by the following procedure. For each element
$(g_1,\dots,g_n;\sigma) \in B(n)$ we write $\sigma = c_1\cdots
c_\ell$ as the product of disjoint cycles (fix-points of $\sigma$ do not appear
in this decomposition). For any cycle $c = (a_1,\dots,a_k)$ we define
\[\sgn(c) := g_{a_1}\cdots g_{a_k} = \pm 1.\]
Then $\rho^+$ is the
partition given by the lengths of \emph{positive cycles}
\[\{c_i: i \in [\ell], \sgn(c_i)
  = 1\}\]
and $\rho^-$ is the
partition given by the lengths of \emph{negative cycles}
\[\{c_i: i \in [\ell], \sgn(c_i)
= -1\}.\]

% An equivalent parametrization of the conjugacy classes  is given by
% pairs of partitions $(\rho^+,\rho^-)$ such that $\rho^+$ does not
% contain parts equal to $1$ and $|\rho^+|+|\rho^-|\leq n$. The
% difference between this convention and the previous one is that
% now the partition $\rho^+$ encoding positive cycles does not count
% fixpoints. We will use the second convention, since it is more suited to the
% considerations of the infinite series of groups; see
% \cref{subsec:Infinite} for more details.

In the permutation model the pair of partitions
$(\rho^+,\rho^-)$ can be understood as follows. Let $\sigma \in B(n)$
be presented as a product of disjoint cycles $\sigma = c_1\cdots
c_k$ (here, as before, fix-points do not appear as cycles in the
decomposition). For any $i \in [k]$ consider a cycle $\overline{c_i} :=
(\overline{a^i_1}\cdots \overline{a^i_{k_i}})$, where $c_i = (a^i_1\cdots a^i_{k_i})$. Then either $c_i$
is disjoint with $\overline{c_i}$ (and we call it a \emph{positive cycle}) or $c_i =
\overline{c_i}$ (and we call it a \emph{negative cycle}). Let $\tilde{\rho}^+$ denote
the partition of lengths of positive cycles and $\tilde{\rho}^-$ denote
the partition of lengths of negative cycles. Note that
$\tilde{\rho}^+$ is necessarily of the form
$(\rho^+_1,\rho^+_1,\dots,\rho^+_\ell,\rho^+_\ell)$ and
$\tilde{\rho}^-$ is necessarily of the form
$(2\rho^-_1,\dots,2\rho^-_{\ell'})$. Then the pair $(\rho_+,\rho_-)$
is given by $\rho^+ = (\rho^+_1,\dots,\rho^+_\ell)$ and  $\rho^- =
(\rho^-_1,\dots,\rho^-_{\ell'})$.

We denote by $\CCC_{\rho^+,\rho^-}$ the conjugacy class associated
with the pair $(\rho_+,\rho_-)$.

\begin{example}
\label{ex:2}
We continue with \cref{ex:1}
 \begin{align*}
   \sigma &= (1,-1,1,-1,-1,1;(124)(35)(6)) \in
  \CCC_{\rho^+,\rho^-} \subset  B(6),
  \intertext{ where $\rho^+ = (3), \rho^- =
  (2)$. Using different presentation we have }
 \sigma &= (1 \overline{2} 4)(\overline{1} 2 \overline{4})(3 \overline{5} \overline{3} 5),
   \end{align*}
   which consists of one pair of positive
  cycles $(1 \overline{2} 4)(\overline{1} 2 \overline{4})$ of length $3$ and one negative cycle
  $(3 \overline{5} \overline{3} 5)$ of length $4 = 2\cdot 2$.
\end{example}

% \subsubsection{Coxeter group of type D and rank $n$}
% A reflection group $D(n)$ is generated by reflections $s_{-1},s_0,\dots,s_{n-2}$ and relations
% \begin{align*}
%   s_{0}^2 &= s_i^2 = 1, i \in [n-1],\\
%   (s_is_j)^2 &= 1, |i-j|>1,\\
%  (s_0s_2)^3 &= 1, (s_{i}s_{i+1})^3 = 1,  i \in [n-2].
% \end{align*}

% A reflection group $D(n)<B(n)$ can be realized as a subgroup of a
% hyperoctahedral group consisting of these colored permutations, whose
% number of negative colors is even. In other terms:
% \[ D(n) = \{(g_1,\dots,g_n;\sigma)\in B(n): g_1\cdots g_n = 1\}.\]

% Since $D(n)$ is a subgroup of $B(n)$ then the restriction of a signed
% reflection length $\phi_{\q,\s}$ to the group $D(n)$ is a well defined
% central function on $D(n)$. The Coxeter-Dynkin diagram for $D(n), n\geq4$ is described in Fig.\ \ref{fig:DCn}, where $\widehat{s}_{n-1}=\overline{s}_{n}s_{n-1}\overline{s}_{n}$ and $\overline{s}_n=(-n,n)$.
% \begin{figure}[h]
% \begin{center}
% \begin{tikzpicture} 
%   [scale=.5,auto=left,every node/.style={circle}]
%   \node (n6) at (1,3) {$s_{n-2}$};
%   \node (n4) at (4,5)  {$\widehat{s}_{n-1}$};
%   \node (n5) at (-3,3)  {$\dots$};
% \node (n1) at (-7,3)  {$s_{2}$};
% \node (n2) at (-11,3)  {$s_{1}$};
%   \node (n3) at (4,1)  {$s_{n-1}$};

%   \foreach \from/\to in {n6/n4,n6/n3,n1/n2,n1/n5,n5/n6}
%     \draw (\from) -- (\to);
% \end{tikzpicture}
% \caption{Coxeter-Dynkin diagram for $D(n).$}
% \label{fig:DCn}
% \end{center}
% \end{figure}

\subsection{Central functions}
\label{subsec:Infinite}

A function $\phi : G \to \C$ on a group $G$ is \emph{central} if it is
conjugacy invariant
\[\phi(g) = \phi(hgh^{-1}) \text{ for any } g,h \in G.\]
In
particular $\phi$ is fixed on the conjugacy classes.

For a Coxeter group $(G,S)$ and an element $g \in G$ % we denote by
% $\ell_S(g)$ its \emph{Coxeter length}, i.e. the minimal number of
% simple reflection that we need to use to decompose $g$ as its product
% \[ \ell_S(g) = \min(k : s_1\cdots s_k = g, s_1,\dots,s_k \in S),\]
% and 
we denote by
$\ell_{\RRR}(g)$ its \emph{reflection length}, i.e. the minimal number
of reflections that we need to use to decompose $g$ as their product
\[ \ell_{\RRR}(g) = \min(k : r_1\cdots r_k = g, r_1,\dots,r_k \in \RRR
  := \{gsg^{-1}: s \in S, g \in G\}).\]
% For any parameter $q \in \C$ one can consider the functions $g \to
% q^{\ell_S(g)}$ and $g \to q^{\ell_{\RRR}(g)}$. The first function was
% considered by Bożejko,  Januszkiewicz,  Speicher and Szwarc \cite{BozejkoJanuszkiewiczSpatzier1988,
%   BozejkoSpeicher1994,BozejkoSzwarc2003}, who showed that it is positive
% definite for any $-1 \leq q \leq 1$. Note however that this function
% is usually not central. The second function is always central, which
% is forced by the choice of the conjugacy invariant set of all
% reflections, so it can understand as a ``centralization'' of the first one. Its
% positive definitness was studied only in the very special case of the infinite symmetric group
% $S(\infty)$, which is the inductive limit of the ascending tower of
% Coxeter groups of type $A$ of order $n-1$ (that is the ascending tower
% of the permutation
% groups $A(n)$ of $n$ elements $A(1)< A(2) < \dots$). This case was
% considered by Bożejko and Guta \cite{BozejkoGuta2002}, who proved that
% the positive definitness of the function $g \to q^{\ell_{\RRR}(g)}$
% behaves very different then its counterpart $g \to
% q^{\ell_S(g)}$ and occurs on the discrete set $q \in
% \{0\}\cup\{\frac{1}{N}: N \in \Z\setminus\{0\}\}$.

The set of reflections in the symmetric group $\Sym{n}$ is given
by all the transpositions, and for a permutation $\sigma \in
\Sym{n}$ its reflection length $\ell_{\RRR}(\sigma)$ is naturally expressed
in terms of lengths of cycles of $\sigma$. Indeed, let $\rho \vdash n$
be a partition whose parts are equal to lengths of cycles in
$\sigma$. Then
\begin{equation}
  \label{eq:LengthInPerm}
  \ell_{\RRR}(\sigma) = \sum_{i=1}^{\ell(\rho)}(\rho_i-1) =
  |\rho|-\ell(\rho) =: \|\rho\|.
  \end{equation}

A similar formula holds true for the hyperoctahedral group. The set of reflections $\RRR$ in $B(n)$ consists of two conjugacy
classes $\RRR_+$ and $\RRR_-$ that we will call \emph{long
  reflections}, and \emph{short reflections}, respectively. These names
reflects the fact that the corresponding roots have length
$\sqrt{2}$ in tha case of $\RRR_+$ and $1$ in the case of $\RRR_-$. In the
signed model these reflections are given by

\begin{align*}
  \mathcal{R}_+ &=
                  \{(\overbrace{\underbrace{1,\dots,1,\epsilon}_i,1,\dots,1,\epsilon}^j,1,\dots,1;(i,j))
                  \colon \epsilon =\pm 1, 1 \leq i < j \leq n\}, \\
  \mathcal{R}_- &= \{(\underbrace{1,\dots,1,-1}_i,1,\dots,1;\id) \colon i \in
                  [n]\},
\end{align*}

and in the permutation model by

\begin{align*}
  \mathcal{R}_+ &= \{(ij)(\overline{i}\overline{j})\colon i,j \in [\pm n],  |i|
                  \neq |j|\}\\
  \mathcal{R}_- &= \{(i\overline{i}) \colon 1 \leq i \leq n\}.
\end{align*}

It is easy to see that we need at least $k-1$ reflections
$r_1,\dots,r_{k-1}$ to express a positive cycle $c_+ \in B(n)$ of length
$k$  (in the
signed model) as their product $c_+ = r_1\cdots r_{k-1}$. Similarly, we need at least $k$ reflections
$r_1,\dots,r_{k}$ to write a negative cycle $c_- \in B(n)$ of length
$k$ as a product $c_- = r_1\cdots r_{k}$. In particular, the
reflection length of $\sigma$ is given by
\[ \ell_{\RRR}(\sigma) =
  \sum_{i=1}^{\ell(\rho^+)}(\rho^+_i-1)+\sum_{i=1}^{\ell(\rho^-)}(\rho^-_i)
  = \|\rho^+\|+|\rho^-|,\]
where $\rho^+$ and $\rho^-$ list the lengths of positive and negative
cycles in $\sigma$, i.e.~$\sigma \in \CCC_{\rho^+,\rho^-}$.

In the case of the hyperoctahedral group $B(n)$ the notion of long
and short reflections suggests the possibility of extending the
univariate central function $\sigma \to q^{\ell_{\RRR}(\sigma)}$ to its
bivariate refinement $\sigma \to
\q^{\ell_{\RRR_+}(\sigma)}\s^{\ell_{\RRR_-}(\sigma)}$ defined as follows.

Suppose that $\sigma \in B(n)$ is expressed as a product of
reflections, where the number of reflections is minimal:
\begin{equation}
  \label{eq:factor}
\sigma=r_{1} \cdots r_{k}, \qquad r_i \in \RRR,
\end{equation}
Then, we would like to set
 \begin{align*}
\ell_{\RRR_+}(\sigma)=\ &\text{The number of long reflections $r_{i}, 1
                  \leq i \leq k$}\\ &\text{appearing in the factorization \cref{eq:factor}}, \\
\ell_{\RRR_-}(\sigma)=\ &\text{The number of short reflections $r_{i}, 1
                  \leq i \leq k$}\\ &\text{appearing in the factorization \cref{eq:factor}}.
 \end{align*}
The problem is that the functions $\ell_{\RRR_+}, \ell_{\RRR_-}$ are
not well-defined, since the number of
long/short reflections appearing in the minimal factorization
\cref{eq:factor} is not an invariant of the factorization. The minimal
example to see this can be already realized in $B(2)$, where the
element $(-1,-1;\id)$ (or the corresponding element $(1\overline{1})(2\overline{2})$ in the permutation
model) can be expressed as the product of two negative reflections
\[ (-1,-1;\id) = (1,-1;\id)\cdot (-1,1;\id);\qquad  (1\overline{1})(2\overline{2}) =
  (1\overline{1})\cdot(2\overline{2}),\]
but also as the product of two long reflections
\[ (-1,-1;\id) = (-1,-1;(12))\cdot (1,1;(12));\ \  (1\overline{1})(2\overline{2}) =
  (12)(\overline{1}\overline{2})\cdot (1\overline{2})(\overline{1}2).\]
In fact, it can be proved that the definition
\cref{eq:funkcjadlugosci'} does not depend on the choice of the
factorization \cref{eq:factor} if and only if $\sigma$ contains at
most one negative cycle, which is the consequence of the work
\cite{BaumeisterGobetRobertsWegener2017}. Nevertheless, there is a way
to correct the definition of
$\ell_{\RRR_+}(\sigma),\ell_{\RRR_-}(\sigma)$ which allows us to
introduce a bivariate central function on $B(n)$ called the \emph{signed
  reflection function} $\phi_{\q,\s} : B(n) \to \C$.

Consider the permutation model of $B(n)$. We say that a factorization
\cref{eq:factor} of $\sigma \in
B(n)$  is \emph{minimal, non-mixing} if:
\begin{enumerate}[label=(F\arabic*), ref=F\arabic*]
  \item \label{M} the number $k$ of reflections
    is minimal,
    \item \label{NM} for each $1 \leq i \leq k$ the support of
$r_i$ belongs to a cycle of $\sigma$. In other terms the orbits
of the action of $r_i$ on $[\pm n]$ form a
sub-partition of the orbits of the action of $\sigma$ on
$[\pm n]$ for each $1 \leq i \leq k$.
\end{enumerate}

\begin{definition}
  \label{def:SignedReflection}
  Let $\q,\s \in \C$ be parameters, and let $\sigma \in B(n)$.
Let
 \begin{align*}
\ell_{\RRR_+}(\sigma)=\ &\text{The number of long reflections $r_{i}, 1
                  \leq i \leq k$}\\ &\text{appearing in the minimal,
                                            non-mixed factorization \cref{eq:factor}}, \\
\ell_{\RRR_-}(\sigma)=\ &\text{The number of short reflections $r_{i}, 1
                  \leq i \leq k$}\\ &\text{appearing in the minimal,
                                            non-mixed factorization \cref{eq:factor}}.
 \end{align*}
We define the \emph{signed
  reflection function} $\phi_{\q,\s}\colon B(n) \to \C$ by
\begin{align}
\label{eq:funkcjadlugosci'}
\phi_{\q,\s}(\sigma) &:= \q^{\ell_{\RRR_+}(\sigma)}\s^{\ell_{\RRR_-}(\sigma)}.
%\\&= \q^{\min\{\tau_1,\dots,\tau_n \in \mathcal{T}: \rho
 % =\tau_1\cdots\tau_n\text{ and } \tau_i  \text{ is positive}\} }\s^{\min\{\tau_1,\dots,\tau_n \in \mathcal{T}: \sigma
  %=\tau_1\cdots\tau_n\text{ and } \tau_i  \text{ is negative}\}  }
\end{align}
\end{definition}

It is clear (assuming that the minimal, non-mixing factorizations
exist, which is proved in \cref{prop:well-def,lem:Factorization}) that
this function can be interpreted as a bivariate refinement of the
reflection function. Indeed
\[ \ell_{\RRR_+}(\sigma)+\ell_{\RRR_-}(\sigma) =
\ell_{\RRR}(\sigma)\]
so for $\q=\s=q$ we recover the reflection length function $\phi_{\q,\s}(\sigma) =
q^{\ell_{\RRR}(\sigma)}$.

\begin{proposition}
  \label{prop:well-def}
  The signed reflection function $\phi_{\q,\s}$ is central on
  $B(n)$ and for the element $\sigma \in \CCC_{\rho^+,\rho^-}$ of the
  conjugacy class associated with the pair $(\rho^+,\rho^-)$ it is given by the explicit formula:
\begin{align}\label{eq:funkcjadlugosci}
\phi_{\q,\s}(\sigma) &= \q^{\|\rho^+\|+\|\rho^-\|}\s^{\ell(\rho^-)}.
%\\&= \q^{\min\{\tau_1,\dots,\tau_n \in \mathcal{T}: \rho
 % =\tau_1\cdots\tau_n\text{ and } \tau_i  \text{ is positive}\} }\s^{\min\{\tau_1,\dots,\tau_n \in \mathcal{T}: \sigma
  %=\tau_1\cdots\tau_n\text{ and } \tau_i  \text{ is negative}\}  }
\end{align}
\end{proposition}

\begin{proof}
  We need to prove that $\ell_{\RRR^+}$ and $\ell_{\RRR^-}$ do not
  depend on the choice of the minimal, non-mixed factorization
  \eqref{eq:factor}. Strictly from the definition (non-mixing
  property \eqref{NM}), it is enough to prove the statement for the case when
  $\sigma$ is a positive/negative cycle. Indeed, each pair of reflections in the
  factorization \eqref{eq:factor} whose
  support belong to different orbits of $\sigma$ commutes, therefore the functions $\ell_{\RRR_+},\ell_{\RRR_-}$ are additive with
   respect to the decomposition into disjoint cycles. Let $\|\sigma\|$ denote the minimal number of transpositions to
  write $\sigma$ as its product (so that $\|\sigma\| = \|\rho\|$ for
  the partition $\rho$ which lists the lengths
  of cycles in $\sigma$, see \cref{eq:LengthInPerm}). For $r \in \RRR_{+}$ we have $\|r\| =
  2$ and for $r \in \RRR_{-}$ we have $\|r\| = 1$.

  Suppose that $\sigma$ is a positive cycle $\sigma = c \cdot
  \overline{c}$, where $c \cap \overline{c} = \emptyset$. Then $\|\sigma\| = 2
  \|c\|$, so that $\ell_{\RRR}(\sigma)=\|c\|$ implies that in the minimal,
  non-mixed factorization of $\sigma$ all the reflections are
  long and each such a factorization is of the form $r_i =
  \tau_i \cdot \overline{\tau_i}$, where $\tau_1\cdots \tau_k = c$ is a
  minimal factorization of $c$ as a product of transpositions.

   Suppose that $\sigma$ is a negative cycle $\sigma = c = \overline{c}$. Then
   $\|\sigma\|$ is odd, so that $\ell_{\RRR}(\sigma) \leq
   \frac{\|c\|+1}{2}$ and when the equality holds then either
   \[\ell_{\RRR_+}(\sigma)=\frac{\|c\|+1}{2} \text{ and }
     \ell_{\RRR_-}(\sigma)=0\]
   or
   \[\ell_{\RRR_+}(\sigma)=\frac{\|c\|-1}{2} \text{ and }
     \ell_{\RRR_-}(\sigma)=1.\]
   We claim that for the non-mixed
   factorization the first case is impossible. Indeed, if $\|c\|=1$,
   then $c$ is a short reflection. Suppose that this hypothesis holds for all
   the negative cycles $c$ such that $\|c\| < 2n+1$ and let $\|c\| =
   2n+1$. Let $r$ be a long reflection whose support is a subset
   of the support of $c$. These are the only reflections which appear
   in the minimal, non-mixed factorization of $c$. Note that $c \cdot r$
   is the disjoint product of two positive cycles $c_1,c_2$ (possibly
   fix-points) and one negative cycle $c'$ such that $\|c\| - 2 =
   \|c_1\|+\|c_2\|+\|c'\|$. Then, by induction, $\ell_{\RRR_-}(c) = 1$
   which finishes the proof.

   Finally, for $\sigma \in \CCC_{\rho^+,\rho^-}$ we will use the fact
   that the functions $\ell_{\RRR_+},\ell_{\RRR_-}$ are additive with
   respect to the decomposition into disjoint cycles, so that
   \[ \ell_{\RRR_+}(\sigma) =
     \|\rho_+\|+\|\rho_-\|, \qquad
     \ell_{\RRR_-}(\sigma) = \ell(\rho_-).\]
Therefore the formula \cref{eq:funkcjadlugosci} holds true and the signed reflection function is constant on the
conjugacy classes, so it is central.
  \end{proof}

  \subsubsection{Factorization lemma}

There is a natural embedding of $B(n-1)$ into $B(n)$, which implies
that there exists a canonical choice of the minimal, non-mixing
factorization \cref{eq:factor}. We will describe it using the
permutation model.
\begin{lemma}
\label{lem:Factorization}
We have the following factorization of the signed
reflection length:
\begin{equation}
\sum_{\sigma \in B(n)}\phi_{\q,\s}(\sigma)\cdot \sigma = \prod_{i=1}^n(1+\q\cdot J^+_i +
\s\cdot J^-_i),
\end{equation}
where 
\[J^+_i = \sum_{ j \in [\pm (i-1)]} (ji)(\overline{j} \overline{i}), \qquad J^{-}_i = (\overline{i}i).\]
\end{lemma}

\begin{proof}
It is known that there exist unique left coset representatives
$\{w(j)\colon j \in [\pm n]\}$ for $B(n)\backslash B(n-1)$ with minimal lengths 
given by 
 \begin{equation}
  w(j)= \begin{cases}\id &\mbox{for } j=n,
 \\(jn)(\overline{j} \overline{n}) &\mbox{for }  j \in  [\pm (n-1)],
 \\
(\overline{n}n)&\mbox{for } j=\overline{n}.
 \end{cases}\label{coset}
 \end{equation} 
%$$B(n-1) \cdot (j,n)(-j,-n)\cdot B_{n-1}\qquad \text{for } -n < j < n$$ and $(-n,n)\cdot \cdot B_{n-1}$,
%the identity holds true after specialization $\s=\q=1$. 
Let $\sigma \in B(n-1)$ and consider three cases according to the Equation \eqref{coset}. 
Obviously, in the first situation  one
has $\phi_{\q,\s}(\id \cdot  \sigma) =  \phi_{\q,\s}(\sigma)$. Suppose that 
$ j \in  [\pm (n-1)]$. If $j$ belongs to a positive cycle $c$ of
$\sigma$ or $j$ is a fix-point of $\sigma$ and $c=\id$ then 
$(jn)(\overline{j} \overline{n})\cdot c$ is a positive cycle of $(jn)(\overline{j} \overline{n})\cdot \sigma$ of
length increased by one. Similarly, if $j$ belongs to a negative cycle $c$ of $\sigma$ then 
$(jn)(\overline{j} \overline{n})\cdot c$ is a negative cycle of $(jn)(\overline{j} \overline{n})\cdot \sigma$ of
length increased by one. In both cases 
\[ \phi_{\q,\s}((jn)(\overline{j} \overline{n})\cdot\sigma) = \q \phi_{\q,\s}(\sigma).\]
In the third situation we have  $\phi_{\q,\s}((\overline{n}n)\cdot \sigma) = \s\cdot \phi_{\q,\s}(\sigma)$, which finishes the proof.
\end{proof}

\begin{remark}
  \label{rem:Infinite}
  Note that this embedding gives rise to the ascending tower of groups:
  \[B(1)< B(2) < \dots,\]
  which allows to define the
infinite group $B(\infty)$ as the inductive limit of
this tower.
It is clear that the conjugacy classes of $B(\infty)$ are parametrized
by pairs of partitions $(\rho^+,\rho^-)$ such that all parts of
$\rho^+$ are greater or equal to two. In particular the signed
reflection length $\phi_{\q,\s}$ can be naturally extended to the
infinite hyperoctahedral group $B(\infty)$
by using the same definition as in equation \eqref{eq:funkcjadlugosci}  
%\[ \phi_{\q,\s}(g) = \q^{|\rho^+|-\ell(\rho^+)+|\rho^-|-\ell(\rho^-)}\s^{\ell(\rho^-)}\]
for $\sigma \in \CCC_{\rho^+,\rho^-} \subset B({\infty})$. Thus,
$\phi_{\q,\s}$ is a central function on $B(\infty)$.
\end{remark}

%Since $D(\infty)<B(\infty)$ the
% signed reflection length is also well defined on the infinite Weyl
% group of type $D$.

% The
% difference between this convention and the previous one is that
% now the partition $\rho^+$ encoding positive cycles does not count
% fixpoints. We will use the second convention, since it is more suited to the
% considerations of the infinite series of groups; see
% \cref{subsec:Infinite} for more details. We denote by
% $\CCC_{\rho^+,\rho^-}$ the conjugacy class associated with the pair $(\rho^+,\rho^-)$.

% $$Contrary to the hyperoctahedral group $B(n)$ all the reflections in the
% Coxeter group $D(n)$ are conjugated.

% In particular, if we want to understand the multivariate reflection
% length of groups $B(n)$ and $D(n)$ it is convenient to understand the first
% definition (given by the presentation) in terms of the second
% one. This is standard to check that
% \[s_0 \to (1,1,\dots,1;\id); s_i \to (1,\dots,1;(i,i+1)), i \in
%   [n-1],\]
% is the desired isomorphism in the case of $B(n)$, and similarly
% \[s_{0} \to (-1,-1,\dots,1;\id); s_{i} \to (1,\dots,1;(i,i+1)), i \in
%   [n-1],\]
% in the case of $D(n)$.

% In particular $\mathcal{R}_{B(n)} = \mathcal{R}_n^+\cup \mathcal{R}_n^-$, where
% \begin{align*}
%   \mathcal{R}^+ &= \{(\overbrace{\underbrace{1,\dots,1,\epsilon}_i,1,\dots,1,\epsilon}^j,1,\dots,1;(i,j)): \epsilon \in\{1,-1\}, 1 \leq i < j \leq n\}, \\
%   \mathcal{R}^- &= \{(\underbrace{1,\dots,1,-1}_i,1,\dots,1;\id): i \in
%                   [n]\},
% \end{align*}
% and $\mathcal{R}_{D(n)}=\mathcal{R}_n^+$.

\subsection{Another view on non-mixing presentations}

Note that, in case of a single positive cycle, any minimal presentation takes only long reflections
while, for a negative cycle, a minimal presentation consists of long reflections and one
short one. It is easy to see that a minimal presentation is non-mixing
if and only if it contains the minimal number of long reflections (among all minimal presentations). It can
be understood in the following way.

Consider a natural map $\varphi\colon B(n)\to A(n-1)$ between
the groups of signed and unsigned permutations (since we would like to present our results in a way that might be applied
to other Coxeter groups we use standard notation $A(n-1)$ for the
Coxeter group of type A and rank $n-1$ which is
isomorphic with the permutation group $\Sym{n}$. It shall not be
confused with the alternating group i.e. the group of even permutations). It sends reflections from $\RRR_+$ to the reflections $\RRR'$
of $A(n-1)$ and reflections from $\RRR_-$ to the identity.  Both groups carry their reflection
lengths $\ell_\RRR$ and $\ell_{\RRR'}$ respectively.  Then
$\ell_{\RRR_+}(\sigma)=\ell_{\RRR'}(\varphi(\sigma))$ and $\ell_{\RRR_-}(\sigma)=\ell_{\RRR}(\sigma)-\ell_{\RRR'}(\varphi(\sigma))$.

Let $(W,S)$ be an arbitrary
Coxeter system and assume that $S$ is partitioned as $S=S'\cup S''$ in such a way that $m_{s's''}$ is
even for all $s'\in S'$ and $s''\in S''$ (that is no element of $S'$
is conjugated to an element of $S''$).
Then there is a natural map $\varphi\colon W \to W'$ between the Coxeter
groups $W$ and $W'$, where $W'$ is the group generated by $S'$. This map sends $S'$ into itself and $S''$ into the identity.  Given $w\in W$ let $\ell_{\RRR}(w)$
define the length of $w$ with respect to the reflections of $W$ and $\ell_{\RRR'}(w)$ define the length
of $\varphi(w)$ with respect to the reflections of $W'$ (which are the same as reflections in $W$
that belong to $W_{S'}$ \cite[Corollary 1.4]{Gal2005}).

\begin{question}
Is it true that $\ell_{\RRR'}(w)$ equals the minimal number of reflections conjugated to elements of $S'$
among all minimal presentations of $w$ as a product of reflections?
\end{question}
 
\section{The signed reflection function and the representation theory of $B(\infty)$%  and
% D
}
\label{sec:main}

In this section we will present our main result which gives a complete
characterization of the parameters $\q,\s \in \C$ for which the signed
reflection function is positive definite on $B(\infty)$. Since
$\phi_{\q,\s}$ is central it means that we classify when the signed
reflection function is a character of $B(\infty)$. We also find the
corresponding points in the Thoma simplex of $B(\infty)$. Finally, we
provide an explicit construction of the representations which realize
the character $\phi_{\q,\s}$, generalizing the Schur-Weyl
construction.

\subsection{Characters}

A function $\phi\colon G \to \C$ is \emph{positive definite} if
for any number $k \in \Z_{> 0}$ and any $z_1,\dots,z_k \in
    \C$, $g_1,\dots,g_k \in G$ we have
    \[\sum_{i,j=1}^{k}z_i \overline{z_j}\phi(g_j^{-1}g_i) \geq 0.\]

Note that the representation theory of a finite group $G$ can be
described by central, positive definite
functions. \Cref{lem:PDcharacter} below is well known
(see for instance \cite{BorodinOlshanski2017} for the proof), and it gives a correspondence between
normalized central positive definite functions on $G$ and probability
measures on the space $\widehat{G}$ of its irreducible
representations. This is an analogy to the classical Bochner theorem which gives a correspondence between
normalized continuous positive definite functions on $\R$ and probability measures on $\R$.

\begin{lemma}
  \label{lem:PDcharacter}
Let $G$ be a finite group and $f:G\to \C$ be a central, normalized
function. The function $f$ is positive definite if and only if it is a
convex combination of normalized characters of irreducible
representations of $G$.
\end{lemma}

When $G$ is infinite its irreducible representations might be infinite
dimensional therefore the definition of characters as traces do not apply
here. Motivated by \cref{lem:PDcharacter} it turns out that the
definition of characters which plays the essential role in the
representation theory of infinite groups should be modified as
follows.

\begin{definition}
  \label{def:character}
Let $G$ be an arbitrary group. A \emph{character} $\phi\colon G \to
\C$ is a central, positive-definite function which takes value $1$ on the identity.
\end{definition}

Our main theorem
gives a classification of the values $\q,\s$ for which the signed
reflection function $\phi_{\q,\s}\colon B(\infty) \to \C$ is a
character of $B(\infty)$.

\begin{theorem}
 \label{theo:PositiveDefinite}
 Let $\q,\s \in \C$. The following conditions are equivalent:
  \begin{enumerate}[label=(\roman*), ref=\roman*]
  \item The signed reflection function $\phi_{\q,\s}\colon B(\infty) \to \C$ is
    positive definite on $B(\infty)$; \label{i}
 \item The signed reflection function $\phi_{\q,\s}\colon B(\infty) \to \C$ is
    a character of $B(\infty)$; \label{ii}
        \item $\q = \frac{\epsilon}{M+N},
    \s=\frac{M-N}{M+N}$ for $M,N \in \N, M+N\neq 0$, $\epsilon \in
    \{1,-1\}$ or $\q=0, -1 \leq \s \leq 1$. \label{iii}
  \end{enumerate}
     Moreover, when these conditions are satisfied and $\q\neq 0$, the signed
    reflection function restricted to $B(n)$ is a normalized character of
    the explicitly constructed unitary representation, generalizing
    the Schur-Weyl construction.
\end{theorem}

% Note that for any reflection group $G$ and for any complex parameters $t_1,\dots,t_\ell \in \C$ the multivariate
% reflection length $f^G_{t_0,\dots,t_\ell}$ is a normalized, central
% function on $G$. This leads to the following question: for which
% parameters $t_1,\dots,t_\ell \in \C$ the multivariate
% reflection length $f^G_{t_0,\dots,t_\ell}$ is positive definite?
% Bożejko and Guta were studying $f^{\Sym{\infty}}_t$ and they showed
% that it is positive definite for $t \in \{0\}\cup\{\frac{1}{n}: n \in \Z^*\}$
% \cite{BozejkoGuta2002}. We are going to extend this result to other
% infinite types B/C and D. Moreover, we completely characterize the set
% of parameters for which the multivariate reflection length is positive
% definite and we construct the corresponding representation (our argument can be also used to show that
% $f^{\Sym{\infty}}_t$ is positive definite if and only if $t \in
% \{0\}\cup\{\frac{1}{n}: n \in \Z^*\}$).

\begin{figure}[h]
\begin{center}
\begin{tikzpicture}[
        %We set the scale and define some styles
        scale=0.7,
        axis/.style={very thick, ->, >=stealth'},
        important line/.style={thick},
        dashed line/.style={dashed, thick},
        every node/.style={color=black,}
     ]
    % Important coordinates are defined
     \begin{scope}[yscale=1.5]
       \coordinate (beg_1) at (0,-1);
    \coordinate (beg_2) at (0,1);
    \coordinate (dev_1) at (1,-.6);
    \coordinate (xint) at (2,-0.33);
     \coordinate (xint2) at (1.5,-0.6);
     \coordinate (xint3) at (2.5,-0.1);
    \coordinate (beg_3) at (3,-.142);
   \coordinate (beg_4) at (4,0);
    \coordinate (beg_5) at (5,0.111);
      \coordinate (beg_6) at (6,0.2);
        \coordinate (beg_7) at (7,0.27);
          \coordinate (beg_8) at (8,0.33);
           \coordinate (kropki) at (2,0);
    %We make some nice shading to annotate different parts of the curve
    % Everything for x<0
   % \begin{scope}
    %    \shade[top color=white, bottom color=red]
     %       ($(beg_2)+(0,.5)$) parabola bend (dev_1) (xint)
      %      (0,0) rectangle (beg_2);
   % \end{scope}
    %  Everything for x>0
    %\begin{scope}
  %      \shade[bottom color=white, top color=green]
   %         (xint) parabola bend (end) ($(end)+(0,-1.25)$);
    %\end{scope}
    % axis
    \draw[axis] (0,0)  -- (8.5,0) node(xline)[right] {$M$};
    \draw[axis] (0,-1) -- (0,1.7) node(yline)[above] {$\mathit{\s}$};
     \draw[dashed] (0,-1)  -- (8.5,-1) node(xline)[right] {};
      \draw[dashed] (0,1)  -- (8.5,1) node(xline)[right] {};
    % J curve is drawn
    %\draw[important line]
     %   (beg_1) -- (beg_2)
      %  (dev_1) parabola (xint)
       % (xint) parabola[bend at end] (end);
    % coordinates are added
    %\fill[red] (beg_2) circle (1pt) node[right] {$A$};
    \fill[red] (beg_1) circle (2pt) node[left] {$-1$};
     \fill[red] (beg_2) circle (0pt) node[left] {$1$};
       \fill[red] (dev_1) circle (2pt) node[left] {};
    \fill[red] (xint) circle (2pt) node[above left] {};
   %   \fill[red] (xint2) circle (2pt) node[above left] {};
    %    \fill[red] (xint3) circle (2pt) node[above left] {};
     \fill[red] (beg_3) circle (2pt) node[left] {};
      \fill[red] (beg_4) circle (2pt) node[left] {};
           \fill[red] (beg_5) circle (2pt) node[above] {};
              \fill[red] (beg_6) circle (2pt) node[above] {};
                    \fill[red] (beg_7) circle (2pt) node[above] {};
                       \fill[red] (beg_8) circle (2pt) node[above] {};
                      %  \fill[red] (kropki) circle (0pt) node[below] {$\dots$};
    % The time of the devaluation is added
    %\draw[dashed line] (beg_2) -- (dev_1);
    \draw[thick] (1,-1.5pt) -- (1,1.5pt) node[below] {};
     \draw[thick] (2,-1.5pt) -- (2,1.5pt) node[below] {};
      \draw[thick] (3,-1.5pt) -- (3,1.5pt) node[below] {};
          \draw[thick] (4,-1.5pt) -- (4,1.5pt) node[below] {$N$};
          \draw[thick] (5,-1.5pt) -- (5,1.5pt) node[below] {};
             \draw[thick] (6,-1.5pt) -- (6,1.5pt) node[below] {};
                \draw[thick] (7,-1.5pt) -- (7,1.5pt) node[below] {};
                   \draw[thick] (8,-1.5pt) -- (8,1.5pt) node[below] {};
         %  \draw[thick] (-0,-5pt) -- (0,-6pt) node[above] {$-1$};
                 \end{scope}
                 \begin{scope}[yscale=1.5, yshift=4cm]
\coordinate (beg_1) at (0,1/4);
    \coordinate (beg_2) at (0,1);
    \coordinate (dev_1) at (1,0.2);
    \coordinate (xint) at (2,1/6);
     \coordinate (xint2) at (1.5,-0.6);
     \coordinate (xint3) at (2.5,-0.1);
    \coordinate (beg_3) at (3,1/7);
   \coordinate (beg_4) at (4,1/8);
    \coordinate (beg_5) at (5,1/9);
      \coordinate (beg_6) at (6,1/10);
      \coordinate (beg_7) at (7,1/11);
          \coordinate (beg_8) at (8,1/12);
           \coordinate (kropki) at (2,0);
    %We make some nice shading to annotate different parts of the curve
    % Everything for x<0
   % \begin{scope}
    %    \shade[top color=white, bottom color=red]
     %       ($(beg_2)+(0,.5)$) parabola bend (dev_1) (xint)
      %      (0,0) rectangle (beg_2);
   % \end{scope}
    %  Everything for x>0
    %\begin{scope}
  %      \shade[bottom color=white, top color=green]
   %         (xint) parabola bend (end) ($(end)+(0,-1.25)$);
    %\end{scope}
    % axis
    \draw[axis] (0,0)  -- (8.5,0) node(xline)[right] {$M$};
    \draw[axis] (0,-1) -- (0,1.7) node(yline)[above] {$\mathit{\q}$};
     \draw[dashed] (0,-1)  -- (8.5,-1) node(xline)[right] {};
      \draw[dashed] (0,1)  -- (8.5,1) node(xline)[right] {};
    % J curve is drawn
    %\draw[important line]
     %   (beg_1) -- (beg_2)
      %  (dev_1) parabola (xint)
       % (xint) parabola[bend at end] (end);
    % coordinates are added
    % \fill[red] (beg_2) circle (1pt) node[right] {$A$};
      \fill[red] (0,-1) circle (0pt) node[left] {$-1$};
    \fill[red] (beg_1) circle (2pt) node[left] {};
     \fill[red] (beg_2) circle (0pt) node[left] {$1$};
       \fill[red] (dev_1) circle (2pt) node[left] {};
    \fill[red] (xint) circle (2pt) node[above left] {};
   %   \fill[red] (xint2) circle (2pt) node[above left] {};
    %    \fill[red] (xint3) circle (2pt) node[above left] {};
     \fill[red] (beg_3) circle (2pt) node[left] {};
      \fill[red] (beg_4) circle (2pt) node[left] {};
           \fill[red] (beg_5) circle (2pt) node[above] {};
              \fill[red] (beg_6) circle (2pt) node[above] {};
                    \fill[red] (beg_7) circle (2pt) node[above] {};
                       \fill[red] (beg_8) circle (2pt) node[above] {};
                      %  \fill[red] (kropki) circle (0pt) node[below] {$\dots$};
    % The time of the devaluation is added
    %\draw[dashed line] (beg_2) -- (dev_1);
    \draw[thick] (1,-1.5pt) -- (1,1.5pt) node[below] {};
     \draw[thick] (2,-1.5pt) -- (2,1.5pt) node[below] {};
      \draw[thick] (3,-1.5pt) -- (3,1.5pt) node[below] {};
          \draw[thick] (4,-1.5pt) -- (4,1.5pt) node[below] {$N$};
          \draw[thick] (5,-1.5pt) -- (5,1.5pt) node[below] {};
             \draw[thick] (6,-1.5pt) -- (6,1.5pt) node[below] {};
                \draw[thick] (7,-1.5pt) -- (7,1.5pt) node[below] {};
                   \draw[thick] (8,-1.5pt) -- (8,1.5pt) node[below] {};
         %  \draw[thick] (-0,-5pt) -- (0,-6pt) node[above] {$-1$};
                   \end{scope}
\end{tikzpicture}
\begin{tikzpicture}[
        %We set the scale and define some styles
        scale=0.7,
        axis/.style={very thick, ->, >=stealth'},
        important line/.style={thick},
        dashed line/.style={dashed, thick},
        every node/.style={color=black,}
     ]
     % Important coordinates are defined
          \begin{scope}[yscale=1.5]
    \coordinate (beg_1) at (0,1);
    \coordinate (beg_2) at (0,-1);
    \coordinate (dev_1) at (1,1/3);
    \coordinate (xint) at (2,0);
     \coordinate (xint2) at (1.5,0.6);
     \coordinate (xint3) at (2.5,0.1);
    \coordinate (beg_3) at (3,-1/5);
   \coordinate (beg_4) at (4,-2/6);
    \coordinate (beg_5) at (5,-3/7);
      \coordinate (beg_6) at (6,-4/8);
        \coordinate (beg_7) at (7,-5/9);
          \coordinate (beg_8) at (8,-6/10);
           \coordinate (kropki) at (2,0);
    %We make some nice shading to annotate different parts of the curve
    % Everything for x<0
   % \begin{scope}
    %    \shade[top color=white, bottom color=red]
     %       ($(beg_2)+(0,.5)$) parabola bend (dev_1) (xint)
      %      (0,0) rectangle (beg_2);
   % \end{scope}
    %  Everything for x>0
    %\begin{scope}
  %      \shade[bottom color=white, top color=green]
   %         (xint) parabola bend (end) ($(end)+(0,-1.25)$);
    %\end{scope}
    % axis
    \draw[axis] (0,0)  -- (8.5,0) node(xline)[right] {$N$};
    \draw[axis] (0,-1) -- (0,1.7) node(yline)[above] {$\mathit{\s}$};
     \draw[dashed] (0,-1)  -- (8.5,-1) node(xline)[right] {};
      \draw[dashed] (0,1)  -- (8.5,1) node(xline)[right] {};
    % J curve is drawn
    %\draw[important line]
     %   (beg_1) -- (beg_2)
      %  (dev_1) parabola (xint)
       % (xint) parabola[bend at end] (end);
    % coordinates are added
    %\fill[red] (beg_2) circle (1pt) node[right] {$A$};
    \fill[red] (beg_1) circle (2pt) node[left] {$1$};
     \fill[red] (beg_2) circle (0pt) node[left] {$-1$};
       \fill[red] (dev_1) circle (2pt) node[left] {};
    \fill[red] (xint) circle (2pt) node[above left] {};
   %   \fill[red] (xint2) circle (2pt) node[above left] {};
    %    \fill[red] (xint3) circle (2pt) node[above left] {};
     \fill[red] (beg_3) circle (2pt) node[left] {};
      \fill[red] (beg_4) circle (2pt) node[left] {};
           \fill[red] (beg_5) circle (2pt) node[above] {};
              \fill[red] (beg_6) circle (2pt) node[above] {};
                    \fill[red] (beg_7) circle (2pt) node[above] {};
                       \fill[red] (beg_8) circle (2pt) node[above] {};
                      %  \fill[red] (kropki) circle (0pt) node[below] {$\dots$};
    % The time of the devaluation is added
    %\draw[dashed line] (beg_2) -- (dev_1);
    \draw[thick] (1,-1.5pt) -- (1,1.5pt) node[below] {};
     \draw[thick] (2,-1.5pt) -- (2,1.5pt) node[above] {$M$};
        \draw[thick] (3,-1.5pt) -- (3,1.5pt) node[below] {};
          \draw[thick] (4,-1.5pt) -- (4,1.5pt) node[below] {};
          \draw[thick] (5,-1.5pt) -- (5,1.5pt) node[below] {};
             \draw[thick] (6,-1.5pt) -- (6,1.5pt) node[below] {};
                \draw[thick] (7,-1.5pt) -- (7,1.5pt) node[below] {};
                   \draw[thick] (8,-1.5pt) -- (8,1.5pt) node[below] {};
                   % \draw[thick] (-0,-5pt) -- (0,-6pt) node[above] {$-1$};
                 \end{scope}
                 \begin{scope}[yscale=1.5, yshift=4cm]
    \coordinate (beg_1) at (0,1);
    \coordinate (beg_2) at (0,-1/2);
    \coordinate (dev_1) at (1,-1/3);
    \coordinate (xint) at (2,-1/4);
     \coordinate (xint2) at (1.5,0.6);
     \coordinate (xint3) at (2.5,0.1);
    \coordinate (beg_3) at (3,-1/5);
   \coordinate (beg_4) at (4,-1/6);
    \coordinate (beg_5) at (5,-1/7);
      \coordinate (beg_6) at (6,-1/8);
        \coordinate (beg_7) at (7,-1/9);
          \coordinate (beg_8) at (8,-1/10);
           \coordinate (kropki) at (2,0);
    %We make some nice shading to annotate different parts of the curve
    % Everything for x<0
   % \begin{scope}
    %    \shade[top color=white, bottom color=red]
     %       ($(beg_2)+(0,.5)$) parabola bend (dev_1) (xint)
      %      (0,0) rectangle (beg_2);
   % \end{scope}
    %  Everything for x>0
    %\begin{scope}
  %      \shade[bottom color=white, top color=green]
   %         (xint) parabola bend (end) ($(end)+(0,-1.25)$);
    %\end{scope}
    % axis
\draw[axis] (0,0)  -- (8.5,0) node(xline)[right] {$N$};
    \draw[axis] (0,-1) -- (0,1.7) node(yline)[above] {$\mathit{\q}$};
     \draw[dashed] (0,-1)  -- (8.5,-1) node(xline)[right] {};
      \draw[dashed] (0,1)  -- (8.5,1) node(xline)[right] {};
    % J curve is drawn
    %\draw[important line]
     %   (beg_1) -- (beg_2)
      %  (dev_1) parabola (xint)
       % (xint) parabola[bend at end] (end);
    % coordinates are added
    % \fill[red] (beg_2) circle (1pt) node[right] {$A$};
      \fill[red] (0,-1) circle (0pt) node[left] {$-1$};
    \fill[red] (beg_2) circle (2pt) node[left] {};
     \fill[red] (beg_1) circle (0pt) node[left] {$1$};
       \fill[red] (dev_1) circle (2pt) node[left] {};
    \fill[red] (xint) circle (2pt) node[above left] {};
   %   \fill[red] (xint2) circle (2pt) node[above left] {};
    %    \fill[red] (xint3) circle (2pt) node[above left] {};
     \fill[red] (beg_3) circle (2pt) node[left] {};
      \fill[red] (beg_4) circle (2pt) node[left] {};
           \fill[red] (beg_5) circle (2pt) node[above] {};
              \fill[red] (beg_6) circle (2pt) node[above] {};
                    \fill[red] (beg_7) circle (2pt) node[above] {};
                       \fill[red] (beg_8) circle (2pt) node[above] {};
                      %  \fill[red] (kropki) circle (0pt) node[below] {$\dots$};
    % The time of the devaluation is added
    %\draw[dashed line] (beg_2) -- (dev_1);
    \draw[thick] (1,-1.5pt) -- (1,1.5pt) node[below] {};
     \draw[thick] (2,-1.5pt) -- (2,1.5pt) node[above] {$M$};
      \draw[thick] (3,-1.5pt) -- (3,1.5pt) node[below] {};
          \draw[thick] (4,-1.5pt) -- (4,1.5pt) node[below] {};
          \draw[thick] (5,-1.5pt) -- (5,1.5pt) node[below] {};
             \draw[thick] (6,-1.5pt) -- (6,1.5pt) node[below] {};
                \draw[thick] (7,-1.5pt) -- (7,1.5pt) node[below] {};
                   \draw[thick] (8,-1.5pt) -- (8,1.5pt) node[below] {};
                   % \draw[thick] (-0,-5pt) -- (0,-6pt) node[above] {$-1$};
                   \end{scope}
\end{tikzpicture}
\caption{Relation between $\q,\s$ and $M$, $N$. In the left hand side
  we fix $\epsilon=1,N=4$ and we look at the evolution of $M$ and in
  the right hand side we fix $\epsilon=-1,M=2$ and we look at the evolution of $N$.}
\label{fig:ms}
\end{center}
\end{figure}
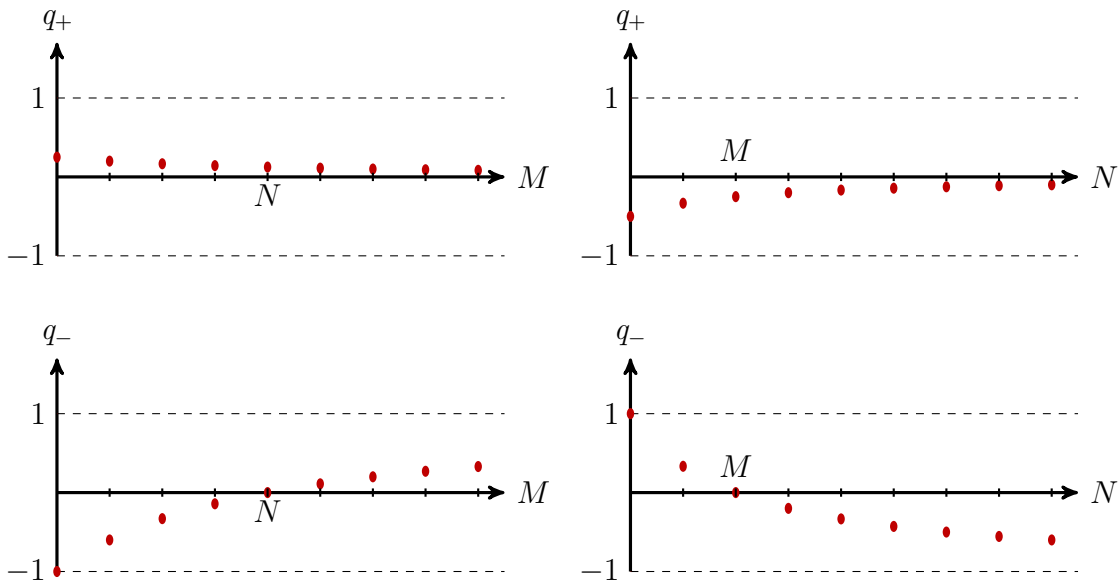

Before we prove our main theorem we quickly recall some information
about the representation
theory of the hyperoctahedral group $B(n)$; this information will be a necessary tool
in our proof of \cref{theo:PositiveDefinite}.

\subsection{Representation theory of the hyperoctahedral group}

The irreducible representations
of $B(n)$ are parametrized by pairs of partitions
$(\lambda^+,\lambda^-)$ such that $|\lambda^+|+|\lambda^-|=n$. We
denote by $\chi_{(\lambda^+,\lambda^-)}$ the trace of the
irreducible representation $\rho_{(\lambda^+,\lambda^-)} \in
  \widehat{B(n)}$ associated with the pair
  $(\lambda^+,\lambda^-)$. The irreducible characters $\chi_{(\lambda^+,\lambda^-)}$ are central on $B(n)$ and we denote by
$\chi_{(\lambda^+,\lambda^-)}(\rho^+,\rho^-)$ the value of
$\chi_{(\lambda^+,\lambda^-)}$ on the conjugacy class $\CCC_{\rho^+,\rho^-} \subset B(n)$. Here, we use the
convention that the conjugacy classes of $B(n)$ are parametrized by
pairs of partitions of total size $n$, so that some parts of $\rho^+$
might be equal to $1$ (see \cref{rem:ConjCl}).

Following \cite{Poirier1998} we
are going to relate the representation theory of $B(n)$ with the theory of symmetric
functions. We denote by $\xx$ ($\yy$ respectively) the infinite alphabet
$x_1,x_2\dots$ ($y_1,y_2\dots$). Let $\epsilon \in \{+,-\}$ and let
\[p_k^\epsilon(\xx,\yy) := \sum_{i\geq
    1}x_i^k+\epsilon \sum_{i\geq 1}y_i^k.\]
Finally, for a partition $\rho$ we define
\[p_\rho^\epsilon(\xx,\yy) :=
  \prod_{i=1}^{\ell(\rho)}p_{\rho_i}^\epsilon(\xx,\yy)\]
and for a pair of partitions $(\rho^+,\rho^-)$ we set
\[ p_{(\rho^+,\rho^-)}(\xx,\yy) := p^+_{\rho^+}(\xx,\yy)p^-_{\rho^-}(\xx,\yy).\]
Let $s_\lambda(\xx)$
denote the Schur symmetric function in the infinite alphabet $\xx =
(x_1,x_2,\dots)$.
We use the following Frobenius formula  for the direct calculation of
$  p_{(\rho^+,\rho^-)}(\xx,\yy)$ in terms of irreducible characters $\chi_{(\lambda^+,\lambda^-)}$.
\begin{lemma}[Frobenius formula, \cite{Poirier1998}]
  \label{lem:FrobeniusTypeB}
  For any pair of partitions $(\rho^+,\rho^-)$ of
  total size equal to $n$ we have the following equality:
  \begin{equation}
    \label{eq:FrobeniusTypeB}
    p_{(\rho^+,\rho^-)}(\xx,\yy)  = \sum_{(\lambda^+,\lambda^-) \in \hat{B}(n)}\chi_{(\lambda^+,\lambda^-)}(\rho^+,\rho^-) s_{\lambda^+}(\xx) s_{\lambda^-}(\yy).
    \end{equation}
  \end{lemma}

\subsection{Proof of the main result }

We are ready to prove \cref{theo:PositiveDefinite}.

\begin{proof}[Proof of \cref{theo:PositiveDefinite}]

The signed reflection function is normalized $\phi_{\q\s}(\id) = 1$, therefore it is clear from \cref{def:character} and
from \cref{rem:Infinite} that the conditions \eqref{i} and \eqref{ii} are equivalent.

We will now prove that \eqref{i} $\Leftrightarrow$ \eqref{iii}. The signed reflection length $\phi_{\q,\s}$ is
positive definite on $B(\infty)$ if and only if it is positive
definite on $B(n)$ for each $n$. Fix two positive integers $N,M \in \Z_{>0}$, set
\[x_i = \begin{cases}1 &\text{ if }
  i\leq M,\\0 &\text{ otherwise;}\end{cases}\qquad y_i = \begin{cases}1 &\text{ if }
  i\leq N,\\0 &\text{ otherwise;}\end{cases}\]
and plug these families $(\xx,\yy)$ into \eqref{eq:FrobeniusTypeB}. We
obtain the following equality
\begin{align*}
  (N+M)^{n-\|\rho^+\|-|\rho^-|}(M-N)^{\ell(\rho^-)} &=\\
  \sum_{(\lambda^+,\lambda^-) \in
    \widehat{B}(n)}\chi_{(\lambda^+,\lambda^-)}(\rho^+,\rho^-)\cdot
  &\bigg(s_{\lambda^+}(\underbrace{1,\dots,1}_M,0\dots) s_{\lambda^-}(\underbrace{1,\dots,1}_N,0\dots)\bigg).
  \end{align*}
We use the classical content-formula for Schur polynomials (see for
instance \cite{Stanley:EC2})
\[ s_{\lambda}(1^N) = \prod_{\square
    \in \lambda}(N+c(\square)),\]
where $c(\square) := x-y$ is the content of the box $\square = (x,y)$.
Plugging it into
the previous identity we obtain
\begin{align}
  (N+M)^{n-\|\rho^+\|-|\rho^-|} &(M-N)^{\ell(\rho^-)} = \nonumber \\
  \sum_{(\lambda^+,\lambda^-) \in
    \widehat{B}(n)}\chi_{(\lambda^+,\lambda^-)}(\rho^+,\rho^-)\prod_{\square
  \in \lambda^+}&(M+c(\square)) \prod_{\square
                  \in \lambda^-}(N+c(\square)).
                    \label{eq:PolynomialIdentity}
\end{align}
Note that the left and the right hand sides of
\eqref{eq:PolynomialIdentity} are both polynomials in $N,M$
which are equal on the grid $\Z_{>0}\times \Z_{>0}$. Therefore equality
\eqref{eq:PolynomialIdentity} holds true for any real parameters $N,M
\in \RR$. Using the following change of variables $\q^{-1} = N+M,
\q^{-1}\s = M-N$ we obtain
\begin{equation}
  \label{eq:rozklad}
  \phi_{\q,\s} = \sum_{(\lambda^+,\lambda^-) \in
    \widehat{B}(n)}\chi_{(\lambda^+,\lambda^-)}\prod_{\square
  \in \lambda^+}\left(\q c(\square)+\frac{1+\s}{2}\right) \prod_{\square
  \in \lambda^-}\left(\q c(\square)+\frac{1-\s}{2}\right).
\end{equation}
This in conjunction with \cref{lem:PDcharacter} gives a sufficient and
necessary condition for
$\phi_{\q,\s}$ to be positive definite:
\[ \prod_{\square
  \in \lambda}\left(\q c(\square)+\frac{1+\s}{2}\right) \prod_{\square
  \in \mu}\left(\q c(\square)+\frac{1-\s}{2}\right)\geq 0\]
for any partitions $\mu,\lambda$. This condition is satisfied only for
$\q,\s$ as in our hypothesis, which finishes the proof of the
equivalence \eqref{i} $\Leftrightarrow$ \eqref{iii}.

\vspace{5pt}

Suppose that the parameters $\q,\s$ are given by \eqref{iii}. We will construct a
  representation $\omega_n$ of
$B(n)$ whose normalized
character
\[\frac{\Tr \omega_n(\cdot)}{\Tr \omega_n (\id)}\colon B(n) \to \C\]
coincides with the signed reflection function
\[ \phi_{\q,\s}(\cdot) \colon B(n) \to \C.\]

  Our construction generalizes
  the action of the symmetric group on the tensor
  product of a fixed vector space.
  Consider
  \[ V = \Span\{e^{+}_1,\dots,e^+_{M},e_1^{-},\dots,e_{N}^{-}\}.\]
  The hyperoctahedral group $B(n)$ acts on
  $W := \underbrace{V\otimes\cdots\otimes V}_n$ as follows
  \[ (g_1,\dots,g_n;\sigma)\cdot f_1\otimes\cdots\otimes f_n := \epsilon^{\|\rho^+\|+\|\rho^-\|}c_1(
    f_{\sigma^{-1}(1)})\otimes\cdots\otimes c_n( f_{\sigma^{-1}(n)}),\]
where $c_i$
is an operator $c_i: V\mapsto V $ defined as:
  \[ c_i( f) = \begin{cases} f &\text{ if } f=e^+_j,\\ g_i\cdot f
      &\text{ if } f=e^-_j,\end{cases}\]
  and $(g_1,\dots,g_n;\sigma) \in
  \CCC_{\rho^+,\rho^-}$.
  It is straightforward to check that this action extended by multilinearity defines the
  representation of $B(n)$ on $W$. Let us compute the character of
  this representation. Pick an element $g=(g_1,\dots,g_n;\sigma) \in
  \CCC_{\rho^+,\rho^-} \subset  B(n)$ and note that the only elementary tensors from $W$ 
  contributing to the trace of $\omega_n(g)$ have
  necessarily the same vectors $e$ in all the coordinates corresponding to the
  points in the same cycle of $\sigma$. Let us choose this vector
  equal to $e_j^{\pm}$ for a fixed cycle $c$ of length $\ell(c)$. Then, the eigenvalue contributing to this cycle is
  given by
  \[
    \begin{cases} \epsilon^{\ell(c)-1}
    &\text{ when $c$ is a positive cycle},\\ \pm \cdot \epsilon^{\ell(c)-1}
    &\text{ when $c$ is a negative cycle}. \end{cases}\]
  This means that
  \[ \Tr \omega_n(g) = \epsilon^{\|\rho^+\|+\|\rho^-\|}(N+M)^{n-\|\rho^+\|-|\rho^-|}(M-N)^{\ell(\rho^-)}.\]
Moreover the dimension of $W$ is equal to $(N+M)^n$, therefore we have the following formula
\[ \frac{\Tr \omega_n(g)}{\Tr \omega_n (\id)} =
  \left(\frac{\epsilon}{N+M}\right)^{\|\rho^+\|+\|\rho^-\|}\cdot \left(\frac{M-N}{N+M}\right)^{\ell(\rho^-)}=\phi_{\q,\s}(g).\]

\end{proof}

\begin{remark}
  \label{remark:SymInfty}
  Note that the classical Frobenius formula relates the irreducible
  characters $\chi_\lambda$ of the symmetric group
  $\Sym{n}$ with the
  Schur symmetric function $s_\lambda(\xx)$:
  \[ p_\mu(\xx) = \sum_{\lambda \vdash
      n}\chi_\lambda(\mu)s_\lambda(\xx).\]
 Using the same arguments as in the proof of
 \cref{theo:PositiveDefinite} one can apply the Frobenius formula to show that the reflection
 function $g \to q^{\ell_S(g)}$ on the infinite symmetric group
 $\Sym{\infty}$ is positive definite if and only if $q=0$ or $q^{-1}
 \in \Z$.
  \end{remark}

\subsection{The signed reflection length and extreme characters of $B(\infty)$}
\label{subsec:Hirai}

The set of characters of the group $B(\infty)$ forms an infinite
dimensional simplex. A character $\phi: B(\infty) \to \C$ is called
\emph{extreme} if it belongs to the extreme points of the
simplex. Since any simplex is uniquely determined by its extreme
points, the classification problem of the characters of $B(\infty)$
reduces to the problem of characterizing the extreme characters. The
classification of all the extreme characters of the group $B(\infty)$, which
is a type $B$ analog of the classical Thoma's theorem, was found by
Hirai and Hirai \cite{HiraiHirai2002}. They proved that the set of
extreme characters of $B(\infty)$ is parametrized by the
following sequences:
\begin{align*}
\alpha_1\geq\alpha_2 \geq \dots &\in \RR_{\geq 0}^\infty,\\
  \beta_1\geq\beta_2 \geq \dots &\in \RR_{\geq 0}^\infty,\\
  \gamma_1\geq\gamma_2 \geq \dots &\in \RR_{\geq 0}^\infty,\\
  \delta_1\geq\delta_2 \geq \dots &\in \RR_{\geq 0}^\infty,\\
  \kappa &\in \RR \text{ such that }\\
  \|\alpha\|+\|\beta\|+\|\gamma\|+\|\delta\| + |\kappa| &= \sum_{i
                                                          \geq
                                                          0}(\alpha_i+\beta_i+\gamma_i+\delta_i)+|\kappa|
                                                          \leq 1.
\end{align*}
For any such sequences $\alpha,\beta,\gamma,\delta,\kappa$ the value of the
associated extreme character
$\psi_{\alpha,\beta,\gamma,\delta,\kappa} $ on the conjugacy class
$\CCC_{(\rho^+,\rho^-)}$ is given by the following formula:
\begin{align*}
  &\psi_{\alpha,\beta,\gamma,\delta,\kappa}(\rho^+,\rho^-) =
  (\|\a\|+\|\beta\|-\|\gamma\|-\|\delta\|+\kappa)^{m_1(\rho^-)}\times\prod_{\epsilon
                                                         \in \{+,-\}}\times\\
&\prod_{j=2}^\infty\left(\sum_{i=1}^\infty\alpha_i^{j}+(-1)^{j-1}\sum_{i=1}^\infty\beta_i^{j}+\epsilon\sum_{i=1}^\infty\gamma_i^{j}+\epsilon
                                                                                (-1)^{j-1}\sum_{i=1}^\infty\delta_i^{j}\right)
                                                                                ^{m_j(\rho^\epsilon)}.
\end{align*}
It is straightforward to check that the signed reflection
function $\phi_{\q,\s}$ with parameters $\q,\s$ given by \eqref{iii}
coincides with the extreme character of $B(\infty)$ corresponding to
the sequences
\begin{align*}
\alpha &= (\underbrace{\frac{1}{M+N},\dots,\frac{1}{M+N}}_M,0\dots),
         \qquad\beta=0, \\
  \gamma &= (\underbrace{\frac{1}{M+N},\dots,\frac{1}{M+N}}_N,0\dots), \qquad\delta=0, \qquad\kappa=0
\end{align*}
when $\q=\frac{1}{M+N},\s=\frac{M-N}{N+M}$,
\begin{align*}
\alpha &= 0,
         \qquad\beta=(\underbrace{\frac{1}{M+N},\dots,\frac{1}{M+N}}_M,0\dots),\\
  \gamma &= 0, \qquad\delta=(\underbrace{\frac{1}{M+N},\dots,\frac{1}{M+N}}_N,0\dots), \qquad\kappa=0
\end{align*}
when $\q=\frac{-1}{M+N},\s=\frac{M-N}{N+M}$, and
\begin{align*}
\alpha &= \beta = \gamma = \delta = 0,\qquad \kappa = \s
\end{align*}
when $\q=0, |\s|\leq 1$.

These considerations in conjunction with
\cref{theo:PositiveDefinite} give directly \cref{theo:PositiveDefinite'}.

% The extreme characters of $D(\infty)$ are parametrized by the same set as the
% extreme characters of $B(\infty)$ and are obtained by
% restriction. The only difference is that now
% $f_{\alpha,\beta,\gamma,\delta,\kappa}$ might be equal to
% $f_{\alpha',\beta',\gamma',\delta',\kappa'}$. This happens precisely
% when
% \[ (\alpha,\beta,\gamma,\delta,\kappa) =
%   (\gamma',\delta',\alpha',\beta',\kappa').\]
% \todo[inline]{MD: UWAGA! Ten wynik sugeruje że być może
%   $f^{D(\infty)}_t$ NIE JEST właściwą funkcją, którą chcielibyśmy
%   zrozumieć. Dwuwymiarowość NIE JEST zabita na poziomie charakterów
%   $D(\infty)$, ale jest zabita na poziomie klas sprzężoności... W
%   szczególności chyba dowód w głównego twierdzenia w przypadku D jest
%   skonocony, ale chcę jak najszybciej wysłać to co spisałem, więc nie
%   mam czasu teraz poprawić.}

\section{Applications}
\label{sec:Applications}

In this section, we assume that the parameters    $\q$ and  $ \s$ are as in \cref{theo:PositiveDefinite'}. 

\subsection{Cyclic Fock space  of type B}
\label{subsec:Fock}

Let $H_\R$ be a separable real Hilbert space and let $H$ be its
complexification with the inner product $\langle\cdot,\cdot\rangle$ linear on the right component and anti-linear on the left. 
% When considering elements in $H_\R$, it holds true that $\langle x,y\rangle=\langle y,x\rangle$. 
%We assume that there exists a self-adjoint involution $x\mapsto \overline{x}$ for $x\in H$ and define 
%$$
%x_{-i}:=\overline x_{i}, \qquad i \in \{1,2,3,\dots\}.
%$$
The Hilbert space $\HH:=H\otimes {H}$ is the complexification of its
real subspace $\HH_\R:=H_\R\otimes {H}_\R$, with the inner product
\[\langle x\otimes y,\xi\otimes \eta  \rangle_{\HH} = \langle
  x,\xi\rangle\langle y,\eta\rangle.\]
%  Moreover, we
% use the superscript $\overline i$
% , in order to labelled vectors in $H$ 
%   and so  we  also consider additional
% map %$\overline{n},\dots,\overline{1}$, where 
% \begin{equation}
%   \label{eq:ED}
% \begin{aligned}
% \overline{ }  :\Z&\to \Z \\
%   i &\mapsto -i.
% \end{aligned}
% \end{equation}
%  This in particular
% means that there is no relation between vectors $x_i$ and $x_{\overline i}$. 
We define $\H:=H^{\otimes n} \otimes {{H}^{\otimes n}}=H^{\otimes 2n}$ and instead of
indexing its simple tensors by $\{1,,\dots,2n\}$ we will index
them by $[\pm n]$:
$$\H\ni\mathbf{x}_{\overline n} \otimes \mathbf{x}_n=x_{\overline n}\otimes\cdots \otimes x_{\overline 1} \otimes  x_1\otimes\cdots \otimes x_n=x_{\overline n}\otimes \cdots  \otimes x_{n}.$$
We use this convention for indexing the elements of $\H$ to define a natural action of the
hyperoctahedral group $B(n)$ on $\H$ by setting:
\begin{align*}
\sigma  :\H&\to \H \\
x_{\overline n}\otimes\cdots \otimes x_{\overline 1} \otimes  x_1\otimes\cdots \otimes x_n &\mapsto x_{\sigma(\overline n)}\otimes\cdots \otimes x_{\sigma(\overline 1)} \otimes x_{\sigma(1)}\otimes\cdots \otimes x_{\sigma(n)}.
% x_1 \otimes \cdots \otimes x_{i-1} \otimes x_{i+1} \otimes x_{i} \otimes x_{i+2}\otimes \cdots \otimes x_{n},& n \geq2,  %\label{Coxeter1} \\
%&\pi_0(x_1\otimes\cdots \otimes x_n)= \overline{x_1}\otimes x_2\otimes\cdots \otimes x_n,& n \geq 1
%\label{Coxeter2}
\end{align*}
for any $\sigma \in B(n)$.
Let $\F$ be the (algebraic) full Fock space over $\HH$
\begin{equation}
\F:= \bigoplus_{n=0}^\infty \H = \bigoplus_{n=0}^\infty  H^{\otimes 2n} 
\end{equation} 
with the convention that $\HH^{\otimes 0} =H^{\otimes 0} \otimes H^{\otimes 0}=\C\Omega \otimes \Omega$ is the one-dimensional normed space along the unit vector $\Omega \otimes \Omega$. Note that elements of $\F$ are finite linear combinations of the elements from $H^{\otimes 2n}, n\in \N\cup\{0\}$ and we do not take the completion. 
We equip $\F$ with the inner product  
$$
\langle x_{\overline n} \otimes \cdots \otimes  x_{\overline 1} \otimes  x_1 \otimes \cdots \otimes x_n, y_{\overline m} \otimes \cdots \otimes y_{\overline 1} \otimes  y_1 \otimes \cdots \otimes y_m\rangle_{0,0}:= \delta_{m,n}\prod_{i=\overline n }^n \langle x_i, y_i\rangle.  
$$ 
 For the parameters $\q$ and $\s$ described in \cref{theo:PositiveDefinite'} \eqref{th:iv} we define the symmetrization operators
 \begin{align*}
&P_{\q,\s}^{(n)}:= \sum_{\sigma \in B(n)}\phi_{\q,\s}(\sigma) \, \sigma,\qquad n \geq1, \\ 
&P_{\q,\s}^{(0)}:= \id_{H^{\otimes 0}\otimes H^{\otimes 0}}. 
\intertext{Moreover let }
&P_{\q,\s}:=\bigoplus_{n=0}^\infty P_{\q,\s}^{(n)}
\end{align*}  be the \emph{cyclic type B symmetrization operator} acting on the algebraic full Fock space. 
For $\x\in \HH_{ n}$ and $\mathbf{y}_{\overline m} \otimes \mathbf{y}_m\in
\HH_{ m}$ we deform the inner product $\langle\cdot,\cdot\rangle_{0,0}$ by using the cyclic type B symmetrization operator: 
\begin{align*}
\langle \x,\mathbf{y}_{\overline m} \otimes \mathbf{y}_m\rangle_{\q,\s}:=\delta_{n,m}\langle \x,P_{\q,\s}^{(m)}\mathbf{y}_{\overline m} \otimes \mathbf{y}_m\rangle_{0,0}
% \langle x_n \otimes \cdots \otimes  x_1 \otimes  y_1 \otimes \cdots \otimes y_n&, \xi_m \otimes \cdots \otimes \xi_1 \otimes  \eta_1 \otimes \cdots \otimes \eta_m\rangle_{\q,\s}\\ &:=\langle x_n \otimes \cdots \otimes  x_1 \otimes  y_1 \otimes \cdots \otimes y_n, P_{\q,\s}^{(n)}(\xi_m \otimes \cdots \otimes \xi_1 \otimes  \eta_1 \otimes \cdots \otimes \eta_m)\rangle_{0,0}, 
\end{align*}
which by \cref{theo:PositiveDefinite'} is a semi-inner product from the positivity of $P_{\q,\s}$. 
%We restrict the parameters to the case $\q,\s \in(-1,1)$ so that the deformed semi-inner product is an inner product. 
For $x\in H$ let $l(x)$ and $r(x)$ be the free left and right
annihilator operators on $H^{\otimes n} $,  respectively, defined by
the equations
\begin{align*}
&l^\ast(x)(x_1\otimes \dots \otimes x_n ):=x\otimes x_1\otimes \dots \otimes x_n  ,
\\
&l(x)(x_1\otimes \dots \otimes x_n ):=\langle  x, x_1 \rangle  x_2\otimes \dots \otimes x_n , \\
&r^\ast(x)(x_1\otimes \dots \otimes x_n ):= x_1\otimes \dots \otimes x_n \otimes x,\\
&r(x)(x_1\otimes \dots \otimes x_n ):=\langle  x, x_n \rangle  x_1\otimes \dots \otimes x_{n-1}, 
\intertext{where the adjoint is taken with respect to the free inner product. 
The  left-right creation and  annihilation operators $\r^\ast(x\otimes y), \r(x\otimes y)$ on $\F$  are defined by }
&\r^\ast(x\otimes y)(\x):=l^\ast(x)\mathbf{x}_{\overline n}\otimes 
 r^\ast( y) \mathbf{x}_n, \quad  &&\r^\ast(x\otimes y)\Omega \otimes \Omega:=x \otimes y, \\
&\r(x\otimes y)(\x ):= l( x) \mathbf{x}_{\overline n}\otimes r(y)\mathbf{x}_n  , \quad &&\r(x\otimes y)\Omega \otimes \Omega :=0,
\end{align*}
where $\x\in  \H,\text{ }
 n\geq 1$. 
Then it holds that $\r^\ast(x\otimes y)^\ast = \r(x\otimes y)$ where the adjoint is taken with respect to $\langle \cdot, \cdot \rangle_{0,0}$ and 
$\r^\ast: \HH \to B(\F)$ is linear, but  $\r: \HH \to B(\F)$ is
anti-linear (here and throughout the paper we will use the notation
$B(X)$ for the space of bounded operators on $X$) .

\begin{definition} Let $\q$ and $\s$ be as in
  \cref{theo:PositiveDefinite'} \eqref{th:iv}. The algebraic full Fock space $\F$ equipped with the inner product $\langle\cdot,\cdot \rangle_{\q,\s}$ is called the \emph{cyclic Fock space of type B} and it is denoted by $\mathcal{F}_{\q,\s}(\HH)$. 
For $x\otimes y \in \HH$ we define $\B^\ast(x\otimes y):=
\r^\ast(x\otimes y)$ and we consider its adjoint operator $\B(x\otimes
y)$ with
respect to the inner product $\langle\cdot,\cdot \rangle_{\q,\s}$
acting on the Hilbert space $\mathcal{F}_{\q,\s}(\HH)$ (note that a
priori $\B^\ast(x\otimes y)$ might not be bounded). The operators $\B^\ast(x\otimes y)$ and $\B(x\otimes y)$ are called \emph{cyclic creation and cyclic annihilation operators of type B}.  
\end{definition} 
The following proposition can be derived directly
from  \cref{lem:Factorization}. 
\begin{proposition}\label{prop1}
We have the decomposition 
\begin{equation}\label{decomposition}
P^{(n)}_{\q,\s}=\left( \id \otimes P^{(n-1)}_{\q,\s}\otimes \id \right)R^{(n)}_{\q,\s} \text{~on $\H$}, \qquad n\geq 1, 
\end{equation}
where
\begin{equation}
R^{(n)}_{\q,\s} = \id+ \s J^-_n+
 \q J^+_n. 
\end{equation}
%where the permutation $(j,i)$ is the transposition of $i$ and $j$.
\end{proposition}

%The operator $R_{\q,\s}^{(n)}$ plays a central role in this paper. 
Now, we can compute the annihilation operator in terms of $R_{\q,\s}^{(n)}$. 
\begin{proposition}\label{prop2} For $n \geq1$, we have 
\begin{equation}
\B(x\otimes y)=\r(x\otimes y) R^{(n)}_{\q,\s} \text{~on $\H$}. 
\end{equation}
%with convention that $H^{\otimes 0}= \C \Omega$. 
%where the adjoint is taken with respect to $\langle\cdot,\cdot \rangle_{0,0}$. 
\end{proposition}

\begin{proof}
Let $f \in \HH_{n-1},g \in \H$. Then 
\begin{equation}
\begin{split}
\langle f, \B(x\otimes y)g \rangle_{\q,\s} 
&=\langle \B^\ast(x\otimes y)f, g \rangle_{\q,\s} =\langle \r^\ast(x\otimes y) f, g \rangle_{\q,\s} \\
&=\left\langle\r^\ast(x\otimes y) f,  P^{(n)}_{\q,\s} g \right\rangle_{0,0}  \\
&=\left\langle \r^\ast(x\otimes y) f,   \left(\id \otimes  P^{(n-1)}_{\q,\s}\otimes \id \right) R^{(n)}_{\q,\s}g \right\rangle_{0,0} \\
&=\left\langle  f,  \r(x\otimes y)  \left(\id \otimes P^{(n-1)}_{\q,\s}\otimes \id \right) R^{(n)}_{\q,\s}g \right\rangle_{0,0}. 
\end{split}
\end{equation}
Observe that  $\r(x\otimes y) (\id \otimes P^{(n-1)}_{\q,\s}\otimes \id) h = P^{(n-1)}_{\q,\s}\r(x\otimes y)h$ for $h\in \H$ and so we get  
\begin{equation}
\begin{split}
 \left\langle  f,  \r(x\otimes y)  (\id \otimes P^{(n-1)}_{\q,\s}\otimes \id) R^{(n)}_{\q,\s}g \right\rangle_{0,0} 
&=\left\langle  f,  P^{(n-1)}_{\q,\s} \r(x\otimes y)R^{(n)}_{\q,\s}g \right\rangle_{0,0}  \\
&=\left\langle  f, \r(x\otimes y)R^{(n)}_{\q,\s} g \right\rangle_{\q,\s}. 
\end{split}
\end{equation}
\end{proof}
In order to simplify the notation we define the operators
\begin{align*}
\J_i(x\otimes y ): \H & \to \HH_{n-1} \textrm{ for $i \in [\pm n]$ by}
\\
 \mathbf{x}_{\overline n} \otimes \mathbf{x}_n & \mapsto  %\langle  x,  x_{\overline i} \rangle \langle y, x_{i} \rangle  \mathcal{C}\big({(i,n)(\overline i,\overline n) \x}\big)
\left\{ \begin{array}{ll}
  \langle  x,  x_{\overline n} \rangle \langle y, x_{n} \rangle  \mathcal{C}({ \x})  & \textrm{ $i =n $}\\
  \langle  x,  x_{\overline i} \rangle \langle y, x_{i} \rangle
          \mathcal{C}({(i,n)(\overline i,\overline n) \x})  & \textrm{ $i \in
                                                    [\pm (n-1)]$}\\
  \langle  x,  x_{ n} \rangle \langle y, x_{\overline n } \rangle  \mathcal{C}({\x})  & \textrm{ $i=\overline n$ }
\end{array} \right.
\end{align*}
where $\mathcal{C}(x_{\overline n}\otimes \cdots  \otimes x_{n})=x_{\overline{n-1}}\otimes \cdots   \otimes x_{n-1}$ and $\mathcal{C}(x_{\overline 1}\otimes x_{1})=\Omega \otimes \Omega.$
\begin{remark}\label{ramerkdecoposition}
By using the above notation we can decompose  $\B(x\otimes y)$ as follows 
$$
\B(x\otimes y)= \a_0(x\otimes y) +\p_{\q}(x\otimes y)+\n_{\s}(x\otimes y)  %\ell_{q}(\overline{x})
, \qquad x\otimes y \in \HH, 
$$ 
where  
\begin{align}
\a_0(x\otimes y)=&\J_{n}  \label{p0}
\\
\label{pq}
\p_{\q}(x\otimes y)=& \q\sum_{i \in [\pm (n-1)]}  \J_{i} 
\\ \n_{\s}(x\otimes y)=&   \s\, \J_{\overline n}.
\label{nt}
\end{align}
\end{remark}

The second quantisation differential operator $\Gamma_{\q} (A\otimes B)$ is defined by the equation
\begin{align*}
\Gamma_{\q} (A\otimes B)x_{\overline n}\otimes\cdots \otimes x_n =\q &\sum_{1\leq i\leq n} x_{\overline n}\otimes\cdots \otimes A x_{\overline i} \otimes \dots \otimes   B x_i\otimes\cdots \otimes x_n 
\\+\q &\sum_{1 \leq i\leq n} x_{\overline n}\otimes\cdots \otimes B x_{ i} \otimes \dots \otimes   A x_{\overline i}\otimes\cdots \otimes x_n 
\end{align*}
where $A\otimes B \in B(\HH)$ and $\Gamma_{\q} \Omega \otimes\Omega=0$. 
The above considerations provide a new commutation relation.

\begin{theorem}\label{commutation}
For $x\otimes y, \xi\otimes \eta \in \HH$ we have the cyclic  commutation relation of type B 
\begin{equation}
\B(x\otimes y)\B^\ast(\xi\otimes \eta)= \langle  x,\xi \rangle\langle  y,\eta \rangle \id +\s \langle  x, \eta \rangle\langle  y, \xi \rangle \id  +   \Gamma_{\q} (|\xi\rangle\langle x| \otimes |\eta\rangle\langle y| ),% + \q\overline \Gamma_{\q} (x\otimes y)
\end{equation}
where $|\xi\rangle\langle x|(\cdot):=\langle x,\cdot\rangle \xi$. We note that $|x\rangle\langle x|$ is 
the projection on the one dimensional space spanned by $x$
and  $\Gamma_{\q} (|\xi\rangle\langle x| \otimes |\eta\rangle\langle y| ) =\p_{\q}(x\otimes y)\r^\ast(\xi\otimes \eta) $.
\end{theorem}

\subsubsection{Exclusion principle}
\label{subsec:Exclusion}

The Pauli exclusion principle is the quantum mechanical principle which states that two or more identical fermions cannot occupy the same quantum state within a quantum system simultaneously. 
We will explain that for $\q< 0$ an exclusion principle is found allowing at most $M$ identical particles on the same state $x\otimes x$, with $\|x\|=1$.  
This might have some interest also from the physics
point of view.
First observe that $\B(x\otimes x)\B^\ast(x\otimes x)\geq 0$. Thus, \cref{commutation} entails that 
\begin{align*}
  (1 +\s) \id &+\Gamma_{\q} (|x\rangle\langle x|\otimes |x\rangle\langle x|) \geq   0.
 \intertext{Computing $\langle \big((1 +\s) \id +\Gamma_{\q} (|x\rangle\langle x|\otimes |x\rangle\langle x|)\big) x^{\otimes 2n},x^{\otimes 2n}\rangle_{\q,\s} $ we obtain that}
  1 +\s &\geq   - \q 2n.
  \intertext{We recall that $\q = \frac{-1}{M+N}$ and
          $\s=\frac{M-N}{M+N}$ which gives the inequality} 
  M &\geq    n. 
\end{align*}
We can also explain this phenomenon by using \eqref{decomposition}, namely 
\begin{align*}
0 & \leq \langle x^{\otimes 2n},x^{\otimes 2n}\rangle_{\q,\s}\\&=\langle x^{\otimes 2n},P_{\q,\s}^{(n)}x^{\otimes 2n}\rangle_{0,0}
\\ & =(1+2(n-1)\q+\s)\langle x^{\otimes 2(n-1)},P_{\q,\s}^{(n-1)}x^{\otimes 2(n-1)}\rangle_{0,0}
\\ & =(1+2(n-1)\q+\s)\langle x^{\otimes 2(n-1)},x^{\otimes 2(n-1)}\rangle_{\q,\s}.
\end{align*}
This gives us the inequality $n<M+1$, and consequently, we obtain $ n\leq M$. 

%\begin{align*}
%& \Gamma_{\q}(x\otimes y)\xi\otimes x_n\otimes\cdots  \otimes x_{1}\otimes y_1 \otimes\cdots  \otimes y_n\otimes \eta \\ &=\sum_{1\leq k\leq n }\langle x, \overline x_k \rangle \langle y, y_k \rangle \, x_n\otimes\cdots  \otimes x_{k+1}  \otimes \xi \otimes x_{k-1}\otimes \cdots \otimes x_{1} \otimes  y_1\otimes\cdots  \otimes y_{k-1}  \otimes \eta \otimes y_{k+1}\otimes \cdots \otimes y_{n} 
%\\&+\sum_{1\leq k\leq n }\langle x, \overline y_k \rangle \langle y, x_k \rangle \, x_n\otimes\cdots  \otimes x_{k+1}  \otimes \eta \otimes x_{k-1}\otimes \cdots \otimes x_{1} \otimes  y_1\otimes\cdots  \otimes y_{k-1}  \otimes \xi \otimes y_{k+1}\otimes \cdots \otimes y_{n}
%\end{align*}

\subsection{Cycles on pair partitions of type B}
\label{subsec:PartitionsB}

Let $S\subseteq\N $ be a finite subset. 
For an ordered set $S$, let $\Part(S)$ denote the lattice of set partitions of that set. For a partition $\pi\in \Part(S)$,
we write $B \in \pi$ if $B$ is a class of $\pi$ and we say that $B$ is a \emph{block of $\pi$}.
 Any partition $\pi$ defines an equivalence relation on $S$,
denoted by $\sim_\pi$, such that the equivalence classes are the
blocks of $\pi$. 
That is, $i\sim_\pi j$ if $i$ and $j$ belong to the same block of $\pi$.
 A block of $\pi$ is called a \emph{singleton} if it consists of one element. %, which we denote by $\Sing(\pi)$.
Similarly, a block of $\pi$ is called a \emph{pair} if it consists of two elements.
% Let $\Sing(\pi)$ and $\Part(\pi)$  denote the set of all singleton and pair of $\pi$, receptively.   
 $\Part(n)$ is a lattice under the \emph{refinement order}, where the relation $\pi\leq \rho$
holds if
every block of $\pi$ is contained in a block of $\rho$.
%The maximal element of $\Part(n)$ under this order is the partition consisting
%of only one block and it is denoted by  $\hat{1}_{n}$.
%On the other hand, the minimal element $\hat{0}_n$  is the unique partition where every block is a singleton.
 %The set of  pair partitions of $[n]$ is denoted by $\P_2(n)$ and the set of  pairs or singletons of $[n]$ is denoted by $\P_{1,2}^{\O }(n)$. The number of block of $\pi$  will be denoted by $|{\pi}|$. 

A partition $\pi$ is called \emph{noncrossing} 
if different blocks do not interlace, i.e., there is no quadruple of
elements $i<j<k<l$ such that $i\sim_\pi k$ and $j\sim_\pi l$ 
but $i\not\sim_\pi j$.  
The set of non-crossing partitions of $S$ is denoted by $\NC(S)$. 
%,in the case where $S=[n]:=\{1, \dots , n\}$ we write $\NC(n):=\NC([n])$. 
%\input{old/example1}
%It turns out that $\NC(n)$ is in fact a lattice, 
%see \cite[Lecture~9]{NicaSpeicher:2006}.
The subclass of noncrossing partitions whose every block is either a
pair or a singleton (i.e.~noncrossing matchings)
is denoted by $\NC_{1,2}(S)$.

%In order to simplify the notation we define $[\overline{n}]$  as the set $\{\overline{n},\dots,\overline{1}\}=\{-{n},\dots,-{1}\}$. 
%We define  $[\overline{n}]$ as the set $\{\overline{n},\dots,\overline{1}\}$.
 %We associate these indices with  ordering $\overline{n}<\dots <\overline{1}$.
%Let $[n]=\{-n,\dots,-1\}$. 
%Let $\NC_{1,2}([\overline{n}]\sqcup[n])$ be a noncrossing pair partition or singletons. 
\begin{definition} \label{def:partycji} 
We denote by $\P_{1,2}^{sym}(n)$ the subset of partitions $\pi \in \Part([\pm n])$ whose every block is either a
pair or a singleton and such that they are symmetric $\overline{\pi} =
\pi$, but every pair $B \in \pi$ is different than its symmetrization
$\overline{B}$, i.e..~$B \neq \overline{B}$. We will order
elements in pairs $\{a,b\}$ of $\pi \in \P_{1,2}^{sym}(n)$ by writing $(a,b)$ which means that $a <
b$ and we call $a$ ($b$ respectively) the \emph{left} (\emph{right},
respectively) leg of $(a,b)$. A pair $(a,b)$ is called \emph{positive} if $b> \overline a$; otherwise
it is called \emph{negative}. The subset of  pair partitions of $\P_{1,2}^{sym}(n)$ is denoted by $\P_2^{sym}(n)$. 
% The set positive and negative  pairs of $\pi$  we denote by $\Pos(\pi) $ and $\Neg (\pi )$, respectively. 
\end{definition}
% \begin{remark}

% \end{remark}

\begin{proposition}
  \label{prop:pihat}
Let   $\pi\in \P_{1,2}^{sym}(n)$. There exists a unique non-crossing  partition    
$\hat\pi\in \NC_{1,2}([{\pm n}])\cap\P_{1,2}(n)$,  such that 
\begin{enumerate}[label=(\alph*)]
\item the set of right legs of the positive pairs of $\pi$ and $\hat\pi$ coincide;
\item the set of left legs of  the negative pairs of $\pi$ and $\hat\pi$ coincide;
\item pairs of $\hat\pi$ do not cover singletons. 
\end{enumerate}
\end{proposition}
\begin{proof}
  Notice that if $\sigma \in \NC_{1,2}([{\pm
    n}])\cap\P_{1,2}(n)$ then each block of $\sigma$ is either
  contained in $\{1,\dots,n\}$ or in
  $\{\overline{n},\dots,\overline{1}\}$. In particular $\sigma$ is
  completely determined by its restriction to
  $\{\overline{n},\dots,\overline{1}\}$. Take a non-crossing partition
  $\pi$ on
  $\{\overline{n},\dots,\overline{1}\}$ whose blocks are either pairs
  or singletons and associate it with the Motzkin path, whose left
  legs correspond to up steps, right legs correspond to down steps and
  singletons correspond to horizontal steps. We recall that a Motzkin
  path of length $n$ is a path starting at $(0,0)$ and finishing in
  $(n,0)$ which never goes below the horizontal and consists of three
  steps: the up step $(1,1)$, the down step $(1,-1)$ and the
  horizontal step $(1,0)$. Notice that pairs do not
  cover singletons in this partition if and only if all the horizontal
  steps in the associated Motzkin path lie on the horizontal axis. Pick a partition $\pi\in
  \P_{1,2}^{sym}(n)$ and associate with it a path with $n$ steps by
  considering the set $\{\overline n,\dots, \overline 1\}$ and placing
  an up step in the place of left legs of  the negative pairs of $\pi$
  and down steps in the other points of $\{\overline n,\dots, \overline
  1\}$. Notice that there is a unique way for replacing some of the
  down steps by horizontal steps so that the resulting path is a
  Motzkin path with all the horizontal
  steps placed on the horizontal axis. Indeed, we successively change
  a down step into a horizontal step whenever we go below the
  horizontal axis. We consider the associated non-crossing partition of $\{\overline n,\dots, \overline 1\}$, and
  by using the symmetry it uniquely determines a partition $\hat\pi\in
  \NC_{1,2}([{\pm n}])\cap\P_{1,2}(n)$ which has all the desired
  properties. This finishes the proof.
\end{proof}

  \begin{definition}
\label{defipihat}
%For pair $A$, $B$ of $\pi$, we say that $A$ is on the left of $B$ or $B$ is on the right of $A$ if $\min A < \min B$.
For the pairs $A$, $B$ of $\pi$, we say that $A$ is \emph{connected} with
$B$  if there exists a pair $C\in \hat \pi $ such that 
$$ A\cap C \neq \emptyset \text{ and }B\cap C \neq \emptyset .$$
We denote this equivalence relation by writing $A \sim B$ see \cref{fig:coneced}.
\end{definition}

\begin{remark}In the graphical presentation of the \cref{defipihat} we
  will usually draw  arcs of $\pi$ above the points and   $\hat \pi$ below  the points. 
From this definition, it also follows that $A \sim B\iff \overline A \sim \overline B$, which shortens the notation in the next definition. 
\end{remark}

\begin{figure}[h]
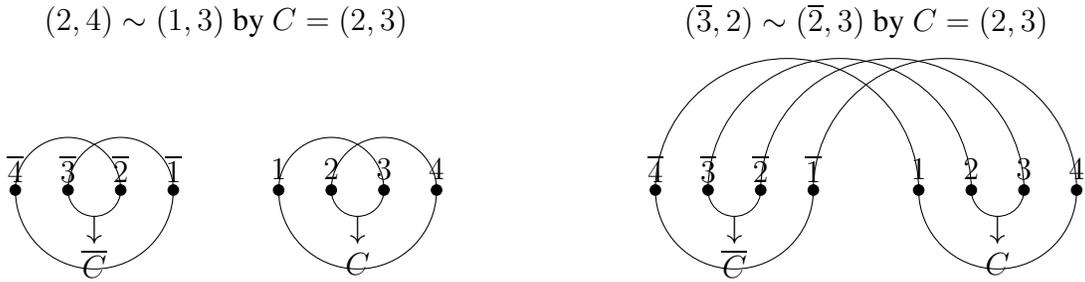

\begin{center}
  \MatchingMeandersab{4}{-4/-2, -3/-1,1/3,2/4 }{-4/-1,-3/-2,1/4,2/3} \hspace{5em} \MatchingMeandersabc{4}{-4/1, -3/2,-2/3,-1/4 }{-4/-1,-3/-2,2/3,1/4}
%$$\Cycle(\pi)=\{(-1,3,2,-4)^+,(7,9,8,10)^+,(-5,6)^-\}$$   
   \end{center}  
\caption{Two examples of connected pairs. The bottom non-crossing
  partitions are the associated $\hat\pi$.}
\label{fig:coneced}
\end{figure}

%\todo[inline]{MD: zrobiłem na dole pełną partycję nieprzecinającą, bo
 % moim zdaniem ten rysunek był mylący - czytając komentarz że zawsze
 % na dole jest $\hat \pi$ miało się wrażenie, że $\hat\pi$ ma
 % singletony $1,4$.}

A \emph{cycle} in $\pi\in \P_{1,2}^{sym}(n)$ is a sequence of pairwise
distinct pairs of $\pi$ cyclically connected by $\hat\pi$, i.e.:
$$(l_1, r_1)\sim (l_2, r_2) \sim \dots \sim (l_{m-1} , r_{m-1})\sim ( l_m , r_m)\sim (l_1, r_1).$$
Note that due to the symmetric nature of the considered partitions we
distinguish two fundamentally different kinds of cycles:
\emph{positive} and \emph{negative}, which resemble the description
of the cycles in the hyperoctahedral group. 

\begin{definition}\label{deficykli}
Let $\pi$ and  
$\hat \pi$ be as in \cref{prop:pihat}. Suppose that
\[ (l_1, r_1)\sim (l_2, r_2) \sim \dots \sim (l_{m-1} , r_{m-1})\sim ( l_m , r_m)\sim (l_1, r_1)
\label{eq:cyclecondition} \tag{$\star$}
\]
is a cycle. There are two possibilities: either
\[ \overline{\{l_1,r_1\dots,l_m,r_m\}} \cap \{l_1,r_1\dots,l_m,r_m\}
  = \emptyset,\]
or
\[ \overline{\{l_1,r_1\dots,l_m,r_m\}} = \{l_1,r_1\dots,l_m,r_m\}.\]
In the first case, there exists the associated cycle
\[ (\overline{r_m},\overline{l_m})\sim (\overline{r_{m-1}},\overline{l_{m-1}}) \sim \dots \sim (\overline{r_1},\overline{l_1}) \sim (\overline{r_m},\overline{l_m})
\]
and we call the pair of cycles
$$\sigma = (l_1, r_1,\dots, l_m , r_m)(\overline{r_m},\overline{l_m},
\dots, \overline{r_1}, \overline{l_1})$$
a\emph{ positive cycle in
$\pi$} of length $m$. We will use the notation $\sigma=(l_1, r_1,\dots, l_m ,
r_m)^+$ (or $\sigma=(\overline  r_m,\overline l_m, \dots, \overline r_1, \overline
l_1)^+$) to indicate positivity. In the second case the parameter $m=2m'$
is necessarily even and due to symmetry this cycle has the
following form
\[ (l_1, r_1) \sim \dots \sim (l_{m'} ,
  r_{m'}) \sim ( \overline{r_1},\overline{l_1}) \sim \dots \sim ( \overline{r_{m'}},\overline{l_{m'}}) \sim  (l_1, r_1).
\label{eq:cyclecondition2} \tag{$\star\star$}
\] 
We call it a
\emph{negative cycle in
$\pi$ } of length $m'$ and we denote it $\sigma = (l_1, r_1,\dots,
l_{m'} , r_{m'})^-$. 
For a cycle $\sigma$ of $\pi$ (positive or negative) we denote its length by
$|\sigma|$ and by $\Cycle(\pi)$ the set of cycles of $\pi$.
$$\Lcycle(\pi)=\sum_{\sigma\in \Cycle(\pi)}(|\sigma|-1)$$
is the total length of  cycles of $\pi$ reduced by $1$.
Let $\SLNB(\pi)$ be the number of  \emph{ negative cycles of $\pi$}, see \cref{fig:examplesofcycle1}.

\begin{figure}[h]
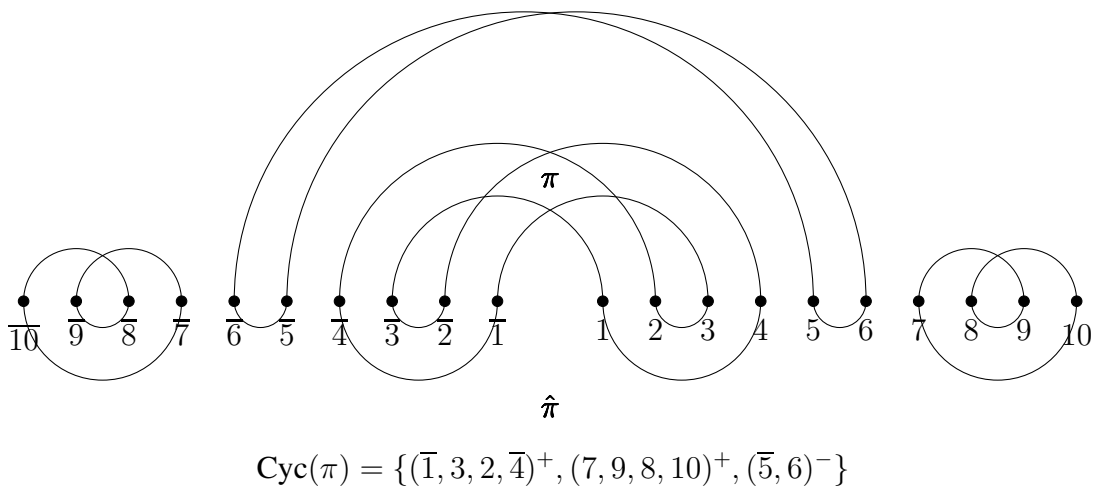

\begin{center}
  \MatchingMeanders{10}{-1/3, -3/1, -2/4 , -4/2 ,-6/5,-5/6,8/10,7/9,-9/-7,-10/-8 }{2/3,-3/-2, 1/4,-4/-1, 5/6,-6/-5, 8/9, 7/10, -9/-8, -10/-7 }
\begin{align*}
\Cycle(\pi)&=\{(\overline 1,3,2, \overline 4)^+,(7,9,8,10)^+,(\overline 5,6)^-\}  
\end{align*} 
   \end{center}  
\caption{An example of a partition $\pi\in \P_{2}^{sym}(10)$ with one
  negative and two
  positive cycles.}
\label{fig:examplesofcycle1}
\end{figure}

\end{definition}

\begin{remark}
1). In the case of a positive cycle  $\sigma = (l_1, r_1)(
\overline{r_1}, \overline{l_1})$ it should be understood that the pairs
$(l_1, r_1),( \overline{r_1}, \overline{l_1})$ lie in both partitions
$\pi$ and $\hat \pi.$

\noindent  2). Our definition of cycles in $\P_{2}^{sym}(n)$  is quite
similar to the definition of partitions of type B \cite[Section 2]{Reiner1997}, but in
our situation a \emph{zero block} (which is invariant under the bar operation) does not necessarily exist. 
\end{remark}

If some pairs are not connected cyclically, we call them \emph{semi-cycles}.

\begin{definition}
Let $\pi$ and  
$\hat \pi$ be as in \cref{defipihat}. Consider the set of pairs in
$\pi$, which are not connected cyclically. This set is partitioned
into chains of connected pairs:
\[ (l_1, r_1) \sim \dots \sim ( l_m , r_m) \not\sim ( l_1,r_1)\]
where $(l_1,r_1)$ is a negative pair, and by symmetry
\[ \overline  (r_1,\overline  l_1) \sim \dots \sim ( \overline  r_m,
  \overline l_m)\not \sim  (\overline r_1, \overline l_1).\]
We say that 
$$\sigma=\sigma^-\sigma^+ =\underbrace{(\overline c_m ,\overline r_m, \overline
  l_m,\dots, \overline r_{1} , \overline
  l_{1})^-}_{\sigma^-}\underbrace{(l_1, r_1,\dots, l_{m} ,
  r_{m},c_m)^+}_{\sigma^+}$$
is a \emph{semi-cycle} of length $m+1$
in $\pi$, where $(r_m, c_m),(\overline c_m, \overline r_m)\in \hat
\pi$ (they exist by \cref{prop:pihat}). Additionally, there are also semi-cycles $\sigma=(\overline k)^-(k)^+$ in $\pi$ of length
$1$ formed by singletons of $\pi$ (with the convention $k \in \{1,\dots,n\}$). 
Similarly as for the cycles we denote the length of a semi-cycle
$\sigma$ by $|\sigma|$, the set of semi-cycles by $\Scycle(\pi)$ and
we introduce the total length $\LScycle(\pi)$:
 $$\LScycle(\pi)=\sum_{\sigma \in \Scycle(\pi)}(|\sigma|-1).$$ 
%We will refer to the $\sigma^+$ and $\sigma^-$ as positive and
%negative semi-cycle and we denote it by    $\Scycle^+(\pi)$ and
%$\Scycle^-(\pi)$ respectively.
See \cref{fig:examplesofcycle2} for an example of semi-cycles. The
points  $c_m$, $l_1$ $(\textrm{resp. } \overline c_m,\text{ }
\overline l_1)$ are called the \emph{edges} of the positive (and the
negative, respectively) parts of a semi-cycle $\sigma$ and they are denoted by 
$$
\ends (\sigma^f)= \left\{ \begin{array}{ll}
c_m & \textrm{if $f=+$}\\
\overline c_m & \textrm{if $f=-$}.
%\\
%l & \textrm{gdy $ d < c $}
\end{array} \right.\,
\quad 
\maks (\sigma^f)= \left\{ \begin{array}{ll}
l_1 & \textrm{if $f=+$}\\
\overline l_{1} & \textrm{if $f=-$}.
%\\
%l & \textrm{gdy $ d < c $}
\end{array} \right.
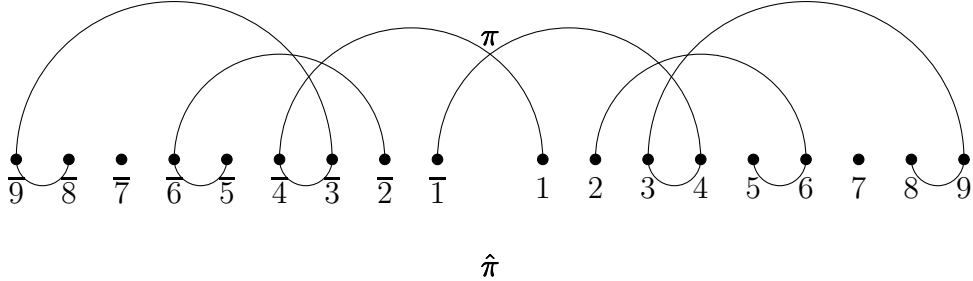
$$

\begin{figure}[h] 
\begin{center}
\MatchingMeanders{9}{-6/-2,2/6,-9/-3,3/9,-4/1, -1/4 }{-4/-3, 3/4, 8/9, -9/-8,-6/-5,5/6}
\begin{align*}
  \Scycle(\pi)&=\{(\overline 8,\overline 9,\overline 3,\overline 4,1)^-(\overline
              1,4,3,9,8)^+,(\overline 5,\overline 6,\overline 2)^-(2,6,5)^+,(\overline 7)^-(7)^+\}
\end{align*} 
\end{center} 
\caption{The partition $\pi$ has one semi-cycle $\sigma_1$ of length
  $3$ with the edges of the negative part $\ends(\sigma_1^-)=\overline
  8$, $\maks(\sigma_1^-)=1$, one semi-cycle $\sigma_2$ of length
  $2$ with the edges of the negative part $\ends(\sigma_2^-)=\overline
  5$, $\maks(\sigma_1^-)=\overline 2$, and one semi-cycle $\sigma_3$ of length
  $1$ with only one edge of the negative part
  $\maks(\sigma_3^-)=\overline 7$.}
\label{fig:examplesofcycle2}
\end{figure} 
\end{definition}
%\begin{remark}
%We will take the convention that the  semi-circles of  length  one i.e.  $(a_{1})$ are equated with singletons see Fig. \ref{fig:examplesofcycle2}.
%\end{remark}

\subsection{Gaussian operator}
\label{subsec:Gaussian}
We construct generalized cyclic Gaussian operators of type B. They are given by
the creation and annihilation operators on a cyclic Fock space of type B.
We show that the distribution of these operators with
respect to the vacuum expectation is a generalized Gaussian distribution, in the
sense that all the moments can be calculated from the second moment by the explicit combinatorial formula. 
The operator $$\G(x)=\B(x\otimes x)+\B^\ast(x\otimes x),\quad x \in \HH_\R, $$
is called the \emph{cyclic Gaussian operator of type B}. Given
 $\epsilon=(\epsilon(1), \dots, \epsilon(n))\in\{1,\ast\}^{n}$, let
 $\P_{1,2;\epsilon}^{sym}(n)$ be the set of partitions $\pi\in
 \P_{1,2}^{sym}(n)$ such that
 \begin{itemize}
   \item if $(a,b)$ is a negative pair in
     $\pi$ then $\epsilon(|b|)= \ast,$ $\epsilon(|a|)=1$,
     \item if $\{c\}$ is a singleton in $\pi$ then $\epsilon( |c|)=\ast$. 
   \end{itemize}

\begin{theorem}
Let $x_{\overline i}\otimes  x_i\in \HH_\R$ for $1 \leq i \leq n$. Then 
\begin{align}\label{formula101}
\begin{split}
\B^{\epsilon(n)}(x_{\overline n}\otimes x_n)\dots \B^{\epsilon(1)}( x_{\overline 1}\otimes x_1)\Omega\otimes \Omega  =&\sum_{\pi\in \P_{1,2;\epsilon}^{sym}(n)}\s^{\SLNB(\pi)}\q^{\Lcycle(\pi)+\LScycle(\pi)} \prod_{\substack{\{i,j\} \in \Pair(\pi)} }\langle x_i, x_j\rangle
% \prod_{\substack{\{i,j\} \in\Neg( \pi)}} \langle x_i, \overline{x}_j\rangle
\\& \bigotimes_{\substack{   \sigma \in \Scycle(\pi)  }}\Bigg\{ x_{\ends(\sigma^\pm)}\Bigg\}_{\maks(\sigma^\pm)}.
% \bigotimes_{\substack{   \sigma \in \Scycle^-(\pi) \cup \Scycle^+(\pi) }}\Bigg\{ x_{\ends(\sigma)}\Bigg\}_{\maks(\sigma)}.
\end{split}
\end{align}
\end{theorem}
\begin{remark}
%\item Before proving we should clarify the notation form above theorem.  Given $\epsilon=(\epsilon(1), \dots, \epsilon(n))\in\{1,\ast\}^n$, let $\P_{1,2;\epsilon}^{sym}(n)$ be the set of partitions $\pi\in \P_{1,2}^{sym}(n)$ such that  when $(\overline b, \overline a),(a,b)\in \pi$ and $(a,b)$ is positive then $\epsilon(|a|)= \ast$,  $\epsilon(|b|)=1$
%or $(\overline c),(c)\in \pi$ then $\epsilon(|c|)= \ast$. 
%$$
%\pi =\{\{a_1,b_1\}, \dots, \{a_k,b_k\}, \{c_1\}, \dots, \{c_m\}\},\qquad k,m \in \N \cup\{0\}, a_i<b_i, i\in[k], 
%$$
%then $\epsilon(a_i)= \ast$ and $\epsilon(b_i)=1$ for all $1 \leq i \leq k$ and $\epsilon(c_i)=\ast$ for all $1 \leq i \leq m$. We also let %5\P_{2;\epsilon}(n):=\P_{1,2;\epsilon}(n)\cap \P_2(n)$. Let $\PB_{1,2;\epsilon}(n)$ be the subset of $\PB_{1,2}(n)$ defined by $(\pi,f)\in %\PB_{1,2;\epsilon}(n) \Leftrightarrow \pi \in \P_{1,2;\epsilon}(n)$. Let  $\PB_{2;\epsilon}(n):= \PB_2(n)\cap\PB_{1,2; \epsilon}(n)$.  %We define $\PB_{1,2;\epsilon}(n), \P_{2;\epsilon}(n), \PB_{2;\epsilon}(n)$ similarly. 
1).  In the above formula we use the following bracket notation
$\{\circ\}_{\star}$. It should be understood that the position of
 $\circ $ (in the tensor product) is ordered with respect to the $\star$.
For
example, for the partition as in \cref{fig:examplesofcycle2}, we have
$$\bigotimes_{\substack{   \sigma \in \Scycle(\pi)  }}\Bigg\{ x_{\ends(\sigma^\pm)}\Bigg\}_{\maks(\sigma^\pm)}=x_{\overline 7 }\otimes x_{\overline 5 } \otimes x_{8 }\otimes  x_{\overline 8 }\otimes x_{5}\otimes x_{ 7 }.$$

2).
  We denote by $(A,B)$
the concatenation of the semi-cycles $A$ and  $B$. 
\end{remark}
\begin{proof}
The proof is by induction. When $n=1$,  $\B(x_{\overline
  1}\otimes x_1)\Omega\otimes\Omega =0$ and $\B^\ast(x_{\overline
  1}\otimes x_1)\Omega\otimes\Omega=x_{\overline 1}\otimes
x_1$. Suppose that the formula  \eqref{formula101} is true for $n=k-1$
and for any $\epsilon\in\{1,\ast\}^{k-1}$. 
%Then for any $\epsilon\in\{1,\ast\}^{k-1}$, we get  
%\begin{align*}
%\B^{\epsilon(k-1)}(x_{\overline{k-1}}\otimes x_{k-1})\cdots  \B^{\epsilon(1)}(x_{\overline 1}\otimes x_1)\Omega\otimes \Omega  =&\sum_{\pi\in %\P_{1,2;\epsilon}^{sym}(k-1)}\s^{\SLNB(\pi)}\q^{\Lcycle(\pi)+\LScycle(\pi)} \prod_{\substack{\{i,j\} \in \Pair(\pi)} }\langle x_i, x_j\rangle %\prod_{\substack{\{i,j\} \in\Neg( \pi)}} \langle x_i, \overline{x}_j\rangle
%\\& \bigotimes_{\substack{   \sigma \in \Scycle^-(\pi) \cup \Scycle^+(\pi) }}\Bigg\{ x_{\ends(\sigma)}\Bigg\}_{\maks(\sigma)}.
%\end{align*}
 %The operator $\B^{\epsilon(k)}(x_{\overline{ k}}\otimes x_{k})$ equals $\r^\ast(x_{\overline{ k}}\otimes x_{k})$ if $\epsilon(k)=\ast$ and $\B(x_{\overline{ k}}\otimes x_{k})$ if $\epsilon(k)=1$.
  We will show that the action of $\B^{\epsilon(k)}(x_{\overline{ k}}\otimes x_{k})$ corresponds to the inductive pictorial description of set partitions, cycles and semi-cycles.  

From now on, we fix a partition $\pi\in \P_{1,2;\epsilon}^{sym}(k-1)$
(then the associated non-crossing $\hat \pi $  is automatically defined) and suppose that $\pi$ has 
the semi-cycles $\sigma^-_1\sigma^+_1,\dots,\sigma^-_i\sigma^+_i ,\dots, \sigma^-_p\sigma^+_p$,  with lengths $d_1,\dots  ,d_p$ and 
we assume that $\sigma^-_p\sigma^+_p$ is the most left-right semi-cycle i.e. $|\maks(\sigma^-_i)|=|\maks(\Ci^+_i)|< |\maks(\Ci^-_p)|=|\maks(\Ci^+_p)|$ for all $1 \leq i \leq p-1$.
In this situation, the  semi-cycles contribute to the tensor product 
\begin{align} \label{eq:tensorproducta}
& \{x_{\start(\Ci^-_p)}\}_{\maks(\Ci^-_p)}\otimes \dots \otimes \{x_{\start(\Ci^f_i)}\}_{\maks(\Ci^f_i)} \otimes   \dots \otimes \{x_{\start(\Ci^g_i)}\}_{\maks(\Ci^g_i)}\otimes \dots  \otimes\{ x_{\ends(\Ci^+_p)}\}_{\maks(\Ci^+_p)}, 
\\ \text{or} \nonumber
  \\ & \{x_{\start(\Ci^+_p)}\}_{\maks(\Ci^+_p)}\otimes \dots \otimes \{x_{\start(\Ci^f_i)}\}_{\maks(\Ci^f_i)} \otimes   \dots \otimes \{x_{\start(\Ci^g_i)}\}_{\maks(\Ci^g_i)}\otimes \dots  \otimes\{ x_{\ends(\Ci^-_p)}\}_{\maks(\Ci^-_p)}. 
   \label{eq:tensorproductb}
 \end{align}
 where $f=\pm $ and  $ g=-f$. 
For the
reader it is convenient to visualize equations
\eqref{eq:tensorproducta} and  \eqref{eq:tensorproductb} as the
following diagrams 
$$ \stackrel{\MatchingProofequation{6}{4/6,-6/-4,1/3, -3/-1 }{5/6,-6/-5,2/3,-3/-2}}{\eqref{eq:tensorproducta},\text{ } f=-}
\qquad
\stackrel{ \MatchingProofequation{6}{4/6,-6/-4,-3/2, -2/3 }{5/6,-6/-5,1/2,-2/-1}} {\eqref{eq:tensorproducta},\text{ }  f=+}\qquad
\stackrel{\MatchingProofequation{6}{-4/6,-6/4,1/3, -3/-1 }{5/6,-6/-5,2/3,-3/-2}}{\eqref{eq:tensorproductb},\text{ }  f=-} 
\qquad
\stackrel{\MatchingProofequation{6}{-4/6,-6/4,-3/2, -2/3 }{5/6,-6/-5,1/2,-2/-1}}{\eqref{eq:tensorproductb},\text{ }  f=+} .$$ 
 %Graphically,    can be represented respectively as $\MatchingProofequation{6}{4/6,-6/-4,1/3, -3/-1 }{5/6,-6/-5,2/3,-3/-2} \MatchingProofequation{6}{-4/6,-6/4,1/3, -3/-1 }{5/6,-6/-5,2/3,-3/-2}$. 
We understand that $p=0$ when there is no semi-cycle. 
We will show that the action of $\B^{\epsilon(k)}(x_{\overline{ k}}\otimes x_{k})$ corresponds to the inductive graphic description of set partitions and corresponding cycles semi-cycles. During this step, we create the new partition $\tyldapi \in \P^{sym}_{1,2;\epsilon}(k)$ and $ \hat \tyldapi \in \NC_{1,2;\epsilon}([\pm k] )$. % as
%\begin{align*}
%\begin{aligned}
%\pi & \xrightarrow{\B^{\epsilon(k)}(x_{\overline{ k}}\otimes x_{k})} \tyldapi  
%\\  \hat{\pi }& \xrightarrow{\B^{\epsilon(k)}(x_{\overline{ k}}\otimes x_{k})} \hat \tyldapi 
%\end{aligned}
%\end{align*}
Note that when there are no  semi-cycles, the arguments below can be modified easily.

Case 1. If $\epsilon(k)=\ast$, then the operator $\r^\ast(x_{\overline{ k}}\otimes x_{k})$ acts on the tensor product \eqref{eq:tensorproducta} or \eqref{eq:tensorproductb}, putting $x_{k}$  on the right and $ x_{\overline{k}}$ on the left. This operation pictorially corresponds to adding the semi-cycle $(\overline{k})^-(k)^+$  to $\pi\in\P^{sym}_{1,2;\epsilon}(k-1)$, with length one.  This map $\pi\mapsto  \tyldapi$ does not change the numbers $\SLNB, $ $\Lcycle$ or $\LScycle$, which is compatible with the fact that the action of $\r^\ast(x_{\overline{ k}}\otimes x_{k})$ does not change the coefficient. 
%Note that if $\epsilon(k)=\ast$, then any $\sigma_g \in \P^{sym}_{1,2;\epsilon}(k)$ has the most right semi-cycle $\{k\}$.
 %Hence the formula \eqref{formula101} holds when $n=k$ and $\epsilon(k)=\ast$. 

Case 2. Now we apply the operator  $\r(x_{\overline{ k}}\otimes x_{k})$ on
\eqref{eq:tensorproducta} and \eqref{eq:tensorproductb} with
$f=\pm$. In the figures below we will illustrate our proof with
$f=-$  because the case $f=+$  does not introduce any new difficulty
and is fully analogous.

Consider the action of $\a_0(x_{\overline{ k}}\otimes x_{k})$
and  $\n_{\s}(x_{\overline{ k}}\otimes x_{k})$. There are four cases
 during which we create cycles. 
  %according to the decomposition form Remark \ref{ramerkdecoposition}.  
%$\p_0(x) +\p_q(x)+  \n_t(x) +\n_q(x)$

Case 2a). Suppose that $\a_0(x_{\overline{ k}}\otimes x_{k})$ acts on the semi-cycle appearing at the edges vectors of the tensor product \eqref{eq:tensorproducta} by using \eqref{p0}, which gives us   the inner product $\langle x_{k},x_{\ends(\Ci^+_p)}\rangle\langle  x_{\overline{k}},x_{\start(\Ci^-_p)}\rangle  $.  Pictorially this corresponds to getting a set partition $\tyldapi \in \P^{sym}_{1,2;\epsilon}(k)$  by adding $\overline{k}$ and $k$ to $\pi$ and
adding the pairs 
$$(\ends(\Ci^+_p), k),(\overline{k},\start(\Ci^-_p))\text{ to }\pi.$$
This results in adding the pairs
$$(\maks(\Ci^+_p), k),(\overline{k},\maks(\Ci^-_p))\text{ to }\hat
{\pi}$$
and creating the positive cycle $(\Ci^+_p,k)^+$.
Note that the above pairs create a closed cycle -- see \cref{fig:proofpositivecycle} $(a)$. 
  We also remove the semi-cycle $\Ci^-_p\Ci^+_p.$ 
 This new cycle has length $d_p$ and so increases the $\Lcycle(\tyldapi)$  by $d_p-1$ and decreases the  length of semi-cycle by $d_p$ because originally $\Ci^-_p\Ci^+_p$ was the semi-cycle of $\pi$. 
% Now the new inner singletons $\{m_1\}, \dots, \{m_r\}$ and new inner negative blocks $V_1,\dots, V_t$ appear. Because $i$ is not a singleton in $(\tyldapi, \tilde{f})$, the number of singletons left to negative blocks decreases by $t$. 
Altogether we have: $\Lcycle(\tyldapi)=\Lcycle(\pi)+d_p-1$,
$\LScycle(\tyldapi)=\LScycle(\pi)-d_p+1$ and  $\SLNB(\tyldapi)=
\SLNB(\pi)$. So the exponent of $\q$ and $\s$ do not change, which is
compatible with  the action \eqref{p0}  on the tensor \eqref{eq:tensorproducta}.

\begin{figure}[h]
\begin{center}
\MatchingProof{3}{1/3, -3/-1 }{2/3,-3/-2}  \MatchingProoff{4}{1/3, -3/-1, -4/-2,2/4 }{-4/-1, 2/3,-3/-2,1/4} 
\end{center}
\begin{center}
\MatchingProofd{3}{-1/3, -3/1 }{2/3,-3/-2}  \MatchingProoffd{4}{-1/3, -3/1, -2/4,-4/2 }{-4/-1, 2/3,-3/-2,1/4} 
\end{center}
\caption{The visualization of the action of $\a_0(x_{\overline k} \otimes x_{k})$}
\label{fig:proofpositivecycle}
\end{figure}

Case 2b).  In the next step, we have the analogous situation, with the case 2a) because then $\a_0(x_{\overline{ k}}\otimes x_{k})$ acts on the tensor product \eqref{eq:tensorproductb}, which gives us the\sloppy inner product $\langle x_{k},x_{\ends(\Ci^-_p)}\rangle\langle  x_{\overline{k}},x_{\start(\Ci^+_p)}\rangle  $ and we proceed as shown in the \cref{fig:proofpositivecycle} $(b)$ by creating the positive cycle $(\overline{k}, \Ci^+_p)^+$.

Case 2c). Suppose that $\n_{\s}(x_{\overline{ k}}\otimes x_{k})$ acts on the tensor product \eqref{eq:tensorproducta} by using \eqref{nt}, which gives us   the inner product $\langle x_{k}, x_{\start(\Ci^-_p)}\rangle \langle  x_{\overline{k}}, x_{\ends(\Ci^+_p)}\rangle $ with coefficient $\s$. We  create the negative cycle $(\overline{k},\Ci^+_p)^-$ by adding  the pairs
$$(\start (\Ci^-_p), k),(\overline{k},\ends(\Ci^+_p))\text{ to }\pi\text{  and }(\maks(\Ci^+_p), k),(\overline{k},\maks(\Ci^-_p))\text{ to }\hat {\pi},$$ 
which create a closed cycle -- see  \cref{fig:proofnegativecycle} (a).
Next, we count the change of the exponent of $\q$ similarly as in Case 2a,  and get 
 $\Lcycle(\tyldapi)=\Lcycle(\pi)+d_p-1$,
 $\LScycle(\tyldapi)=\LScycle(\pi)-d_p+1$. Moreover, the new negative cycle appears so that $\SLNB(\tyldapi)= \SLNB(\pi)+1$. Thus the exponent of $\s$  increases by $1$ which agrees with the action \eqref{nt}. 

\begin{figure}[h]
\begin{center}
\MatchingProof{3}{1/3, -3/-1 }{2/3,-3/-2}  \MatchingProofff{4}{1/3, -3/-1, -4/2,-2/4 }{-4/-1, 2/3,-3/-2,1/4} 
\end{center}
\begin{center}
\MatchingProofd{3}{-1/3, -3/1 }{2/3,-3/-2}  \MatchingProoffdd{4}{-1/3, -3/1, -4/-2,2/4 }{-4/-1, 2/3,-3/-2,1/4} 
\end{center}
\caption{The visualization of the action of $\n_{\s}(x_{k})$}
\label{fig:proofnegativecycle}
\end{figure}

Case 2d). Suppose that $\n_{\s}(x_{\overline{ k}}\otimes x_{k})$ acts on the tensor
product \eqref{eq:tensorproductb}, which gives us    $\langle
x_{k},x_{\ends(\Ci^+_p)}\rangle\langle
x_{\overline{k}},x_{\start(\Ci^-_p)}\rangle  $ with coefficient
$\s$. This situation is fully analogous to the Case 2c) and we proceed
as shown in \cref{fig:proofnegativecycle} $(b)$ by creating the
negative cycle $ (\overline k,\Ci^-_p)^-$.

\vspace{10pt}

Case 3. Finally we consider the action of $\p_{\q}(x_{\overline{
    k}}\otimes x_{k})$, which creates semi-cycles.
In this case we have two situations but similarly as before, they are
very similar to each other and we describe in details only the first one.

Case 3a).
If $\p_{\q}(x_{\overline{ k}}\otimes x_{k})$ acts on the tensor
product \eqref{eq:tensorproducta} then there are $2(p-1)$ new terms
which appear in \eqref{pq}. %In the $i^{\rm th}$  term the inner product appears
%Let $J_j = \langle  x,  x_i \rangle \langle y, y_i \rangle  \J_{i} ({\bf \eta }_{n})$, then 
The action on the $i^{\rm th}$  term in the tensor product \eqref{eq:tensorproducta}, for $1\leq i\leq p-1$ contributes to the
 \begin{enumerate}[label=(\alph*)]
\item inner product   $\langle x_{k},x_{\ends(\Ci^+_i)}\rangle\langle
  x_{\overline{k}},x_{\ends(\Ci^-_i)}\rangle$ with coefficient ${\q}$
  if we apply the action of $\q \J_i$;
\item inner product  $\langle x_{k},x_{\ends(\Ci^-_i)}\rangle\langle
  x_{\overline{k}},x_{\ends(\Ci^+_i)}\rangle$ with coefficient $\q$ if
  we apply the action of $\q \J_{\overline i}$; 
\end{enumerate}
 Pictorially this corresponds to getting a set partition $\tyldapi \in \P^{sym}_{1,2;\epsilon}(k)$  by adding $\overline{k}$ and $k$  to $\pi$ and creating the pairs 
 \begin{enumerate}[label=(\alph*)]
\item $(\ends(\Ci_i^+), k)$ and  $(\overline{k},\ends(\Ci_i^-))$ in $\tyldapi$;  see \cref{fig:proofpositivesemicycle} $(a)$;
\item $(\ends(\Ci_i^-), k)$ and  $(\overline{k},\ends(\Ci_i^+))$ in $\tyldapi$; see \cref{fig:proofpositivesemicycle} $(b)$.
\end{enumerate} 
In the above situations we also add the pairs $(\maks(\Ci_p^+), k)$ and $(\overline{k},\maks(\Ci_p^-))$ to $\hat {\pi}$. 
 These new pairs create a semi-cycle  by connecting $\Ci_i^{\pm}$ and $\Ci_p^{\pm}$ through $k$ or $\overline k$  as in  \cref{fig:proofpositivesemicycle} $(a)$, $(b)$. More precisely the new semi-cycle is $\tilde{\Ci}^-\tilde{\Ci}^+$, where  
 \begin{enumerate}[label=(\alph*)]
\item $\tilde{\Ci}^-=(\Ci_p^-,\overline{k} ,\Ci^-_i)$  and $\tilde{\Ci}^+=(\Ci_i^+, k ,\Ci^+_p)$ with  $\maks(\tilde {\Ci}^\pm)= \maks(\Ci^\pm_i)$ and $\start(\tilde{\Ci}^\pm)=\start(\Ci^\pm_p)$;
\item $\tilde{\Ci}^-=(\Ci_p^-,\overline{k} ,\Ci^+_i)$  and $\tilde{\Ci}^+=(\Ci_i^-, k ,\Ci^+_p)$ with  $\maks(\tilde {\Ci}^\pm)= \maks(\Ci^\mp_i)$ and $\start(\tilde{\Ci}^\pm)=\start(\Ci^\pm_p)$;
\end{enumerate}
We can analyze the change in the statistic generated by the new semi-cycle
analogously to the previous cases and we see that the length of the new semi-cycle is $
\tilde{d}_i=
{d}_i + d_p$, which gives $\LScycle(\tyldapi)=\LScycle(\pi)+1$. This
is because we remove semi-cycles $\Ci^-_i\Ci^+_i$ and $\Ci_p^-
\Ci_p^+$ from $\pi$, which have a contribution $
d_i+d_p-2$ to $\LScycle(\pi)$, and replace them by one semi-cycle
which has contribution $d_i+d_p-1$ to $\LScycle(\tyldapi)$. The
other statistics remain unchanged  $\Lcycle(\tyldapi)=\Lcycle(\pi)$,
$\SLNB(\tyldapi)= \SLNB(\pi)$, so the total change of the coefficient
appears only in increasing the exponent of $\q$ by $1$ (see
\cref{fig:proofpositivesemicycle} $(a)$ and $(b)$). This agrees with the action \eqref{pq}. 
  
  %This situation is also compatible with changes inside the tensor product, i.e
%$$\{x_{\ends(C_1)}\}_{k_1}  \otimes \dots \otimes \{x_{\ends(C_p )}\}_{k_i} \otimes \dots \otimes\{ x_{\ends(C_{p-1})}\}_{k_p-1}.$$ 

\begin{figure}
\begin{center}
 \MatchingProofp{6}{4/6,-6/-4,1/3, -3/-1 }{5/6,-6/-5,2/3,-3/-2} \MatchingProoffpp{7}{-7/-2,2/7, 4/6,-6/-4,1/3, -3/-1  }{-7/-4, 4/7, 5/6,-6/-5,2/3,-3/-2} 
\end{center}
\begin{center}
\MatchingProofp{6}{4/6,-6/-4,1/3, -3/-1 }{5/6,-6/-5,2/3,-3/-2}  \MatchingProoffppb{7}{-7/2,-2/7, 4/6,-6/-4,1/3, -3/-1  }{-7/-4, 4/7, 5/6,-6/-5,2/3,-3/-2} 
\end{center}
\caption{The visualization of the action of $\p_{\q}(x_{k})$ on  \eqref{eq:tensorproducta}}
\label{fig:proofpositivesemicycle}
\end{figure}
 
 Case 3b). In the last situation we apply equation \eqref{pq} to
 \eqref{eq:tensorproductb} and we proceed as in
 \cref{fig:proofpositivesemicycleb} $(a)$, $(b)$.  
\begin{figure}
\begin{center}
\MatchingProofp{6}{-4/6,-6/4,1/3, -3/-1 }{5/6,-6/-5,2/3,-3/-2}  \MatchingProoffppc{7}{-7/2,-2/7, -4/6,-6/4,1/3, -3/-1 }{-7/-4, 4/7, 5/6,-6/-5,2/3,-3/-2} 
\end{center}
\begin{center}
\MatchingProofp{6}{-4/6,-6/4,1/3, -3/-1 }{5/6,-6/-5,2/3,-3/-2}  \MatchingProoffppd{7}{-7/-2,2/7, -4/6,-6/4,1/3, -3/-1 }{-7/-4, 4/7, 5/6,-6/-5,2/3,-3/-2} 
\end{center}
\caption{The visualization of the action of $\p_{\q}(x_{k})$ on  \eqref{eq:tensorproductb}}
\label{fig:proofpositivesemicycleb}
\end{figure}

We emphasize that in all of the above steps  $\maks(\Ci^\pm_p) $  is the most right or left point representing vector in the tensor product \eqref{eq:tensorproducta} and \eqref{eq:tensorproductb}.  By using this point we create  the partition $\hat \tyldapi$.  This explains why $\hat \tyldapi$  is non-crossing and don't cover singletons. 

Note that as $\pi$ runs over $\P^{sym}_{1,2;\epsilon}(k-1)$, every set
partition $\tyldapi \in \P^{sym}_{1,2;\epsilon}(k)$ appears exactly
once in one of the cases 2 or 3. Therefore in Case 2 and Case 3, the
pictorial inductive step and the actual action of $\B(x_{\overline
  k}\otimes x_{k})$ both create the same terms with the same
coefficients, and hence the formula \eqref{formula101} is true when
$n=k$ and $\epsilon(k)=1$. Therefore the formula \eqref{formula101}
holds for all $n \in \N$ by induction, which finishes the proof.
\end{proof}
 When equipped with the vacuum expectation state $\state(\cdot)= \langle\Omega,\cdot\, \Omega\rangle_{\q,\s} $ the moments of
the Gaussian operator can be expressed as cycles of type B.
\begin{theorem}\label{thm2} 
Suppose that $x_1,\dots,x_{2n} \in H_\R$, then  
\begin{align*}
\state( \G(x_{2n})\dots \G(x_1))=&\sum_{\pi\in \P_{2}^{sym}(2n)}\s^{\SLNB(\pi)}\q^{\Lcycle(\pi)} \prod_{\substack{\{i,j\} \in \Pair(\pi)} }\langle x_i, x_j\rangle
% \prod_{\substack{\{i,j\} \in\Neg( \pi)}} \langle x_i, \overline{x}_j\rangle
\\=&\sum_{\pi\in \P_{2}^{sym}(2n)}\s^{\SLNB(\pi)}\q^{n-\cyc(\pi)} \prod_{\substack{\{i,j\} \in \Pair(\pi)} }\langle x_i, x_j\rangle, 
%\prod_{\substack{\{i,j\} \in\Neg( \pi)}} \langle x_i, \overline{x}_j\rangle
\end{align*}
where $\cyc(\pi)$ is the number of cycles of $\pi$.
%Note that the sum over the empty set is understood to be 0. 
\end{theorem}
\begin{proof}
The first equation is a direct consequence of \eqref{formula101} by
taking the sum over all $\epsilon\in\{1,\ast\}^{2n}$ and $
\Scycle(\pi) =\emptyset$. The second one  follows from the observation
that  $\Lcycle(\pi)=|\pi|-\cyc(\pi)$. Indeed, the length of each cycle
can be represented as the number of pairs which create it, and so $$\Lcycle(\pi)=\sum_{\sigma\in \Cycle(\pi)}(|\sigma|-1)=\sum_{\sigma\in \Cycle(\pi)}|\sigma|-\cyc(\pi)=n-\cyc(\pi).$$
\end{proof}

\subsection{Orthogonal polynomials and bivariate generating functions}
\label{subsec:orthogonal}

\subsubsection{Orthogonal polynomials}
For a probability measure $\mu$ with finite moments of all orders, let us orthogonalize the sequence $(1,x,x^2,x^3,\dots)$ in the Hilbert space $L^2(\R,\mu)$, following the Gram-Schmidt method. This procedure yields orthogonal polynomials $(P_n(t))_{n=0}^\infty$  
%$(P_0(x), P_1(x), P_2(x), \dots)$ 
with $\text{deg}\, P_n(x) =n$. Multiplying by constants, we take $P_n(x)$ to be monic, i.e., the coefficient of $x^n$ is 1. It is known that they satisfy a recurrence relation
\begin{align} \label{wielortogonalnerekursia}
x P_n(x) = P_{n+1}(x) +\beta_n P_n(x) + \gamma_{n-1} P_{n-1}(x),\qquad n =0,1,2,\dots
\end{align}
with the convention that $P_{-1}(x)=0$. The coefficients $\beta_n$ and $\gamma_n$ are called \emph{Jacobi parameters} and they satisfy $\beta_n \in \R$ and $\gamma_n \geq 0$. 
It is known that 
\begin{equation}\label{eq54}
\gamma_0 \cdots \gamma_n=\int_{\R}|P_{n+1}(x)|^2\mu(d x),\qquad n \geq 0.
\end{equation} 
Moreover, the measure $\mu$ has a finite support of cardinality $N$ if and only if $\gamma_{n}=0$ for $n \geq N-1$ and $\gamma_n > 0$ for $n = 0,\dots, N-2$
-- see \cite{Chihara}. 

%The continued fraction representation of the Cauchy transform can be expressed in terms of the Jacobi parameters:
%\[
%\int_{\R}\frac{\mu(d t)}{z-t} = \dfrac{1}{z-\beta_0 -\dfrac{\gamma_0}{z-\beta_1-\dfrac{\gamma_1}{z- \beta_2 - \cdots}}}.
%\]

%\subsubsubsection{ Askey-Wimp-Kerov distribution}
Let us recall here  that for any $c \in (-1, \infty)$ the Askey-Wimp-Kerov distribution $\nu_c$ (see e.g.~\cite{AskeyWimp1984} or \cite[Section 8.4]{Kerov1998}) is the measure on $\R$, with Lebesgue density
\begin{align*}
&\frac{1}{\sqrt{2\pi }\Gamma(c+1)}|D_{-c}(i x)|^{-2} \qquad x\in \R, 
\intertext{where $D_{-c}(z)$ is the solution to the differential Weber equation:}
 & \frac{d^2y}{dz^2}+\left(\frac{1}{2}-c-\frac{z^2}{4}\right)y=0,
\intertext{satisfying the initial conditions:}
& D_{-c}(0)=\frac{\Gamma\left(\frac{1}{2}\right)2^{-c/2}}{\Gamma\left(\frac{1+c}{2}\right)} \text{ and } D'_{-c}(0)=\frac{\Gamma\left(-\frac{1}{2}\right)2^{-(c+1)/2}}{\Gamma\left(\frac{c}{2}\right)}. 
\intertext{When $c > 0$, the solution $D_{-c}$ has the integral representation}
&D_{-c}(z)=\frac{e^{-z^2/4}}{\Gamma(c)}\int_0^\infty e^{-xz}x^{c-1}e^{-x^2/2}dx.
\end{align*}
It was proved in \cite{AskeyWimp1984} that for any $c \in (-1, \infty)$ the measure $\nu_c$ is a probability measure.
%The case $c = -1$ corresponds to the standard Gaussian distribution $N(0, 1)$, and 
The family
$(\nu_c)_{c\in(-1,\infty)}$ can be extended continuously at $c=-1$ by defining $\nu_{-1}$ to be the Dirac point mass
$\delta_0$ at $0$.
The orthogonal polynomials $(H_n(t))_{n=0}^\infty$, with respect to
$\nu_c$ are called \emph{the associated Hermite polynomials} and given by the recurrence
relation:
$$t H_n(t) = H_{n+1}(t) +(n+c){H}_{n-1}(t), \qquad n=0,1,2,\dots $$
with ${H}_{-1}(t)=0,$ ${H}_0(t)=1.$
 \begin{remark}
 For any $c\in [-1,0]$ the Askey–Wimp–Kerov distribution $\nu_c$ is freely infinitely divisible and freely selfdecomposable -- see \cite[Theorem 3.1]{BelinschiBozejkoLehnerSpeicher2011} and \cite[Theorem 3.3]{HasebeSakumaSteen2019}.
%With the above notations, we have    
% Thus 
% by \cite[Theorem 3.1]{SerbanBozejkoLehnerSpeicher2011} and \cite[
%Theorem 3.3]{HasebeSakumaSteen2019} the probability measure $\qMP_{\q,\s}$ 
% because with the notations from above, $c=\frac{1}{M}-1\in [-1,0)$. 
\end{remark}

Let $({Q}_n(t))_{n=0}^\infty$ be the family of orthogonal polynomials with the recursion relation
${Q}_{-1}(t)=0,$ ${Q}_0(t)=1$ and 
\begin{equation}\label{recursion2}
t Q_n(t) = Q_{n+1}(t) +\lambda(n){Q}_{n-1}(t), \qquad n=0,1,2,\dots
\end{equation}
where 
$$\lambda(n)=1 + 2(n-1) \q +\s%=2\q\left[n+\left(\frac{1+\s}{2\q}-1\right)\right]
=2\q(n+c),  \quad c=\frac{1+\s}{2\q}-1.$$
 Let
$\qMP_{\q,\s}$  be a probability measure associated with the orthogonal polynomials $Q_n(t)$. 
If $\q>0$, then  $c= M-1\geq -1$, since $\q=\frac{1}{M+N}$ and $\s=\frac{M-N}{M+N}.$
The above information leads to the conclusion that 
  the measure $\qMP_{\q,\s}$  is the dilatation by $\sqrt{2\q}$ (for $\q>0$)
of the  Askey-Wimp-Kerov distribution  $\nu_{\frac{1+\s}{2\q}-1}$, with Lebesgue density  
$$\frac{1}{2\sqrt{\q\pi }\Gamma\big(\frac{1+\s}{2\q}\big)|D_{-\frac{1+\s}{2\q}+1}\big(\frac{i x}{\sqrt{2\q}}\big)|^2}, \qquad x\in \R.$$
%Askey  and Wimp  have found explicit orthogonality relation for \eqref{recursion2}. 
%For and $\q > 0$ these polynomials are called \emph{the associated Hermite polynomials}. The orthogonalizing probability measure $\qMP_{\q,\s}$ is known in \cite{AskeyWimp1984} or \cite[Section 8.4]{Kerov1998}  
%supported on $\R$ and absolutely continuous with respect to the Lebesgue measure with density 
%$$\frac{1}{2\sqrt{\q\pi }\Gamma\big(\frac{1+\s}{2\q}\big)|D_{-\frac{1+\s}{2\q}+1}\big(\frac{i x}{\sqrt{2\q}}\big)|^2}dx,$$
%where $y=D_{\nu}(x)$ is parabolic cylinder functions can be defined as a solution of Weber equations $\frac{d^2y}{dz^2}+\left(\nu +\frac{1}{2}-\frac{z^2}{4}\right)y=0$ -- see \cite[Page 324]{Magnus}. In the case of negative $$v=-\frac{1+\s}{2\q}+1=-M\epsilon+1<0$$ there exist integral  representation  $D_{v}(z)=\frac{e^{-z^2/4}}{\Gamma(-\nu)}\int_0^\infty e^{-xz}x^{-\nu-1}e^{-x^2/2}dx$ so up to the constant factor function $e^{z^2/4}D_{\nu}(z)$ is the Laplace transform of the $\chi^2$-distribution with $-\nu$ degree of freedom. 

In the case $\q=0$, we obtain a recursion %$t Q_n(t) = Q_{n+1}(t) +(1+\s){Q}_{n-1}(t)$ 
of the orthogonal polynomials for the semi-circle
distribution dilated by  $\sqrt{1+\s}$ i.e. $\mu_{0,\s}$ i.e. 
$$\frac{1}{2\pi (1+\s)}\sqrt{4(1+\s)-x^2} , $$ 
for $|x| \leq 2\sqrt{1+\s}.$
%If $\q=0$, then $\mu_{0,\s}$ is semicircular distribution  
%$$\frac{1}{2\pi (1+\s)}\sqrt{4(1+\s)-x^2}dx, \qquad - 2\sqrt{1+\s} \leq x \leq 2\sqrt{1+\s}. $$ 
%Moreover, the measure $\qMP_{\q,\s}$ for $\q<0 $ has a finite support of cardinality $ \lfloor {\frac{-1-\s}{2\q}+1} \rfloor$. 
If $\q<0$, then by  \cref{subsec:Exclusion} we have  $\lambda(n)=1 + 2(n-1) \q +\s > 0$  for $n\leq M$ and $\lambda(n)= 0$ for $n >  M$, thus  the measure is discrete. % -- see \cite{HoraObata2007}. 
%\todo[inline]{MD: a skąd wiadomo, że jeśli tylko skończenie wiele
 % parametrów $\gamma$ jest dodatnich to miara jest dyskretna? Jakieś cytowanie?}

\begin{theorem} \label{rozkladgaussa}  Let    $x\otimes x \in \HH_\R$ and  $ \|x\otimes x\|=1$. Let $\m$ be the probability distribution of $\G(x)$, with respect to the vacuum
state. Then $\m$ is equal to $\qMP_{\q,\s}.$
%$$\|(\xx)^{\otimes n}\|_{\q}^2=[n]_{q_1}! [n]_{q_2}!$ 
\end{theorem}
\begin{proof}Let $\gamma_{n-1}= 1+2\q(n-1)+\s$, $n=1,2,\dots,$  then 
\begin{align*}
\|x^{\otimes n}\otimes x^{\otimes n}\|_{\q,\s}^2
&= \langle x^{\otimes n}\otimes x^{\otimes n},P_{\q,\s}^{(n)}x^{\otimes n}\otimes x^{\otimes n}\rangle_{0,0}  \\&
 =\langle x^{\otimes n}\otimes x^{\otimes n},( I\otimes P^{(n-1)}_{\q,\s}\otimes I)R^{(n)}_{\q,\s}x^{\otimes n}\otimes x^{\otimes n}\rangle_{0,0}.
 \intertext{Using the identity $R^{(n)}_{\q,\s}x^{\otimes n}\otimes
  x^{\otimes n}=(1+2\q(n-1)+\s)x^{\otimes n}\otimes x^{\otimes n}$ we
  have that }
 &\|x^{\otimes n}\otimes x^{\otimes n}\|_{\q,\s}^2=%\left([n]_{q_j}\langle \xi_j^{\otimes n},( I\otimes P^{(n-1)}_{q_j})\xi_j^{\otimes n}\rangle_{0}\right)=  
 \gamma_{n-1} \|x^{\otimes (n-1)}\otimes x^{\otimes (n-1)}\|_{\q,\s}^2 =\gamma_0\gamma_1\cdots \gamma_{n-1},
\end{align*}
%\todo[inline]{MD: Rozumiem, że ten dowód działa tylko dla $q_+ >0$? W
 % przeciwny razie ten produkt $\gamma$ po
  %prawej stronie może być ujemny, więc nie może być równy lewej stronie}
and hence by\eqref{eq54} we obtain that 
\begin{equation}
\|x^{\otimes n}\otimes x^{\otimes n}\|_{\q,\s}= \|Q_n\|_{L^2},\qquad n\in\N\cup\{0\}. 
\end{equation}
Therefore, the map $\Phi\colon (\text{span}\{x^{\otimes n}\otimes x^{\otimes n}\mid n \geq 0\}, \|\cdot\|_{\q,\s}) \to L^2(\R,\qMP_{\q,\s})$ defined by $\Phi(x^{\otimes n}\otimes x^{\otimes n})= Q_n(t)$ is an isometry. 
Note that  
\begin{align*}
\G(x)\text{ }x^{\otimes n}\otimes x^{\otimes n} 
&= \B^\ast (x\otimes x) x^{\otimes n}\otimes x^{\otimes n}  +\B (x\otimes x) x^{\otimes n}\otimes x^{\otimes n} 
%\\&= x^{\otimes (n+1)}\otimes x^{\otimes (n+1)}+(a_{\q,\s}(\xi)\xi ^{\otimes n})\otimes (a_{v,w}(\eta)\eta ^{\otimes n})
\\
&= x^{\otimes (n+1)}\otimes x^{\otimes (n+1)}+\gamma_{n-1}x^{\otimes (n-1)}\otimes x^{\otimes (n-1)}.
\end{align*}
%where Proposition \ref{prop2} was used on the second line. 
Hence, by induction we can compute $\G^n(x)\Omega\otimes   \Omega$ and show that $\Phi(\G^n(x)\Omega \otimes \Omega) = x^n$. Since $\Phi$ is an isometry we get 
$\langle \Omega\otimes  \Omega, \G^n(x) \Omega\otimes  \Omega\rangle_{\q,\s} = m_n(\qMP_{\q,\s})$ for  $n\in \N$. 
%For odd integers $n$ we can show that $\langle \Omega\O \overline \Omega, \G_{\xx}^n|\Omega\O \overline \Omega\rangle_{\q} =0= m_n(\qMP_{\qq})$.  
Since the moment problem is determined (see \cref{rem:Ortho} 1)), the probability measure
$\qMP_{\q,\s}$  giving the moment sequence $m_n(\qMP_{\q,\s})$ is
uniquely determined and hence $\qMP_{\q,\s}=\m$.  
\end{proof}
\begin{remark}
  \label{rem:Ortho}
In the above proof, the condition $\s<0$ implies that $n\leq M$; see  \cref{subsec:Exclusion}.  In this situation we  
consider the finitely many vectors i.e. $x^{\otimes n}\otimes x^{\otimes n} $ for $n\leq M$. 
%product $\gamma_0\gamma_1\cdots \gamma_{n-1}$ is nonnegative. 
\end{remark}

\subsubsection{Moments and combinatorial statistics}

The moment problem for the probability measure $\qMP_{\q,\s}$ is determined. 
Indeed, when $\q\geq 0$, then the best known criterion (in this case) of Hamburger moment problem is due to Carleman \cite{Carleman,Chihara1989}.  Carleman's theorem states that the moment problem is determined if 
$$\infty=\sum_{n\geq 1} \frac{1}{\lambda(n)}=\sum_{n\geq 1}\frac{1}{1 + 2(n-1) \q +\s}, \qquad \q\geq 0.
$$

The odd moments $m_n(\qMP_{\q,\s})$ are equal to zero and even
moments can be computed for instance using Viennot's combinatorial
theory of the orthogonal polynomials \cite{Viennot1985} (see also
\cite{AccardiBozejko1998} and \cite[Section 1.6]{HoraObata2007}).

A Dyck path $D$ of length $n$ is a finite sequence of steps
$w_1,\ldots,w_{2n}$, where $w_i\in \{-1,1\}$, for $1\leq i \leq 2n$
and
\begin{itemize}
\item $\sum_{i=1}^\ell w_i \geq 0$ for all $1\leq \ell < 2n$,
\item $\sum_{i=1}^{2n} w_i = 0$.
\end{itemize}
We denote the set of Dyck paths of length $n$ by $\Dyck_n$ and its
cardinality is given by the $n$-th Catalan number $C_n :=
\frac{1}{n+1}\binom{2n}{n}$.
The combinatorial theory of the orthogonal polynomials asserts that the $2n$-th moment $m_n(\qMP_{\q,\s})$ is given by the following
weighted generating function of Dyck paths:
\[ m_{2n}(\qMP_{\q,\s}) = \sum_{(w_1,\dots,w_{2n}) \in \Dyck_n}
  \prod_{\ell \colon
    w_\ell =
    -1}\big(1+\s+\q\cdot2\sum_{i=1}^{\ell}w_i\big). \]
This in conjunction with \cref{thm2} gives a curious identity between
two bivariate generating functions of different combinatorial objects:
\begin{equation}
  \label{eq:Moments}
\sum_{\pi\in \P_{2}^{sym}(2n)}\s^{\SLNB(\pi)}\q^{\Lcycle(\pi)} = \sum_{(w_1,\dots,w_{2n}) \in \Dyck_n}
  \prod_{\ell \colon
    w_\ell =
    -1}\big(1+\s+\q\cdot2\sum_{i=1}^{\ell}w_i\big).
\end{equation}
Indeed, note that both sides of the above identity are polynomials in
$\q,\s$ which agree on a certain infinite set
and by the same argument as used in the proof of
\cref{theo:PositiveDefinite} the equality between polynomials also
holds true. Note that only the LHS is a bivariate generating function
sensu stricto since the weights associated to Dyck paths appearing in
the RHS are not monomials in $\q,\s$. Nevertheless it is convenient to
rewrite the LHS in terms of another ``bivariate generating function'',
which is not a monomial in $\q,\s$ but it is a monomial in $x,y$ after the
following change of variables $x = 1+\s, y=2\q$.

\begin{proposition}
  Let $x = 1+\s, y=2\q$. Then
  \begin{equation}
    \label{eq:Moments2}
    m_{2n}(\qMP_{\q,\s}) = \sum_{\pi\in \P_{2}^{sym}(2n)}\s^{\SLNB(\pi)}\q^{\Lcycle(\pi)} =\sum_{\pi\in \P_{2}(2n)}x^{c(\pi)}y^{n-c(\pi)},
  \end{equation}
  where $\P_{2}(2n)$ is the set of matchings (pair-partitions) on the
  set $[2n]$ and $c(\pi)$ denotes the number of cycles created by a
  concatenation of $\pi$ with the unique non-crossing matching
  $\hat{\pi} \in P_{2}(2n)$ whose left legs coincides with the leg
  legs of $\pi$.
\end{proposition}

\begin{proof}
Define the natural projection $P: \P_{2}^{sym}(2n) \to \P_{2}(2n)$
which is sending a pair $(i,j)$ of $\pi \in \P_{2}^{sym}(2n)$ into a
pair $(|i|,|j|)$ of $P(\pi)$. We claim that for every pair $\pi \in
\P_{2}(2n)$ one has
\[ \sum_{\sigma \in P^{-1}(\pi)}\s^{\SLNB(\sigma)}\q^{\Lcycle(\sigma)}
  = (1+\s)^{c(\pi)}(2\q)^{n-c(\pi)}.\]
Then the proof follows immediately from the claim.

In order to prove
the claim it is enough to notice the following properties of $P$. First of all
for any $\pi \in \P_{2}(2n)$ there are precisely $2^n$ preimages in
$\P_{2}^{sym}(2n)$: for every pair $\{i,j\} \pi$ the corresponding
pairs in the partition from $P^{-1}(\pi)$ are either $\{i,j\}$ and
$\{\overline{i},\overline{j}\}$ or $\{i, \overline{j}\}$ and
$\{\overline{i},j\}$. Exchanging these two pairs by each other in
elements from $P^{-1}(\pi)$ is an involution $f_{\{i,j\}}$ on
$P^{-1}(\pi)$ which has a property that it fixes the number of cycles
and it changes the sign of the
cycle, which contains the pairs associated with $\{i,j\}$. The group
generated by $\{f_{B} \colon B \in \pi\}$ acts transitively on
$P^{-1}(\pi)$. In particular the number of cycles is an invariant of
$P^{-1}(\pi)$ so it is clearly equal to $c(\pi)$, which proves the
formula
\[  \sum_{\sigma \in
    P^{-1}(\pi)}\s^{\SLNB(\sigma)}\q^{\Lcycle(\sigma)} = \sum_{\sigma \in P^{-1}(\pi)}\s^{\SLNB(\sigma)}\q^{n-c(\sigma)}
  = (1+\s)^{c(\pi)}(2\q)^{n-c(\pi)}.\]
  \end{proof}

Notice that substituting $x=c$, $y=1$ in \eqref{eq:Moments2} we obtain moments of the
Askey--Wimp--Kerov distribution $\nu_{c-1}$ with the shifted parameter
$c-1$. These moments were interpreted by Drake \cite{Drake2009} as a
generating function of pair-partitions with respect to two statistics.

\begin{theorem}[\cite{Drake2009}]
The (even) moments $m_{2n}(\nu_{c-1})$ are the generating series
of pair-partitions with respect to the following statistics:
\begin{align*}
  m_{2n}(\nu_{c-1}) &= \sum_{\pi\in
  \P_{2}(2n)}c^{\#\text{non-nested pairs in } \pi},\\
    m_{2n}(\nu_{c-1}) &= \sum_{\pi\in \P_{2}(2n)}c^{\#\text{pairs in }
                        \pi \text{ with no right crossing }}.
  \end{align*}
\end{theorem}
We recall that a pair $\{i,j\} \in \pi$ is \emph{nested} if there
exists a pair $\{i',j'\} \in \pi$ such that $i'<i<j<j'$, and a pair
$\{i,j\} \in \pi$ has a \emph{right crossing} if there
exists a pair $\{i',j'\} \in \pi$ such that $i<i'<j<j'$. As an
immediate corollary from \eqref{eq:Moments2} we have an interpretation
of the moments $m_{2n}(\nu_{c-1})$ as the generating series
of pair-partitions with respect to a third, different statistic:

\begin{equation*}
m_{2n}(\nu_{c-1}) = \sum_{\pi\in
  \P_{2}(2n)}c^{\#\text{cycles in } \pi}.  
  \end{equation*}

\begin{remark}  
We would like to mention potential interpretations of the moments as
bivariate generating functions of maps via unknown statistics. A (orientable) map is an embedding of a graph into a
(orientable) compact, connected surface without a boundary such that the complement of the image of the graph
is a disjoint union of simply connected pieces. A map is rooted if it
is distinguished with an oriented corner (i.e.~a small neighbourhood
around a vertex, called the \emph{root}, delimited by two consecutive
half-edges). Drake, based on the result of  showed that the moment $m_{2n}(\nu_{c})$
is the generating function of rooted orientable maps with $n$ edges, where
the exponent of $c$ gives the number of vertices different from the
root. This corresponds to the substitution $y=1,c=\q$ in
\eqref{eq:Moments2} and suggests that the bivariate generating
function
\[ \sum_{\pi\in \P_{2}(2n)}(1+x)^{c(\pi)}y^{n-c(\pi)}\]
can be interpreted as the generating function of rooted orientable maps with $n$ edges, where
the exponent of $x$ gives the number of vertices different from the
root and the exponent of $y$ is a statistic on maps
yet to be found. Another interesting specialization is given by $\q=0$, which gives the
following univariate generating function of the symmetric-pair partitions
with all cycles of size $1$:
$$\sum_{\substack{\pi\in \P_{2}^{sym}(2n):\\ \sigma \in \pi \iff
    |\sigma|=1} }\s^{\SLNB(\pi)} = \sum_{\pi\in
  \P^{\text{non-crossing}}_{2}(2n)}(1+\s)^n = C_n(1+\s)^n.$$
This is the weighted generating function for objects enumerated by
$2^nC_n$, which is given by \href{https://oeis.org/A151374}{A151374}
(see also \href{https://oeis.org/A052701}{A052701}). Among others, it
counts pointed, rooted, bipartite planar
maps with $n$ edges. This suggest that there might be a natural
injection of pointed, rooted, bipartite planar
maps with $n$ edges into the set of rooted, orientable
maps with $n$ edges. Finally, the set $\P_{2}^{sym}(2n)$ of the
symmetric pair-partitions is in a natural bijection with the set of
rooted (orientable and non-orientable) maps with only one-face (i.e.~the complement of the image of the graph embedded into a
surface is simply connected) and it is natural to ask for two statistics
which expresses these moments as a bivariate generating function of
rooted, one-face maps. Note that the set of rooted, one-face maps plays
a crucial role in studying the structure of maps via bijective methods
(see
\cite{Schaeffer1997,BouttierDiFrancescoGuitter2004,ChapuyMarcusSchaeffer2009,ChapuyFerayFusy2013,ChapuyDolega2017,Lepoutre2019,DolegaLepoutre2020}
among others)
and we believe that the statistics in question might be found
through the aforementioned bijections. We leave these questions open,
as they are out of scope of this paper.

% Then, using the same reasoning as
% previously we
% obtain the possible cycles 
% decomposition of size two $$\# \{\Cycle(\pi)\mid \pi \in \P_{2}^{sym}(n) \text{ and  } \sigma \in \pi \iff |\sigma|=2\}= 2^n C_n , $$
% where  $C_n=\frac{1}{n+1} {{2n}\choose{n} } $ is the Catalan number.  
%  In the context of enumeration of such objects we can say that these numbers  occur in literature \cite[eq. (5.4), Fig. 3, Fig. 7]{GuitterKristjansenNielsen1999}, with a very similar  combinatorics interpretation, specifically it is counting the number of 
%   a  random Eulerian triangulation  or  Meander configuration with $4$ connected components. We skip formal description of it and just heuristically explain that we can combine negative cycle  with object appearing in  \cite[Fig. 7]{GuitterKristjansenNielsen1999} as in the figure: 
% $$\MatchingProofequationDwupartycje{2}{-2/1,-1/2 }{-2/-1,1/2} \hspace{1em} \longleftrightarrow \hspace{1em} \MatchingProofequationDwupartycje{2}{-2/2,-1/1 }{-2/-1,1/2}  
% %\hspace{1em} \longleftrightarrow \hspace{1em} \MatchingProofequationDwupartycje{2}{-2/-1,1/2 }{-1/1,-2/2} 
% .$$ 

% %\begin{figure}[h] 
% %It is worth to mention that the number of such pair partitions is the Euler number; see Corollary \ref{cor:euler}.
\end{remark}

\appendix
\section{How to derive the picture of type D from type B}

In this section we explain that the analogous problem in the case
of reflection groups of type D turned out to be similar to
the case of type A and it follows from our main result (in particular
the possible applications coincides with the special choice of
parameters $M,N$ in \cref{theo:PositiveDefinite'}). This will exhaust the
classification problem for the inductive limits of all three infinite
series of Coxeter groups of Weyl type (A,B/C and D). 

We recall that
the Coxeter group $D_n$ of rank $n$ and type $D$ is a normal subgroup of
$B_n$ of index two:
\[ D_n := \{(g_1,\dots,g_n;\sigma) \in B_n \colon g_1\cdots g_n =
  \id\}.\]
In other terms, it is isomorphic to the kernel of the one-dimensional
representation
\[ B_n \ni g \mapsto (-1)^{\text{number of negative cycles in } g}.\]

Therefore the conjugacy classes of $D_n$ are parametrized by pairs of
partitions $(\rho^+,\rho^-)$ of total size $n$, where the second partition has an even
number of parts (because the number of negative cycles is even). They coincide with the conjugacy classes
$\mathcal{C}_{\rho^+,\rho^-}$, except of the class given by
$\mathcal{C}_{\rho,\emptyset}$ for $\rho$ of the form
\[ \rho = (\rho_1,\dots,\rho_\ell) = 2\cdot \mu =
  (2\mu_1,\dots,2\mu_\ell); \qquad \mu \vdash n/2.\]
This conjugacy class of $B_n$ splits into two conjugacy classes of
$D_n$ that we denote by $\mathcal{C}^+_{\rho,\emptyset}$ and
$\mathcal{C}^-_{\rho,\emptyset}$ (note that these classes exist only
if $2|n$).
The set of reflections in $D_n$ coincides with the set $\mathcal{R}_+$
of long reflections in $B_n$. In particular all the reflections are
conjugated to each other and the reflection function is given by
restricting the signed reflection function $\psi_{\q,\s}$ to $D_n$ and
substituting $\q=\s=q$.

Representation theory of $D_n$ can be derived from the representation
theory of the hyperoctahedral group $B_n$ through the Clifford
theory (see for instance \cite{Kerber1975}). The irreducible representations of $D_n$ are indexed by non-ordered pairs of
partitions $(\lambda^+,\lambda^-)$ of total sum $n$ such that
$\lambda^+\neq \lambda^-$ (in other terms the irreducible
representations $\rho_{\lambda^+,\lambda^-}$ and
$\rho_{\lambda^-,\lambda^+}$ are the same) and there are two
additional irreducible representations denoted $\rho_{\lambda,\lambda;+}$ and
$\rho_{\lambda,\lambda;-}$. The relation between these representations
and irreducible representations of $B_n$ can be described by the
restriction. Let $\rho^{B_n}_{\lambda^+,\lambda^-}$ be the irreducible
representation of $B_n$ indexed by a pair of partitions
$(\lambda^+,\lambda^-)$ and let $\rho^{D_n}_{\lambda^+,\lambda^-}$
denote the restriction of $\rho^{B_n}_{\lambda^+,\lambda^-}$ to
$D_n$. Then
\[ \rho^{D_n}_{\lambda^+,\lambda^-} = \begin{cases}
    \rho_{\lambda^+,\lambda^-} &\text{ if } \lambda^+\neq
    \lambda^-,\\ \rho_{\lambda,\lambda;+} + \rho_{\lambda,\lambda;-} &\text{ if } \lambda^+=
    \lambda^-=\lambda.\end{cases}\]
Therefore, if we restrict \eqref{eq:rozklad} to $D_n$ we obtain that
the reflection function of $D_n$ is given by
\begin{align*}
\sum_{\substack{(\lambda^+,\lambda^-)\\|\lambda^+|+|\lambda^-|=n}}\big(\delta_{\lambda^+\neq\lambda^-}\chi_{\lambda^+,\lambda^-}+\delta_{\lambda^+\neq\lambda^-}(\chi_{\lambda^+,\lambda^-;+}+\chi_{\lambda^+,\lambda^-;-})\big)\\
  \cdot\frac{1}{2}\Bigg(\prod_{\square
  \in \lambda^+}\left(q c(\square)+\frac{1+q}{2}\right) \prod_{\square
                                                                                       \in
                                                                                       \lambda^-}\left(q
  c(\square)+\frac{1-q}{2}\right)\\
  +\prod_{\square
  \in \lambda^-}\left(q c(\square)+\frac{1+q}{2}\right) \prod_{\square
  \in \lambda^+}\left(q
  c(\square)+\frac{1-q}{2}\right)\Bigg).
\end{align*}
In particular it is positive definite on $D(\infty)$ if and only if
\begin{align*}
  &\prod_{\square
  \in \lambda}\left(q c(\square)+\frac{1+q}{2}\right) \prod_{\square
                                                                                       \in
                                                                                       \mu}\left(q
  c(\square)+\frac{1-q}{2}\right)\\
  +&\prod_{\square
  \in \mu}\left(q c(\square)+\frac{1+q}{2}\right) \prod_{\square
  \in \lambda}\left(q
  c(\square)+\frac{1-q}{2}\right) \geq 0
  \end{align*}
for all partitions $\lambda,\mu$. It is easy to check that this
condition holds true if and only if
\[ q \in \{ \frac{1}{2N+1} \colon N \in \Z \}\cup\{0\}.\]
Note that due to \cref{theo:PositiveDefinite'} this set of parameters corresponds
precisely to the case when $\phi_{\q,\s}$ is positive
definite on $B(\infty)$ and $\q=\s$ (which can happen for
$\epsilon = \pm 1, N\in
\Z_{\geq 0}$ and $M=N+1$). Therefore we obtained the following theorem.

\begin{theorem}
  \label{theo:PositiveDefiniteD}
 Let $q \in \C$. The following conditions are equivalent:
  \begin{enumerate}[label=(\roman*), ref=\roman*]
  \item The reflection function $D(\infty) \ni \sigma \to q^{\ell_{\mathcal{R}}(\sigma)}$ is
    positive definite on $D(\infty)$;
 \item The reflection function $D(\infty) \ni \sigma \to q^{\ell_{\mathcal{R}}(\sigma)}$ is
    a character of $D(\infty)$;
    \item The reflection function $D(\infty) \ni \sigma \to q^{\ell_{\mathcal{R}}(\sigma)}$ is
      an extreme character of $D(\infty)$;
    \item The reflection function $D(\infty) \ni \sigma \to q^{\ell_{\mathcal{R}}(\sigma)}$ is
      the restriction of the positive definite reflection function
      $B(\infty) \ni \sigma \to q^{\ell_{\mathcal{R}}(\sigma)}$ on $B(\infty)$;
 \item $q \in \{ \frac{1}{2N+1} \colon N \in \Z \}\cup\{0\}$.
  \end{enumerate}
\end{theorem}

\subsection{Concluding remarks}

We conclude by making a remark that the problem of classifying
positive definite reflection functions (or their multivariate
refinements) seems to be approachable for a wide class of groups
$\Sym{\infty}(T)$ (studied by Hirai and Hirai \cite{HiraiHirai2005}) by
using a method presented in this paper. These groups are of the following form: for a finite
group $T$ we define $\Sym{n}(T)$ as the wreath product $T \wr \Sym{n}$
and we set $\Sym{\infty}(T)$ as the inductive limit of the ascending
tower of groups $\Sym{1}(T)<\Sym{2}(T)<\cdots$. This class
contains for instance the infinite analogue of the Sheppard--Todd groups
$G(r,1,\infty)$, which corresponds to the choice $T=\Z_r$. It turns out that the representation theory of the
group $\Sym{n}(T)$ can be explicitly described in terms of the representation theory
of $\Sym{n}$ and $T$. This makes it possible to use the same methods as
we used in this paper to tackle the case of the hyperoctahedral group (which is given by the choice $T=\Z_2$). We did not study this more general context
because of the applications that we presented in the second part of the
paper and that are specific to type $B$. Therefore we leave this comment as an invitation for further investigations.

\section*{Acknowledgement}

MD would like to thank to Joel Brewster Lewis and Jacinta Torres for their helpful comments.

\bibliographystyle{amsalpha}

\bibliography{biblio2015}

\end{document}